\documentclass[11pt,leqno]{amsart}
\usepackage[left=2.6cm, right=2.6cm, top=2.6cm, bottom= 2.6cm]{geometry}
\numberwithin{equation}{section}

\usepackage{setspace}

\newcommand{\bP}{\mathbb{P}}
\newcommand{\A}{\mathbb{A}}
\newcommand{\Y}{\mathbb{Y}}

\newcommand{\C}{\mathbb{C}}
\newcommand{\E}{\mathbb{E}}
\newcommand{\N}{\mathbb{N}}
\newcommand{\D}{\mathbb{D}}

\newcommand{\LL}{\mathbb{L}}
\newcommand{\R}{\mathbb{R}}
\newcommand{\HH}{\mathbb{H}}
\newcommand{\eps}{\epsilon}

\newcommand{\cM}{\mathcal{M}}

\newcommand{\cF}{\mathcal{F}}
\newcommand{\cS}{\mathcal{S}}

\newcommand{\cP}{\mathcal{P}}

\newcommand{\cB}{\mathcal{B}}
\newcommand{\cU}{\mathcal{U}}

\newcommand{\op}{\bar{p}}

\newtheorem{theorem}{Theorem}[section]
\newtheorem{lemma}[theorem]{Lemma}
\newtheorem{proposition}[theorem]{Proposition}
\newtheorem{corollary}[theorem]{Corollary}

\newtheorem{remark}[theorem]{Remark}
\newtheorem{definition}[theorem]{Definition}

\begin{document}

\setstretch{1.15}

\title[Compact support for heat equations with stable noise]{The compact support property for solutions to stochastic heat equations with stable noise}

\author{Thomas Hughes}
\address{Department of Mathematical Sciences, University of Bath, Bath, UK.}
\email{th2275@bath.ac.uk}

\maketitle

\begin{abstract} We consider non-negative weak solutions to the stochastic partial differential equation 
\begin{equation*}
\partial_t Y(t,x) = \Delta Y(t,x) + Y(t,x)^\gamma \dot{L}(t,x), 
\end{equation*}
for $(t,x) \in \R_+ \times \R^d$, where $\gamma > 0$ and $\dot{L}$ is a one-sided white stable noise of index $\alpha \in (1,2)$. We prove that solutions with compactly supported initial data have compact support for all times if $\gamma \in (2-\alpha, 1)$ for $d=1$, and if $\gamma \in [1/\alpha,1)$ in dimensions $d \in [2,2/(\alpha-1)) \cap \N$. This complements known results on solutions to the equation with Gaussian noise. 

We also establish a stochastic integral formula for the density of a solution and associated moment bounds which hold in all dimensions for which solutions are defined.
\end{abstract}

\maketitle



\section{Introduction and main results} \label{s_intro}
In this paper, we consider the behaviour of non-negative solutions to the parabolic stochastic partial differential equation (SPDE)
\begin{equation} \label{e_spde1}
\partial_t Y(t,x) = \Delta Y(t,x) + Y(t,x)^\gamma \dot{L}(t,x)
\end{equation}
over the space-time domain $(t,x) \in \R_+ \times \R^d$, with $\R_+ = (0,\infty)$, where $\gamma > 0$ and, for an index $\alpha \in(1,2)$, $\dot L$ is a spectrally positive $\alpha$-stable noise which is white in space and time. Our main result is a partial answer to the following question: for which values of $\gamma$ do solutions to \eqref{e_spde1} have compact support? Our work is motivated by the known results for the same equation with white Gaussian noise and to some extent by the theory of superprocesses. This is the first work to address the question of compact support for an SPDE with stable noise.

We consider weak solutions to \eqref{e_spde1}, which are defined in Section~\ref{ss_prel}. Our main result, Theorem~\ref{thm_compact}, states that for a range of values of $\gamma$ with $\gamma<1$, weak solutions to \eqref{e_spde1} have the compact support property, which means that the solution almost surely remains supported inside some random compact set until time $T$ for all $T>0$. Our other main result, Theorem~\ref{thm_stochinteg}, establishes certain fundamental properties of weak solutions. In particular, we show that there is a density process associated to a weak solution defined by a stochastic integral formula and prove several properties of this version of the density. The stochastic integral formula implies that the weak solutions we consider are also mild solutions.

The introduction is divided into three main sections. Section~\ref{ss_prel} contains some background on SPDEs with stable and L\'evy noise, the definition of weak solutions to \eqref{e_spde1}, and some discussion of related work. Section~\ref{ss_density} contains the statement of Theorem~\ref{thm_stochinteg}, our result on the stochastic integral formula for the density of a weak solution. The main result on the compact support property and our discussion thereof are in Section~\ref{ss_csp}.

\subsection{Preliminaries} \label{ss_prel} 
There is a growing literature on SPDEs driven by multiplicative white stable noise, or more generally by multiplicative heavy-tailed white L\'evy noise. Existence of solutions (and in some cases uniqueness) has been considered in works such as \cite{B2013, BCL2023, C2017b, C2017, M2021, M2002, YZ2017}. Path properties have been studied in \cite{CDH2019, YZ2017}, and intermittency in \cite{BCL2023, CK2019}. With the exception of \cite{M2002} and \cite{YZ2017}, however, these works consider equations with Lipschitz noise coefficients, which excludes \eqref{e_spde1} when $\gamma < 1$. More general work on equations with L\'evy noise, e.g. the Hilbert space-valued solution theory of \cite{PZbook}, also tend to exclude equations such as \eqref{e_spde1}, except when $\gamma = 1$. The literature on solutions to \eqref{e_spde1} with $\gamma < 1$, especially concerning their path properties, is therefore quite sparse. We discuss the existing work in greater detail shortly.

The equation \eqref{e_spde1} is formal, because, as with white Gaussian noise, the white $\alpha$-stable noise is too rough for solutions to be differentiable. Weak solutions are defined as satisfying a stochastic integration by parts formula when integrated against smooth test functions. Supposing $Y(t,x)$ were in fact a smooth solution to \eqref{e_spde1} with initial condition $Y(0,\cdot)$, and integrating it against a test function $\phi \in \cS$, we obtain from integration by parts that, for $t>0$, 
\begin{align*}
&\int \phi(x) Y(t,x) dx -\int \phi(x) Y(0,x) dx 
\\ &\hspace{2 cm}= \int_{(0,t] \times \R^d}   Y(s,x) \Delta \phi(x) ds dx + \int_{(0,t] \times \R^d} \phi(x) Y(s,x)^\gamma L(ds,dx).
\end{align*}
Solutions to \eqref{e_spde1} will be defined as satisfying, in an appropriate sense, the integration by parts formula above. The integrator $L(ds,dx)$ is a spectrally positive $\alpha$-stable martingale measure, defined e.g. by Mytnik \cite[Definition~1.2]{M2002}, of which the $\alpha$-stable noise $\dot{L}$ is the distributional derivative. We define the $\alpha$-stable martingale measure in Section~\ref{s_stablemeasures}, in particular see Definition~\ref{def_stablemeasure}, and the stochastic integral with respect to $L$ is discussed in Section~\ref{s_stochcalc}. For the time being, we simply note that the relationship between the spectrally positive $\alpha$-stable martingale measure and the $\alpha$-stable process with no negative jumps is analogous to the relationship between the Brownian sheet and Brownian motion (see e.g. \cite{Walsh}).

We now introduce some notation in order to define weak solutions. For $\alpha \in (1,2)$ and $\gamma > 0$, we define
\begin{equation*}
p := \alpha \gamma.
\end{equation*}
Let $\cM_f(\R^d)$ denote the space of finite, non-negative measures on $\R^d$, equipped with the topology of weak convergence. We denote by $\mathbb{D}([0,\infty),\cM_f(\R^d))$ the space of c\`adl\`ag paths in $\cM_f(\R^d)$ equipped with the Skorokhod topology. For $\mu \in \cM_f(\R^d)$ and a bounded or non-negative function $\phi:\R^d \to \R$, we write $\langle \mu, \phi \rangle = \int \phi(x) \mu(dx)$. We denote by $\cS$ the Schwartz space of smooth functions on $\R^d$ with rapidly decaying derivatives of all orders.

Our set-up is a stochastic basis, or filtered probability space, which we denote by $(\Omega, \cF, (\cF_t)_{t \geq 0}, \bP)$, with filtration $(\cF_t)_{t \geq 0}$ satisfying the usual conditions of completeness and right continuity, and we write $\E$ to denote the expectation associated to $\bP$. For the definition of the class of predictable processes on $\Omega \times \R_+ \times \R^d$, see Section~\ref{s_stochcalc}.  

\begin{definition} \label{def_sol}
A pair $(Y,L)$ defined on some stochastic basis $(\Omega, \cF, (\cF_t)_{t \geq 0}, \bP)$ is a weak solution to \eqref{e_spde1} with initial state $Y_0 \in \cM_f(\R^d)$ if the following hold:
\begin{itemize}
\item $L$ is spectrally positive $\alpha$-stable $\cF_t$-martingale measure on $\R_+ \times \R^d$.
\item $\{Y(t,x) : t > 0, x \in \R^d\}$ is predictable, non-negative, and satisfies
\begin{equation} \label{assumption_integ}
\int_{(0,t] \times \R^d} Y(s,x)^p \,dsdx < \infty \,\, \text{ for all $t>0$ a.s.}
\end{equation}
\item The measure-valued process $(1_{\{t >0\}}Y(t,x)dx + 1_{\{t = 0\}} Y_0(dx))_{t \geq 0}$ has a c\`adl\`ag version, denoted $(Y_t)_{t \geq 0} \in \mathbb{D}([0,\infty),\cM_f(\R^d))$, such that for all $\phi \in \cS$, with probability one,
\begin{equation} \label{e_spde1weak}
\langle Y_t, \phi\rangle - \langle Y_0, \phi\rangle = \int_{(0,t]} \langle Y_s, \Delta \phi\rangle ds + \int_{(0,t] \times \R^d} Y(s,x)^\gamma \phi(x) L(ds,dx), \quad t\geq 0.
\end{equation}
\end{itemize} 
\end{definition}

A few remarks are in order.

\begin{remark} \label{remark_densmeas} The statement that $(Y_t)_{t \geq 0}$ is a version of $(1_{\{t >0\}}Y(t,x)dx + 1_{\{t = 0\}} Y_0(dx))_{t \geq 0}$ means that $(Y_t)_{t \geq 0}$ has initial state $Y_0$ and
\begin{equation} \label{e_denssol}
\bP(Y_t(dx) = Y(t,x)dx) = 1 \,\, \text{ for all } t > 0.
\end{equation}
Indeed, an equivalent formulation of Definition~\ref{def_sol} is to define a weak solution as consisting of both a density process $\{Y(t,x) : t>0,x\in \R^d\}$, with the same assumptions as above, and a measure-valued process $(Y_t)_{t \geq 0} \in \mathbb{D}([0,\infty),\cM_f(\R^d))$ started from $Y_0$, such that \eqref{e_denssol} holds, and \eqref{e_spde1weak} is a.s. satisfied for every $\phi \in \cS$. While this would make for a less economical definition, this perspective will be useful in the sequel when we will need to compare Definition~\ref{def_sol} to a slightly different definition of weak solutions.
\end{remark}

\begin{remark}
The integrability assumption \eqref{assumption_integ} is natural and corresponds to imposing that the stochastic integral in \eqref{e_spde1weak} is well-defined when one takes $\phi \equiv 1$ (see Section~\ref{s_stochcalc}). It is unknown if there exist solutions not satisfying this property.\end{remark}

Finally, we emphasize that we do not assume that $t \to Y(t,\cdot)$ is an $\mathbb{L}^p(\R^d)$-valued process. See Remark~\ref{rmk_stochinteg} for a further discussion of this. 

For $\alpha \in (1,2)$ and $d \in \N$ with $d < \frac{2}{\alpha-1}$, $\gamma \in (0,1)$, and $Y_0 \in \cM_f(\R^d)$, there exist weak solutions to \eqref{e_spde1} in the sense of Definition~\ref{def_sol}. Existence of weak solutions to \eqref{e_spde1} was proved by Mytnik \cite{M2002}, whose result holds for $\alpha$ and $d$ with the same conditions and $\gamma \in (0,((2/d)+1)/\alpha)$. However, there is a small difference between Definition~\ref{def_sol} and the solutions constructed in \cite{M2002}. In order to contrast them, it is useful to view weak solutions as in Remark~\ref{remark_densmeas}, as consisting of a predictable density $\{Y(t,x) : t>0, x\in \R^d\}$ and a measure-valued process $(Y_t)_{t \geq 0} \in \mathbb{D}([0,\infty),\cM_f(\R^d))$, related by \eqref{e_denssol}, and satisfying \eqref{e_spde1weak} a.s. for all $\phi \in \cS$. The solutions in \cite{M2002} likewise consist of a density and measure-valued process, which we denote the same way, but instead of \eqref{e_denssol}, they satisfy
\begin{equation} \label{e_densprob2intro}
\bP \big(1_{(0,t]}(s)Y_s(dx)ds = 1_{(0,t]}(s)Y(s,x)dxds\big) = 1 \,\, \text{ for all } t > 0.
\end{equation}
(Of course, this implies that $Y_s(dx)ds = Y(s,x)dxds$ a.s. as measures on $\R_+ \times \R^d$.) The condition \eqref{e_densprob2intro} is strictly weaker than \eqref{e_denssol}. Indeed, under this assumption it is not difficult to construct examples of weak solutions for which there exist (non-random) times $t_0$ such that $\bP(Y_{t_0}(dx) = Y(t_0,x)dx) < ~1$. At such a time $t_0$, it is not obvious how to interpret \eqref{e_spde1weak}, because the equation involves a measure $Y_{t_0}$ which is not a.s. equal to $Y(t_0,x)dx$. Similarly, the fact that such times can exist under \eqref{e_densprob2intro} means that $(Y_t)_{t \geq 0}$ is not a priori a version of $(1_{\{t >0\}}Y(t,x)dx + 1_{\{t = 0\}} Y_0(dx))_{t \geq 0}$ under this assumption. 

For the reasons above, we prefer to work with weak solutions satisfying \eqref{e_denssol} instead of \eqref{e_densprob2intro}, and have made our definition accordingly. However, the construction in \cite{M2002} gives \eqref{e_densprob2intro}, so replacing it with \eqref{e_denssol}, and thereby obtaining weak solutions in the sense of Definition~\ref{def_sol}, must be justified. We provide such a justification in Section~\ref{s_stochintegralrep}. There we argue that it can be handled simultaneously with the proof of Theorem~\ref{thm_stochinteg}, which we prove for $\gamma \in (0,1)$. More precisely, we show that, if we can prove Theorem~\ref{thm_stochinteg} under slightly weaker assumptions, which essentially means relaxing \eqref{e_denssol} to \eqref{e_densprob2intro}, then it implies that a weak solution as in Definition~\ref{def_sol} can be constructed from a weak solution of the kind constructed in \cite{M2002}, thus proving the existence of the former. This allows us to simultaneously prove Theorem~\ref{thm_stochinteg} and, for $\gamma \in (0,1)$, bridge the small gap between the construction in \cite{M2002} and Definition~\ref{def_sol}. This argument is described in detail at the beginning of Section~\ref{s_stochintegralrep}.

For the rest of the introduction, and the rest of the paper with the exception of Section~\ref{s_stochintegralrep}, a weak solution to \eqref{e_spde1}, or simply a weak solution, when it is clear from context, indicates a weak solution as defined in Definition~\ref{def_sol}. We remind the reader that the sense in which ``weak" is meant is that, rather than a solution constructed with respect to a given $\alpha$-stable noise (a so-called strong solution), our solution is a pair $(Y,L)$. We remark that our solutions are also ``weak" in the PDE sense, meaning they satisfy a (stochastic) integration by parts formula; however, Theorem~\ref{thm_stochinteg} establishes that a weak solution is also the solution to the stochastic integral equation which corresponds to \eqref{e_spde1}, which is the definition of a so-called mild solution; see Remark~\ref{remark_mild}.

With these subtleties about the definition of solutions out of the way, we now summarize the literature concerning existence and uniqueness of solutions to \eqref{e_spde1}. We only consider the equation with $\alpha \in (1,2)$ but note that the $\alpha \in (0,1)$ case has been considered by Mueller \cite{M1998}. For $\alpha \in (1,2)$, as we have noted, weak solutions to \eqref{e_spde1} were constructed in a pioneering work of Mytnik \cite{M2002} for $p< 1  + \frac 2 d$ in spatial dimensions $d < 2/(\alpha-1)$ by constructing solutions to an associated martingale problem. The case $\gamma = 1/\alpha$ corresponds to super-Brownian motion with $\alpha$-stable branching \cite[Theorem~1.6]{M2002}. Hence, if $\gamma = 1/\alpha$ (i.e. $p=1$), uniqueness in law of solutions follows from uniqueness in law of the superprocess, which is itself a consequence of duality. This special case has recently been extended in a preprint of Maitra \cite{M2021}, who proved uniqueness in law of solutions to \eqref{e_spde1} when $\gamma \in (1/\alpha,1)$ in the one-dimensional case. Also in dimension one, and under some additional assumptions on solutions, Yang and Zhou \cite{YZ2017} established pathwise uniqueness under the condition $\gamma \in (2(\alpha-1)/(2-\alpha)^2, 1/\alpha + (\alpha-1)/2)$. In fact, their result applies to a class of SPDEs with more general coefficients. 

\subsection{The stochastic integral formula} \label{ss_density}
We now state a theorem concerning some basic properties of weak solutions. In particular, we construct a density process $\{\bar{Y}(t,x) : t>0 , x\in \R^d\}$ defined by a stochastic integral formula (see \eqref{e_thm_density}), sometimes called the Green's function representation. $\bar{Y}(t,x)$ and $Y(t,x)$ are essentially the same, in that they are equal a.e. on $(0,\infty)\times \R^d$ a.s. We also prove moment bounds and an approximation result for $\bar{Y}(t,x)$. Besides being useful in the proof of the compact support property, the properties established in this theorem are fundamental and will hopefully be useful in future work on \eqref{e_spde1}.

Let $(P_t)_{t \geq 0}$ denote the heat semigroup associated to the $d$-dimensional Laplacian, and let $p_t(\cdot)$ denote associated Gaussian heat kernel, that is
\begin{equation*}
p_t(x) = (4 \pi t)^{-d/2} \exp(-|x|^2/4t).
\end{equation*}
We denote by $C^\infty_c(\R^d)$ the space of compactly supported smooth functions on $\R^d$, we write $\psi * \mu$ to denote the convolution of a function $\psi$ and measure $\mu$, and we use the shorthand $\mu(1) = \langle \mu, 1 \rangle$ for the total mass of a measure. For $q \geq 1$, $\LL^q(\bP)$ denotes the space of $q$-integrable random variables with respect to $\bP$.

\begin{theorem} \label{thm_stochinteg} Let $\alpha \in (1,2)$, $\gamma \in (0,1)$, and $d \in [1,\frac{2}{\alpha-1})\cap \N$, and let $(Y,L)$ be a weak solution to \eqref{e_spde1} on $\R_+ \times \R^d$ with initial condition $Y_0 \in \cM_f(\R^d)$.

{\bf (a) (Density formula)} For every $(t,x) \in \R_+ \times \R^d$, we may define
\begin{equation} \label{e_thm_density}
\bar{Y}(t,x) := P_tY_0(x) + \int_{(0,t] \times \R^d} p_{t-s}(x-y) Y(s,y)^\gamma \, L(ds,dy).
\end{equation}
For each $t>0$, $\bP(Y_t(dx) = \bar{Y}(t,x)dx) = 1$, and with probability one, $Y(t,x) = \bar{Y}(t,x)$ a.e. on $\R_+ \times \R^d$. Moreover, the process $\bar{Y} = \{\bar{Y}(t,x) : t >0 , x\in\R^d\}$ has a predictable version.

{\bf(b) (Moment bounds)} For all $(t,x) \in \R_+\times \R^d$ and $q \in (0,1]$,
\begin{equation} \label{e_thm_meanmeasure}
\E(\bar{Y}(t,x)^q) \leq P_tY_0(x)^q.
\end{equation}
Moreover, for each $q \in (1,\alpha)$, there exists a family of constants $C_T = C(T,\alpha,\gamma,d,q)$, increasing in $T>0$, such that the following holds: if $p > 1$, then for all $(t,x) \in (0,T] \times \R^d$,
\begin{equation} \label{e_thm_momentbd_pgo}
\E(\bar{Y}(t,x)^q) \leq C_T t^{-(\alpha - 1)\frac d 2 \frac q \alpha}[1+ t^{1-(p-1)\frac d 2} Y_0(1)^{p-1} P_t Y_0(x)]^{q/\alpha} + C_T P_tY_0(x)^q;
\end{equation}
and if $p \leq 1$, then for all $(t,x) \in (0,T] \times \R^d$,
\begin{equation} \label{e_thm_momentbd_plo}
\E(\bar{Y}(t,x)^q) \leq C_T t^{-(\alpha - 1)\frac d 2 \frac q \alpha}[1+ t P_t Y_0(x)]^{q/\alpha} + C_T P_tY_0(x)^q.
\end{equation}

{\bf(c) (Approximation)} Let $\psi \in C^\infty_c(\R^d)$ be non-negative and satisfy $\int \psi = 1$, and for $\eps>0$ define $\psi_\eps$ by $\psi_\eps(x) = \eps^{-d} \psi(x/\eps)$. Then for every $(t,x) \in \R_+ \times \R^d$,
\begin{equation*}
(\psi_\eps * Y_t)(x) \to \bar{Y}(t,x) \,\, \text{ in $\mathbb{L}^q(\bP)$ as $\eps \downarrow 0$}
\end{equation*}
for every $q \in [1,\alpha)$.
\end{theorem}

\begin{remark} \label{remark_mild} Since $\bar{Y}(t,x) = Y(t,x)$ a.e., their stochastic integrals with respect to $L$ are a.s. equal. In particular, this implies that \eqref{e_thm_density} holds when the integrand on the right hand side is replaced with $p_{t-s}(x-y)\bar{Y}(s,x)^\gamma$, and hence $\bar{Y}$ is the solution to the stochastic integral equation associated to \eqref{e_spde1}. In particular, $\bar{Y}$ is a {\it mild solution} to \eqref{e_spde1} with initial data $Y_0$, meaning for every $(t,x) \in \R_+ \times \R^d$, with probability one,
\begin{equation*}
\bar{Y}(t,x) = P_tY_0(x) + \int_{(0,t] \times \R^d}p_{t-s}(x-y) \bar{Y}(s,y)^\gamma L(ds,dy).
\end{equation*}
\end{remark}

\begin{remark} \label{rmk_stochinteg} In the case $d=1$, several other works have proved versions of the stochastic integral formula \eqref{e_thm_density} using proofs of various lengths and degrees of technicality, and with different assumptions on solutions. In \cite{YZ2017}, Yang and Zhou proved similar results to Theorem~\ref{thm_stochinteg}(a) and to the bounds \eqref{e_thm_momentbd_pgo} and \eqref{e_thm_momentbd_plo} for solutions to a class of equations containing those considered here, in one spatial dimension. This was done under an integrability condition stronger than \eqref{assumption_integ}.

In \cite{M2021}, again in $d=1$, \eqref{e_thm_density} is proved in several lines under the assumption that $s \to \|Y(s,\cdot)\|_p$ is a c\`adl\`ag map, where $\|\cdot\|_p$ is the usual $\mathbb{L}^p$-norm. However, this seems to require the assumption that $(Y_s)_{s \geq 0} \in \D([0,\infty),\mathbb{L}^p)$, which we do not assume.\end{remark} 

In view of the above remark, the considerable effort expended here in proving Theorem~\ref{thm_stochinteg} is justified for two reasons. The first and main one is that our proof holds in spatial dimensions greater than one. The second is that, a priori, the construction of solutions in \cite{M2002}, which is the only one we are aware of, does not imply the identification of $(Y_t)_{t \geq 0}$ with a $\mathbb{L}^p$-valued process, and our proof is the first to establish \eqref{e_thm_density} and related results without assuming either this or an integrability condition strictly stronger than \eqref{assumption_integ}.

Certain crucial arguments in the proof of Theorem~\ref{thm_stochinteg}, in particular those used to obtain moment estimates, are based on the arguments of Yang and Zhou in \cite{YZ2017}. Our weaker integrability assumption introduces some complications, but the more significant difference between our work and theirs is that our proof holds for spatial dimensions greater than one, and this requires several new arguments, primarily owing to the more singular behaviour of the heat kernel in higher dimensions. Nonetheless, our argument owes a significant debt to the work done in \cite{YZ2017}. We also take this opportunity to point out that \cite{YZ2017} also proves fixed-time H\"older regularity of solutions when $d=1$.

\subsection{The compact support property} \label{ss_csp}
Our main interest in this work is the support properties of weak solutions to \eqref{e_spde1}, in particular whether or not solutions are compactly supported. The problem of compactness versus non-compactness of the support of a non-negative solution to a stochastic heat equation with space-time white noise is well understood when the noise is Gaussian. Consider the stochastic equation
\begin{equation*} 
\partial_t X(t,x) = \Delta X(t,x) + X(t,x)^\gamma \dot \xi(t,x)
\end{equation*}
for $(t,x) \in \R_+ \times \R$, where $\dot \xi$ is a space-time white Gaussian noise. For a solution with compactly supported initial data, the following holds:
\begin{itemize}
\item If $\gamma \in (0,1)$, then for each $t>0$ there a.s. exists a compact set $K_t$ such that $X(s,x) = 0$ for all $(s,x) \in [0,t] \times K_t^c$.
\item If $\gamma \geq 1$, then $X(t,x) > 0$ for all $(t,x) \in \R_+ \times \R$ a.s.
\end{itemize}
The strict positivity result is due to Mueller~\cite{M1991}; in the case $\gamma = 1$, Moreno Flores \cite{MF2014} has given a shorter proof using the connection to the directed random polymer model. The compact support property for $\gamma \in (0,1)$ was proved in stages over a series of papers. For the special case $\gamma = 1/2$, $X(t,x)$ is the density of a binary-branching super-Brownian motion (see \cite{KS1988, R1989}), and the proof of compact support is due to Iscoe \cite{I1988}, who used superprocess duality. Shiga \cite{S1994} proved the result for $\gamma \in (0,1/2)$ with an argument essentially based on a comparison with super-Brownian motion. The first proof for all $\gamma \in (0,1)$ is due to Mueller and Perkins \cite{MP1992}, whose method was to construct solutions as the density of a super-Brownian motion with density-dependent branch rate and establish a historical modulus of continuity for this process. Krylov revisited the problem in \cite{Krylov1997} and gave a shorter proof purely using stochastic analysis for all $\gamma \in (0,1)$. The proof of our main result is based on Krylov's method. Recently, the compact support property has been proved for solutions to a family of parabolic SPDEs with coloured Gaussian noise, also using a proof based on Krylov's method, by Han, Kim and Yi \cite{HKY2023}.

As we have already noted, there is a superprocess connection in the stable noise regime as well. With $\alpha$-stable noise, when $\gamma = 1/\alpha$ (and hence $p=1$) we may interpret the solution as the density of a super-Brownian motion with an $\alpha$-stable branching mechanism \cite{M2002}. As in the Gaussian regime, the tools of branching processes, including duality, are at one's disposal in this case, and it is considerably simpler to prove the compact support property. Thus our compact support result, Theorem~\ref{thm_compact}, is known for $\gamma = 1/\alpha$. For example, see \cite[Theorem~9.2.2]{Dawson}. Otherwise it is new and is the first result concerning the supports of solutions to \eqref{e_spde1}.

Let $S(Y_s)$ denote the closed (topological) support of $Y_s$, and for $t>0$ define
\[ \mathbf{S}_t^* := \cup_{s \in [0,t]} S(Y_s). \]
Let $\mathbf{S}_t$ denote the closure of $\mathbf{S}_t^*$. We say that $Y$ has the \textit{compact support property} if 
\[ \bP(\mathbf{S}_t \text{ is compact for all } t >0) = 1.\]
Our theorem has separate statements for the cases $d=1$ and $d>1$ because the result holds for different parameter regimes in the two cases. Since $2-\alpha < 1/\alpha$ for $\alpha \in (1,2)$, the result is strictly stronger in dimension one.

\begin{theorem} \label{thm_compact} Fix $\alpha \in (1,2)$. 

{\bf (a) [Dimension one]} Let $d=1$ and suppose that $2-\alpha < \gamma <1$. Then for any weak solution $(Y,L)$ to \eqref{e_spde1} with compactly supported initial data $Y_0 \in \cM_f(\R)$, $Y$ has the compact support property.

{\bf (b) [Higher dimensions]} Let $d \in [2,\frac{2}{\alpha-1})\cap \N$ and suppose that $1/\alpha \leq \gamma < 1$. Then for any weak solution $(Y,L)$ to \eqref{e_spde1} with compactly supported initial data $Y_0 \in \cM_f(\R^d)$, $Y$ has the compact support property.
\end{theorem}

We remark that, in terms of the density, we can view the compact support property as follows: $\mathbf{S}_t \subset \R^d$ is an a.s. compact random set such that for all $(s,x) \in (0,t] \times \R^d$,
\[ \bar{Y}(s,x)1_{\{x \in \mathbf{S}_t^c\}} = 0 \, \text{ a.s.}\]
This is an immediate consequence of Theorem~\ref{thm_stochinteg}(c).

Our proof method requires that $Y_0$ has compact support. In analogy to what is known for super-Brownian motion, see e.g. \cite[Corollary~III.1.7]{P2002}, it is natural to expect that a solution with initial condition given by a finite measure with unbounded support will have compact support by time $t$ for any $t>0$, and we conjecture that this is the case. 

The restriction on the dimension in part (b) is to ensure that solutions exist, per the (sufficient) condition $d < \frac{2}{\alpha-1}$ from Mytnik's existence theorem \cite{M2002}. We note that, for superprocesses, i.e. when $p=1$, this condition is necessary as well (see \cite[Theorem~8.3.1]{Dawson}); it is likely that it is also necessary for the existence of solutions to \eqref{e_spde1}. Finally, we observe that $\gamma < 1$, along with $d<\frac{2}{\alpha-1}$, ensures the condition $p<1 + \frac 2 d$ is satisfied, and thus weak solutions exist for all the parameter regimes for which the theorem applies.

Let us comment as well on the restrictions on $\gamma$. Our proof is based on the one used by Krylov for the equation with Gaussian noise. Krylov's argument works for all $\gamma \in (0,1)$ in the Gaussian case; using the same strategy, we are able to prove the result for $\gamma$ satisfying the restrictions in Theorem~\ref{thm_compact}. In brief, the parameter restriction is due to several technical difficulties which arise due to the stable noise. In particular, solutions are unbounded and discontinuous and have infinite $q$th moments for $q \geq \alpha$. Another difficulty is posed by the fact that stochastic integrals with respect to the stable noise are discontinuous and are more difficult to control than their continuous counterparts. In the one-dimensional case, extending the result from $\gamma \in [1/\alpha,1)$ to $\gamma \in (2-\alpha,1)$ requires a technical argument based on the moment bounds obtained in Theorem~\ref{thm_stochinteg}.

Heuristic arguments suggest that the smaller $\gamma$ is, the ``easier" it should be for solutions of \eqref{e_spde1} to hit zero, and hence to have compact support. It is remarkable then that we can prove Theorem~\ref{thm_compact} for relatively large values of $\gamma$ but not small ones. (We note, however, that, heuristics aside, the equation is somewhat nicer to work with when $\gamma \geq 1/\alpha$, or equivalently $p \geq 1$.) We conjecture that the compact support property holds for $\gamma \in (0,1 / \alpha)$ for all $d \in [1, \frac{2}{\alpha-1}) \cap \N$. 

We now describe the proof method. In dimension one, it is conceptually the same as the method of Krylov, but in higher dimensions it must be modified. (In all dimensions, there are numerous technical challenges introduced by the stable noise.) First we consider the case $d=1$. The argument is based on an analysis of the local time, or occupation density $A_t(x) = \int_0^t \bar{Y}(s,x) ds$. The main principle is that if $\text{supp}(Y_0)$ is contained in a ball of radius $R>0$, and $A_t(x) = 0$ for some $x > R$ then no mass has reached $[x, \infty)$ by time $t$, because if any mass had ``passed through" the point $x$, then we would have $A_t(x)>0$. The main part of the proof is showing that with probability one, there exists some sufficiently large $x$ such that $A_t(x) = 0$. This argument proceeds first by obtaining a representation for $A_t(x)$ using integration by parts, and then derives delicate stochastic estimates which allow one to compare the distributions of $A_t(x_1)$ and $A_t(x_2)$ for nearby points $x_1 $ and $x_2$.

For $d>1$, the proof must be modified. (With Gaussian noise one is restricted to $d=1$ because solutions do not exist in higher dimensions, so Krylov's original argument was only in dimension one.) In dimension one, it suffices to show that $A_t(x) = 0$ for some $x$ because removing $x$ disconnects the space; there is no other way that mass can move from $(-\infty,x)$ to $(x,\infty)$ except through $x$. When $d>1$, we can no longer disconnect the space by removing a point, but must remove some surface of dimension $d-1$. Hence, instead of quantifying the mass which passes through a point, as with $A_t(x)$, we must quantify the mass which has passed through a surface. Our proof replaces the point with a hyperplane of dimension $d-1$ of the form $\{x\in \R^d : x_1 = R\}$. We define an occupation density on this hyperplane by considering the solution projected to the first coordinate axis. The principle is then the same: if the support of $Y_0$ is contained in the half-space to one side of this hyperplane and the occupation density on the hyperplane equals zero, then no mass has reached the other side of the hyperplane by time $t$. The analysis corresponding to the occupation density in this case is parallel to the arguments in the one-dimensional case, but several new arguments are required. We remark that the paper of Han, Kim and Yi \cite{HKY2023}, in which is proved the compact support property for a class of parabolic SPDEs with coloured-in-space, white-in-time Gaussian noise, similarly generalizes Krylov's method to higher dimensions.\\

\noindent {\bf Organization of the paper.} Section~\ref{s_stablecalc} contains basic results on stochastic integrals with respect to stable martingale measures. We state several important properties of the stable stochastic integral, including a representation as a time-changed stable process. All of the proofs of results in this section are given in an appendix.

We postpone the proof of Theorem~\ref{thm_stochinteg} to the end of the paper so that we may first prove the compact support property (Theorem~\ref{thm_compact}), whose proofs for $d=1$ and $d>1$ are given in Sections~\ref{s_dimone} and \ref{s_highdim}, respectively. The relatively long and technical proof of Theorem~\ref{thm_stochinteg} is then given in Section~\ref{s_stochintegralrep}. We also address a few issues concerning definitions of solutions at the beginning of this section. Note that, despite being later in the paper, the proof of Theorem~\ref{thm_stochinteg} is independent of everything in Sections~\ref{s_dimone} and \ref{s_highdim}, and so the use of Theorem~\ref{thm_stochinteg} in these sections is justified.

The appendix at the end of the paper includes some more background on stochastic integration and the proofs of several results in Section~\ref{s_stablecalc}.\\

\noindent {\bf Notation.} We write $C$ to denote any positive constant whose value is not important and may change line to line in a calculation. To emphasize that an inequality requires enlarging a constant $C$, we sometimes write the new constant as $C'$. To indicate that a constant's value depends on a parameter, e.g. $T>0$, but the value of the constant is unimportant, we may write $C(T)$. We occasionally number constants in an adhoc fashion ($C_1$, $C_2$, etc.) when it is useful to do so.

The symbol $\int$ without any upper or lower bounds of integration will always denote an integral over $\R^d$. We often adopt the shorthand
\begin{equation*}
\int_{(0,t] \times \R^d} \phi(s,x) dsdx = \int_0^t \int_{\R^d} \phi(s,x) dx ds,
\end{equation*}
and we use the same shorthand for stochastic integrals.

\section{Stable stochastic integration} \label{s_stablecalc}
\subsection{Stable processes} \label{s_stable}
Before defining stable martingale measures and their stochastic integrals we review a few important properties of stable processes. We refer to Chapter VIII of Bertoin \cite{Bertoin96} for a detailed discussion of these processes.

In this section and for the rest of the paper, a stable process is assumed to start at zero. We only consider stable processes in one dimension. Let $\alpha \in (1,2)$. A one-dimensional spectrally positive (or one-sided) $\alpha$-stable process $(W_t)_{t \geq 0 }$ is an $\R$-valued L\'{e}vy process with Laplace exponent
\begin{equation*}
\log \E (e^{-\lambda W_t}) = t \lambda^{\alpha}, \quad \lambda \in \R_+.
\end{equation*}
The L\'{e}vy measure of a one-sided $\alpha$-stable process is
\begin{equation} \label{def_jumpmeasure}
\nu(dr) = \sigma_\alpha r^{-1 - \alpha} 1_{\{r > 0\}} dr,
\end{equation}
where
\[\sigma_\alpha = \frac{\alpha(\alpha-1)}{\Gamma(2-\alpha)}.\]
$(W_t)_{t\geq 0}$ is a c\`adl\`ag, purely discontinuous martingale with no negative jumps. For $q>0$, the $q$th moment of $W_t$ is finite if and only if $q < \alpha$. Finally, stable processes are self-similar. In particular, the $\alpha$-stable process is self-similar with index $\alpha$, that is for $\lambda > 0$,
\begin{equation*}
(\lambda W_t)_{t \geq 0} \stackrel{d}{=}  (W_{\lambda^\alpha t})_{t \geq 0},
\end{equation*}
where $\stackrel{d}{=}$ indicates equality in distribution.

\subsection{Stable martingale measures} \label{s_stablemeasures}
We now define the $\alpha$-stable martingale measure $L$ and its distributional derivative, the $\alpha$-stable white noise $\dot L$. Our treatment follows that of Mytnik \cite{M2002}. We first recall the definition of a martingale measure, following Walsh \cite{Walsh}. We denote by $\cB(\R^d)$ the collection of Borel sets of $\R^d$, and write $\underline{\cB}(\R^d)$ for the collection of all Borel sets with finite Lebesgue measure.
\begin{definition} \label{def_martmeasure} Let $(\cF_t)_{t \geq 0}$ be a filtration. A process $M = \{M_t(A) : t \geq 0, A \in \underline{\cB}(\R^d)\}$ is a (local) martingale measure with respect to $(\cF_t)_{t \geq 0}$ if the following hold:

(i) $M_0(A) = 0$ a.s. for all $A \in \underline{\cB}(\R^d)$.

(ii) For $t>0$, $M_t(\cdot)$ is a $\sigma$-finite Borel measure.

(iii) $(M_t(A))_{t \geq 0}$ is a (local) $\cF_t$-martingale.
\end{definition}
Fix $\alpha \in (1,2)$. In the following we write $|A|$ to denote the Lebesgue measure of $A \subseteq \R^d$.
\begin{definition} \label{def_stablemeasure} A martingale measure $L$ is called a spectrally positive $\alpha$-stable martingale measure (or an $\alpha$-stable martingale measure with no negative jumps) if, for all $A \in \underline{\cB}(\R^d)$, the process $(W_t)_{t \geq 0}$ defined by $W_t := L_t(A)$ is a one-sided $\alpha$-stable process run at speed $|A|$. That is, for all $t>0$ and $\lambda>0$,
\begin{equation*}
\log \E( e^{-\lambda L_t(A)}) = t|A|\lambda^\alpha.
\end{equation*}
\end{definition}
The $\alpha$-stable white noise without negative jumps is then defined to be the distributional derivative $\dot L$ of an $\alpha$-stable martingale measure without negative jumps.

It is often useful to express the $\alpha$-stable martingale measure in terms of a compensated Poisson random measure. For a thorough treatment of this material, see Chapter II of Ikeda and Watanabe \cite{IK}, or Chapter II of Jacod and Shiryaev \cite{JS}. Let $N(ds,dx,dr)$ be a Poisson random measure on $\R_+ \times \R^d \times \R_+$ with compensator $\hat{N}(ds,dx,dr) = ds dx \nu(dr)$, where $\nu$ is the $\alpha$-stable jump measure from \eqref{def_jumpmeasure}. Let $\tilde{N}$ denote the compensated Poisson random measure. One can then realize the integral with respect to a stable martingale measure $L$ as
\begin{equation*}
L(ds,dx) = \int_0^\infty r \tilde{N}(ds,dx,dr).
\end{equation*}
That is, given such a compensated Poisson random measure realized on some probability space, $L$ defined as above has the law of a $\alpha$-stable martingale measure with no negative jumps. In general, given a stable martingale measure $L$, one can enlarge the probability space to express it in terms of a Poisson random measure $\tilde{N}$ as above, such that for a permissible integrand $\phi$,
\begin{equation} \label{e_Poisson_integral}
\int_{(0,t] \times \R^d} \phi(s,x) L(ds,dx) = \int_0^t \int_{\R^d} \int_0^\infty \phi(s,x) r \tilde{N}(ds,dx,dr), \quad t \geq 0.
\end{equation}
We will use this representation when proving certain properties of the stable stochastic integral in the appendix.

\subsection{Stochastic integration with stable noise} \label{s_stochcalc}
We now introduce the stochastic integral with respect to the $\alpha$-stable martingale measure and state some of its properties. Per the previous section, stochastic integration with respect to the $\alpha$-stable noise is a special case of stochastic integration with respect to compensated Poisson random measures, for which the basic theory can be found in \cite{IK, JS}. For a more general theory of stochastic integration with respect to random measures, we refer to \cite{CK2015} as well as the references mentioned in its introduction (especially \cite{Bichteler, BJ}). We briefly describe a construction of the integral in an appendix at the end of the paper. There we follow the construction of Balan \cite{B2013}, because this approach is tailored to the stable case and directly gives an optimal integrability criterion for integrands. 

The main purposes of this section are to introduce notation and to state some useful properties of the integral with respect to stable noise which will be used throughout the paper. Some of these properties are special to stable noise, owing to the self-similarity of the jump measure. 

The set-up is a filtered probability space $(\Omega, \cF, (\cF_t)_{t \geq 0}, \bP)$ on which $L$ is a spectrally positive $\alpha$-stable $\cF_t$-martingale measure. We assume here as elsewhere that $(\cF_t)_{t\geq 0}$ satisfies the usual conditions, that is, it is right continuous and $\cF_0$ contains all $\bP$-null sets. We write $\E$ for the expectation associated to $\bP$. 

We begin by defining some function spaces. Let $\cP_0$ denote the predictable $\sigma$-algebra on $\Omega \times \R_+$ associated to $(\cF_t)_{t\geq 0}$, and let $\cP = \cP_0 \times \mathcal{B}(\R^d)$, where the latter is the Borel $\sigma$-algebra on $\R^d$. We call $\phi: \Omega \times \R_+ \times \R^d$ predictable and write $\phi \in \cP$ if $\phi$ is $\cP$-measurable. We will generally omit the dependence on $\omega \in \Omega$ of $\phi$ and we occasionally refer to a process $\{\phi(t,x) : t>0, x\in \R^d\} \in \cP$ as a predictable random field.

For $q \geq 1$ and $t > 0$, let $\|\cdot\|_{q,t}$ denote the standard $\LL^q$-norm on $(0,t]\times \R^d$, and let $\LL^{q,t}$ denote the space of functions for which it is finite. We then define $\LL^{q,t}_{\text{a.s.}}$ by \begin{equation*}
\LL^{q,t}_{\text{a.s.}} = \{ \phi \in \cP : \|\phi\|_{q,t} < \infty \text{ a.s.}\}.
\end{equation*}
We have defined the spaces above for general $q\geq 1$, but it is the case $q = \alpha$ which is most important for our integration theory. For $t\geq 0$ and $\phi : (0,t] \times \R^d \to \R$, define $T_\phi(t) := \|\phi\|^\alpha_{\alpha,t}$. In other words,
\begin{equation*} 
T_\phi(t) := \int_{(0,t] \times \R^d}|\phi(s,x)|^\alpha dsdx,
\end{equation*}
with the convention that $T_\phi(0) = 0$. We remark that we may write $\LL^{\alpha,t}_{\text{a.s.}}$ as 
\begin{equation*}
\LL^{\alpha,t}_{\text{a.s.}} = \{\phi \in \cP : T_\phi(t) < \infty \text{ a.s.}\}.
\end{equation*}
Recalling that $\LL^\alpha(\bP)$ denotes the space of $\alpha$-integrable random variables, we define $\LL^\alpha(\LL^{\alpha,t}) := \{ \phi \in \cP : \|\phi\|_{\alpha,t} \in \LL^\alpha(\bP) \}$, which is equivalent to
\begin{equation*}
\LL^\alpha(\LL^{\alpha,t}) := \{ \phi \in \cP : \E(T_\phi(t)) < \infty \}.
\end{equation*}

We define the stochastic integral as a process indexed by $s \in (0,t]$ for fixed $t>0$, as we have no use for the stochastic integral defined on $(0,\infty)$ in this work. The extension to a process on $(0,\infty)$ is, however, standard. Our convention is that the stochastic integral at time $0$ is equal to $0$, and with this extension we define a process on $[0,t]$. For the class of integrands considered, this convention makes the integral right continuous at time $0$. (This can be proved as a consequence of Proposition~\ref{prop_isometryint}.) We adopt the notation
\begin{equation*}
(\phi \cdot L)_t := \int_{(0,t] \times \R^d} \phi(s,x) L(ds,dx) 
\end{equation*}
to denote the stochastic integral when it is well-defined. To denote the stochastic integral as a process on $[0,t]$ we will often simply write $(\phi \cdot L)$. Because it is an It\^{o} integral with respect to a martingale measure, $(\phi \cdot L)$ is a local martingale. Furthermore, we can, and always will, take c\`adl\`ag versions of our stochastic integrals.

For $t>0$, the stochastic integral $((\phi \cdot L)_s)_{s \in [0,t]}$ is defined for all $\phi \in \LL^{\alpha,t}_{\text{a.s.}}$. This is shown in \cite{B2013}. We refer to that work for details, and also to the appendix of the present work, which summarizes the construction. The rest of the current section contains the statements of several important results concerning the stochastic integral with respect to $L$. The proofs (of the results which are new) are postponed until the appendix.

$T_\phi(t)$ is the critical quantity in the analysis of $(\phi \cdot L)_t$. Indeed, as stated above, one can define the stochastic integral up to time $t$ for precisely the integrands for which $T_\phi(t) < \infty$. (In the non-spatial setting, the space analogous to $\LL^{\alpha,t}_{\text{a.s.}}$ is known to be the optimal space of integrands, c.f. \cite[Theorem~4.1]{RW1986}; we expect the same to be true for space-time integrals.) Moreover, in practice, controlling the stochastic integral $(\phi \cdot L)_t$ generally amounts to controlling $T_\phi(t)$, and the next three results allow one to analyze $(\phi \cdot L)_t$ via $T_\phi(t)$. The first two concern moments of $\sup_{s\in [0,t]} |(\phi\cdot L)_s|$, and the third states that an $\alpha$-stable stochastic integral is a time-changed $\alpha$-stable process.

\begin{proposition} \label{prop_isometryint} There exists a constant $C_\alpha > 1$ such that for all $t>0$ and all $\phi \in \LL^\alpha(\LL^{\alpha,t})$,
\begin{align} \label{e_prop_isometry}
C_\alpha^{-1} \E(T_\phi(t)) \leq  \sup_{\lambda > 0} \lambda^\alpha \bP\bigg( \bigg( \sup_{s \in [0,t]} |(\phi \cdot L)_s| \bigg) > \lambda\bigg) \leq C_\alpha \E(T_\phi(t)).
\end{align}
\end{proposition}
In the non-spatial setting with symmetric stable integrators, the upper bound is due to Gin\'{e} and Marcus \cite[Theorem~3.5]{GM1983} and the lower bound is due to Rosi\'{n}ski and Woyczy\'{n}ski \cite[Theorem~2.1]{RW1986}. The generalization of the upper bound to asymmetric stable noise and to the space-time setting is due to Balan \cite[Lemma~14]{B2013}. The lower bound for asymmetric stable processes does not appear in any of these papers, but the proof requires little modification. For completeness, we prove it in the appendix.

The following moment bound is an easy consequence of Proposition~\ref{prop_isometryint}. It is similar to what one obtains from applying the Burkholder-Davis-Gundy inequality, but it has the advantage that it allows one to bound the moments of stable stochastic integrals directly, without handling the small jumps and large jumps separately.

\begin{lemma} \label{lemma_bdg} Let $q \in [1,\alpha)$. There exists a constant $C_q > 1$ such that for all $t>0$ and all $\phi \in \LL^{\alpha,t}_{\text{a.s.}}$,
\begin{equation} \label{e_lemma_bdg}
\E \bigg( \bigg( \sup_{s \in [0,t]} |(\phi \cdot L)_s| \bigg)^q \bigg)\leq C_q \E(T_\phi(t))^{q/\alpha}.
\end{equation}
Moreover, we may take $C_q = (1 + \frac{q}{\alpha-q} C_\alpha)$, with $C_\alpha$ as in Proposition~\ref{prop_isometryint}.
\end{lemma}

\begin{remark} \label{remark_stop}
In both Proposition~\ref{prop_isometryint} and Lemma~\ref{lemma_bdg}, we may replace $t$ with any stopping time $\tau$ which is a.s. bounded above by $t$. Indeed, this follows by applying the results to the integrand $\hat{\phi}(s,x) = \phi(s,x) 1_{\{s \leq \tau\}}$, and the elementary facts that $(\hat{\phi}\cdot L)_t = (\phi \cdot L)_{t \wedge \tau}$ and $T_{\hat{\phi}}(t) = T_\phi(t \wedge \tau)$ almost surely.
\end{remark}

The next result establishes that stable stochastic integrals are time-changed stable processes. More precisely, due to self-similarity of $\alpha$-stable processes (and $\alpha$-stable martingale measures), $(\phi \cdot L)_s$ has the law of a time-changed $\alpha$-stable process, and the time-change, or inner clock, is $T_\phi(s)$. Because our stable processes are asymmetric, this representation is only true for non-negative integrands; for a general result, one can split into positive and negative parts (c.f. \cite{K1992}).

\begin{proposition} \label{prop_integral_rep} 
For any $t>0$ and non-negative $\phi \in \LL^{\alpha,t}_{\text{a.s.}}$, enlarging our probability space if necessary, there exists a one-sided $\alpha$-stable process $(W_s)_{s \geq 0}$ such that, $\bP$-a.s.,
\begin{equation*} 
(\phi \cdot L)_s  = W_{T_\phi(s)}, \,\, s \in [0,t]. 
\end{equation*}
\end{proposition}

We remark that the enlargement of the probability space in the above is only necessary so that we may define $W_s$ for all $s \geq 0$ and not just $s \in [0,T_\phi(t)]$, which will be convenient. No enlargement is required to define $W$ on $[0,T_\phi(t)]$. The non-spatial version of this result for stochastic integrals with respect to a symmetric $\alpha$-stable process is due to Rosi\'{n}ski and Woyczy\'{n}ski \cite[Theorem~3.1]{RW1986}, and a more general version with signed integrands and asymmetric stable processes is due to Kallenberg \cite[Theorem~4.1]{K1992}. Generalizing the proof to space-time integrals is straightforward, although our proof uses a slightly different argument. It is also simplified by our restriction to non-negative integrands. 

We conclude the section with two more general properties of the stochastic integral. The first is a standard stochastic Fubini-type theorem. 

\begin{lemma} \label{lemma_fubini1} Let $t>0$, $(G, \mathcal{G}, \mu)$ be a finite measure space and $\phi : \Omega  \times (0,t] \times \R^d \times G$ be jointly measurable with respect to $\cP \times \mathcal{G}$. Suppose that
\begin{equation*}
\E\bigg( \int_G \mu(dx) \bigg[ \int_{(0,t] \times \R^d} |\phi(s,y,x)|^\alpha ds dy\bigg] \bigg) < \infty.
\end{equation*}
Then the equality
\begin{align*}
\int_G \mu(dx) \bigg[ \int_{(0,t] \times \R^d} \phi(s,y,x) L(ds,dy) \bigg] =  \int_{(0,t] \times \R^d} \bigg[ \int_G \phi(s,y,x) \mu(dx) \bigg] L(ds,dy)
\end{align*}
holds with probability one.
\end{lemma}

Such results are standard and proceed by approximation with simple functions, for which the result is trivial. See Theorem 2.6 of Walsh \cite{Walsh} for all the details of the $\LL^2$ theory. In the present setting, analogous arguments can be made via approximation in $\LL^\alpha(\LL^{\alpha,t})$; this is justified, for example, by Proposition~\ref{prop_isometryint}. Later we prove another stochastic Fubini theorem, Lemma~\ref{lemma_fubini2}, which specifically pertains to solutions of \eqref{e_spde1} and allows for a partial relaxation of the integrability assumption. 

The following is a version of the dominated convergence theorem. We write the lemma for a family of functions indexed by the integers, but it is also true for families with an uncountable index set.

\begin{lemma}\label{lemma_dctbdg} Let $t>0$ and suppose that $(\phi_n)_{n \in \N}$ is a collection of functions in $\mathbb{L}^{\alpha,t}_{\text{a.s.}}$ such that $\phi_n \to 0$ Lebesgue-a.e. $\bP$-a.s. and $|\phi_n| \leq \phi$ $\bP$-a.s. for all $n \in \N$ for some $\phi \in \mathbb{L}^{\alpha,t}_{\text{a.s.}}$. Then \[\lim_{n \to \infty} \sup_{s \in [0,t]} |(\phi_n \cdot L)_s| = 0\] in probability.
\end{lemma}

\section{Compact support in dimension one}  \label{s_dimone}
This section contains the proof of the compact support property for $d=1$. Although Theorem~\ref{thm_stochinteg} is proved in Section~\ref{s_stochintegralrep}, at the end of the paper, none of the arguments there depend on the arguments in this section, and so we can and will use this result. Throughout this section we use the density process $\bar{Y}(t,x)$ constructed in Theorem~\ref{thm_stochinteg} for all integrals and stochastic integrals associated to the solution.

The set-up for this section is as follows: we assume that $(\Omega,\cF, (\cF_t)_{t \geq 0}, \bP)$ is a stochastic basis on which $(Y,L)$ is a weak solution to \eqref{e_spde1} in dimension $d=1$, for $\alpha \in (1,2)$, with initial data $Y_0 \in \cM_f(\R)$. These assumptions are in force throughout the section and are implicit in the statement of each result. We write $\E$ for the expectation associated to $\bP$. Although our main result is only proved for $\gamma \in (2-\alpha,1)$, most of the lemmas hold for all $\alpha \in (1,2)$ and $\gamma \in (0,1)$. Assumptions on $\alpha$ and $\gamma$ which are stronger than this are stated explicitly; otherwise, the result holds for all $\alpha \in (1,2)$ and $\gamma \in (0,1)$.
We prove the following.

\begin{theorem} \label{thm_onedim}
Suppose that $\gamma \in (2-\alpha, 1)$. If $Y_0([x_0,\infty)) = 0$ for some $x_0 \in \R$, then with probability one, for every $t>0$ there exists a (random) $y(t) \in \R$ such that $Y_s((y(t),\infty)) = 0$ for all $s \in [0,t]$.
\end{theorem}

To conclude that Theorem~\ref{thm_compact}(a) holds from this is straightforward. We have stated the result about the mass to the right of some point, but of course it also holds to the left if $Y_0((-\infty, x_0]) = 0$ for some $x_0$. Hence if $Y_0$ has compact support, given $t>0$, we can apply Theorem~\ref{thm_onedim} to the right and the left and conclude that there exists a random $R >0$ such that $Y_s([-R,R]^c) = 0$ for all $s \leq t$. This is the compact support property.

We introduce some notation. For the remainder of the section we assume that $Y_0$ satisfies $Y_0([x,\infty)) = 0$ for some $x \in \R$. We will denote the minimal such point by $x_r$, that is, $x_r$ is the right endpoint of the support of $Y_0$ defined by
\begin{equation*}
x_r := \inf \{x \in \R : Y_0([x, \infty)) = 0\}.
\end{equation*}
For $t>0$ we define
\begin{equation*} 
V_t := \langle Y_t, 1 \rangle = Y_t(1),
\end{equation*}
and for $(t,x) \in \R_+ \times \R$ we define
\begin{equation*}
A_t(x) := \int_0^t \bar{Y}(s,x) ds.
\end{equation*}
The map $t \to A_t(x)$ is an occupation density of the solution at $x$. From the definition, it is evident that $t \to A_t(x)$ is an a.s. continuous and non-decreasing process. The proof of Theorem~\ref{thm_onedim} is based on the idea that if $Y_0([x,\infty)) = 0$ and $A_t(x) = 0$, then we must have $Y_s((x,\infty)) = 0$ for all $s \in (0,t]$; see Lemma~\ref{lemma_crosstheline}. First, we state some more elementary properties. The first is an immediate consequence of \eqref{e_thm_meanmeasure} and the integrability of the one-dimensional heat-kernel. 

\begin{lemma} \label{lemma_At_momemt}
For all $(t,x)\in \R_+ \times \R$, $\E(A_t(x)) \leq \int_0^t P_s Y_0(x) ds < \infty$.
\end{lemma}

Next we prove an approximation result for $A_t(x)$.

\begin{lemma} \label{lemma_At_convergence} Let $\xi \in \C^\infty_c(\R)$ be non-negative and satisfy $\int \xi = 1$. For $\eps > 0$, define $\xi_\eps(y) = \eps^{-1} \xi(y/\eps)$, and for $(t,x) \in \R_+ \times \R^d$, define
\begin{equation*}
A_t^\eps(x) := \int_0^t (\xi_\eps * Y_s )(x). \end{equation*}

(a) For every $(t,x) \in \R_+ \times \R$, $A_t^\eps(x) \to A_t(x)$ in $\LL^1(\bP)$ as $\eps \downarrow 0$.

(b) For every $x \in \R$, there exists a sequence $(\eps_n)_{n \in \N}$ with $\eps_n \downarrow 0$ (which may depend on $\xi$ and $x$) such that \[\bP\bigg(\lim_{n \to \infty}A_t^{\eps_n}(x) = A_t(x) \text{ for all } t > 0\bigg) = 1.\]
\end{lemma}
\begin{proof}
Fix $x \in \R$ and $t>0$. To simplify notation we define 
\[ R_\eps(s) := \xi_\eps * Y_s(x) - \bar{Y}(s,x).\]
We will show that $\lim_{\eps \to 0} \E(|\int_0^t R_\eps(s) ds|) = 0$. Let $\delta \in (0, t)$. First we control the integral for small values of $s$. By Theorem~\ref{thm_stochinteg}(a), $Y_s(dx) = \bar{Y}(s,x)dx$ for a.e. $s \in (0,\delta]$ $\bP$-a.s., and hence by \eqref{e_thm_meanmeasure} we have
\begin{align} \label{e_L1integaux1}
\E \bigg( \bigg| \int_0^\delta R_\eps(s) \,ds \bigg| \bigg) & \leq 2 \int_0^\delta \sup_y \,\E(\bar{Y}(s,y)) ds \notag 
\\ &\leq 2C Y_0(1) \int_0^\delta s^{-1/2} ds \notag
\\& \leq C \delta^{1/2}.
\end{align}
for a constant $C>0$ which only depends on $Y_0$. Next, by Theorem~\ref{thm_stochinteg}(c), $R_\eps(s) \to 0$ in $\mathbb{L}^1(\bP)$ for all $s \in (0,t]$, and hence there must exist $\eps_0(\delta)>0$ such that for all $\eps \leq \eps_0(\delta)$, 
\begin{equation} \label{e_L1integaux2}
\text{Leb}(B_{\delta,\eps}) < \delta \,\, \text{ for all $\eps < \eps_0(\delta)$,} 
\end{equation}
where $B_{\delta,\eps} = \{s \in (\delta,t]: \E(|R_\eps(s)|) > \delta \}$ and $\text{Leb}(B_{\delta,\eps})$ denotes its Lebesgue measure. Arguing similarly to as in \eqref{e_L1integaux1}, for $\eps \leq \eps_0(\delta)$ we have
\begin{align} \label{e_L1integaux3}
\E \bigg( \bigg| \int_{B_{\delta,\eps}} R_\eps(s)\, ds \bigg| \bigg) & \leq 2 \delta \sup_{y \in \R, s \in [\delta,t]} \E(\bar{Y}(s,y)) \notag
\\ &\leq C \delta^{1/2},
\end{align}
where again the constant does not depend on $\delta$. Combining \eqref{e_L1integaux1}, \eqref{e_L1integaux2} and \eqref{e_L1integaux3}, we obtain that for a constant $C>0$,
\begin{align*}
\E \bigg( \bigg| \int_0^t R_\eps(s)\, ds \bigg| \bigg) &\leq \E \bigg( \bigg| \int_0^\delta R_\eps(s) \,ds \bigg| \bigg) + \E \bigg( \bigg| \int_{B_{\delta,\eps}} R_\eps(s) \,ds \bigg| \bigg) +  \int_{(\delta,t] \backslash B_{\delta,\eps}} \E(|R_\eps(s)|) ds 
\\ &\leq 2C \delta^{1/2} + t \delta 
\end{align*}
for all $\eps \leq \eps_0(\delta)$. Hence for any $\tilde{\delta} > 0$, we can choose $\delta \in (0,t)$ such that the right hand side of the above is less than $\tilde{\delta}$, which proves that $\E(|\int_0^t R_\eps(s)\,ds|) \to 0$ as $\eps \to 0$. This proves part (a).

Now we prove part (b). Fix $x \in \R$ and let $\mathcal{I}$ be a countable, dense subset of $\R_+$. By part (b), $A^\eps_{s}(x) \to A_{s}(x)$ in probability as $\eps \downarrow 0$ for each $s \in \mathcal{I}$, and hence a.s. along some subsequence. Since $\mathcal{I}$ is countable, by a standard argument we may take a diagonal subsequence $(\eps_n)_{n \in \N}$ with $\eps_n \downarrow 0$ such that $\lim_{n\to \infty} A^{\eps_n}_{s}(x) = A_s(x)$ a.s. for every $s \in \mathcal{I}$. In particular, there exists $\Omega_0 \in \cF$ such that $\bP(\Omega_0) = 1$ and for each $\omega \in \Omega_0$, $A^{\eps_n}_{s}(x)(\omega) \to A_{s}(x)(\omega)$ as $n \to \infty$ for every $s \in \mathcal{I}$.

We now fix $\omega \in \Omega_0$ and hereafter suppress dependence on it. Let $t>0$. Let $(a_m)_{m \in \N}$ and $(b_m)_{m \in \N}$ be sequences in $\mathcal{I}$ such that $a_m \uparrow t$ and $b_m \downarrow t$. Since $\xi_\eps \geq 0$, it is immediate from the definition of $A^\eps_s(x)$ that $s \to A^\eps_s(x)$ is non-decreasing for every $\eps>0$, and in particular we have
\begin{equation*}
A^{\eps_n}_{a_m}(x) \leq A^{\eps_n}_{t}(x) \leq A^{\eps_n}_{b_m}(x) 
\end{equation*}
for every $n, m \in \N$. Since $\omega \in \Omega_0$ and $a_m, b_m \in \mathcal{I}$, the left and right-hand sides converge respectively to $A_{a_m}(x)$ and $A_{b_m}(x)$ as $n \to \infty$. As this holds for every $m \in \N$, it follows that
\[ \sup_{m \in \N} A_{a_m}(x) \leq \liminf_{n \to \infty} A^{\eps_n}_{t}(x) \leq  \limsup_{n \to \infty} A^{\eps_n}_{t}(x) \leq \inf_{m \in \N} A_{b_m}(x).\] 
By continuity of $s \to A_s(x)$, the left- and right-hand sides equal $A_t(x)$, and hence $\lim_{n\to \infty} A^{\eps_n}_{t}(x) = A_t(x)$. Hence for $\omega \in \Omega_0$, $\lim_{n\to \infty} A^{\eps_n}_{t}(x)(\omega) = A_t(x)(\omega)$ for every $t>0$. This completes the proof. \end{proof}

For a real-valued stochastic process $(X_t)_{t \geq 0}$, we introduce the following notation:
\begin{align*}
X_t^* &:= \sup_{s \in [0,t]} X_s,  \\
|X_t|^* &:= \sup_{s \in [0,t]} |X_s|,  \\
X_{t,+} &:= (X_t \vee 0),  \\
X_{t,-} &:= - (X_t \wedge 0).  
\end{align*}
Note that $X_{t,-} \geq 0$ so we write $X_t = X_{t,+} - X_{t,-}$. Likewise, for a real-valued function $\phi$ we write $\phi_+ : = \phi \vee 0$ to denote its positive part. We also recall that $V_t = Y_t(1)$, the total mass of $Y_t$, and observe that since $(Y_t)_{t \geq 0} \in \mathbb{D}([0,\infty), \cM_f(\R))$, $V_t^* < \infty$ almost surely. We will use this fact freely in the sequel.

In Lemmas~\ref{lemma_integ}, \ref{lemma_pre_ibp_boundary} and \ref{lemma_ibp_boundary} we extend the stochastic integration by parts formula \eqref{e_spde1weak} to a function $\phi$ which is not smooth in order to obtain an integration by parts formula involving $A_t(x)$. Some of these results are true for all $x \in \R$ but we state and prove them for $x>x_r$ to simplify the proofs.

The next few lemmas use an auxiliary smooth function which for convenience we define here. We fix a function $\psi$ with the following properties:
\begin{align} \label{e_moll_prop}
&\psi \in C^\infty_c(\R), \quad \psi \geq 0, \quad \psi \text{ is even}, \quad  \psi \text{ is non-increasing on $[0,\infty)$,} \notag
\\ &\psi(y) = 1 \text{ for } y\in [0,1], \quad \psi(y) = 0 \text{ for } y \geq 2.
\end{align}

\begin{lemma} \label{lemma_integ}
Let $x > x_r$. Then with probability one, for all $t>0$,
\begin{equation*}
\sup_{s \in [0,t]} \int_{x}^\infty (y-x) Y_s(dy)< \infty.
\end{equation*}
\end{lemma}

\begin{proof} 
Without loss of generality we assume that $x_r < x = 0$. Since the supremum over $[0,t]$ is non-decreasing in $t$, it suffices to prove the result for fixed $t>0$. We observe that
\begin{align*}
\int_0^\infty y Y_s(dy) &\leq \int_0^1 Y_s(dy)+ \int_1^\infty y Y_s(dy) 
\\ & \leq V_s + \int_1^\infty yY_s(dy)
\end{align*}
for all $s \in [0,t]$. Since $V_t^* < \infty$ a.s., it remains to show that the supremum of the second term above over $s \in [0,t]$ is a.s. bounded.

Let $\psi$ be as in \eqref{e_moll_prop}. For $n \in \N$, let $\psi_n(y) = \psi(y/n)$. Let $\eta$ be a smooth, non-decreasing function which vanishes on $(-\infty, \frac 1 4]$ and equals $1$ on $[\frac 3 4, \infty)$. By \eqref{e_spde1weak}, we have
\begin{equation} \label{e_intlemma_ibp1}
\int_{\frac 1 4}^\infty y \eta(y) \psi_n(y) Y_t(dy) = D^n_t  + M^n_t,
\end{equation}
where
\[ D^n_t = \int_0^t \int \left[y \eta(y) \psi_n(y) \right]'' Y_s(dy)ds\]
and  $M^n_t$ is the local martingale
\begin{equation*}
M^n_t = \int_{(0,t] \times \R} y \eta(y) \psi_n(y) \bar{Y}(s,y)^\gamma \,L(ds,dy).
\end{equation*}
An elementary calculation yields that $\sup_n \sup_y |[y \eta(y) \psi_n(y)]''| \leq C$ for some finite $C>0$. This implies that
\[ |D^n_s| \leq C \int_0^s V_u du \leq Ct V_t^*\]
for all $s \in [0,t]$ and $n \in \N$. Thus, the family $\{|D^n_t|^* : n \in \N\}$ is tight. Moreover, because the left-hand side of \eqref{e_intlemma_ibp1} is non-negative it follows that $M^n_s \geq - |D_s^n| \geq - Cs V^*_s$ for all $s >0$. Since $s\to sV^*_s$ is non-decreasing, we obtain
\begin{equation} \label{e_intlemma_mtgbd1}
(M^n_{t,-})^* \leq  Ct V^*_t \quad \forall \, n \in \N.
\end{equation}
We complete the proof by arguing that this uniform bound on the negative part of the family of local martingales implies that the local martingales themselves form a tight family. 

Let $\eps > 0$. Let $K \geq 1$ be sufficiently large so that 
\begin{equation} \label{e_intlemma_negbd}
\sup_{n \in \N} \, \bP( (M^n_{t,-})^* \geq K) < \eps /2.
\end{equation}
Such a $K$ exists by \eqref{e_intlemma_mtgbd1} and the fact that $V_t^*<\infty$ almost surely. Define the stopping times $\tau^n_K = \inf \{ s > 0 : M^n_s \leq - K\}$. Since $M^n_t$ has no negative jumps, $M^n_{\tau^n_K} = -K$ on $\{\tau_K^n < \infty\}$. Thus for each $n \in \N$, $M^n_{t \wedge \tau^n_K}$ is a local martingale bounded below by $-K$, and hence is a supermartingale. In particular, for all $n \in \N$, $M^n_{t \wedge \tau^n_K} + K$ is a non-negative supermartingale with initial value $K$. It then follows from Doob's supermartingale inequality (see e.g. Exercise II.1.15 of Revuz and Yor \cite{RevuzYor}) that for any $\lambda > K $,
\begin{equation*}
\bP((M^n_{t \wedge \tau_K^n})^* \geq \lambda - K) \leq \frac{K}{\lambda}. 
\end{equation*}
In particular, there exists $K' \geq K$ such that
\begin{equation} \label{e_intlemma_posbd}
\sup_{n \in \N} \, \bP( (M^n_{t \wedge \tau_K^n})^* \geq K') < \eps /2.
\end{equation}
It follows from \eqref{e_intlemma_negbd}, \eqref{e_intlemma_posbd} and $K \leq K'$ that
\begin{align*}
\bP(|M^n_t|^* \geq K') &\leq \bP((M^n_{t,-})^* \geq K) + \bP((M^n_{t})^* \geq K', (M^n_{t,-})^*< K)
\\ &\leq \bP((M^n_{t,-})^* \geq K) + \bP((M^n_{t \wedge \tau_K^n})^* \geq K')
\\& < \eps/2 + \eps/2 = \eps. 
\end{align*}
The second inequality uses the fact that, by the definition of $\tau^n_K$, if $(M_{t,-}^n)^* < K$, then $\tau^n_K >t$, which implies that $(M^n_{t \wedge \tau^n_K})^* = (M^n_t)^*$. Thus the family $\{|M^n_t|^* : n \in \N\}$ is tight, as is $\{|D_n^t|^* : n \in \N\}$, and we deduce from \eqref{e_intlemma_ibp1} that the family 
\begin{equation*}
\left\{ \sup_{s \in [0,t]} \int_{\frac 1 4}^\infty y \eta(y) \psi_n(y) Y_s(dy) : n \in \N \right\} 
\end{equation*}
is tight. We remark that
\[ \int_1^\infty y Y_s(dy) \leq \limsup_{n \to \infty} \int_{\frac 1 4}^\infty y \eta(y) \psi_n(y) Y_s(dy).\]
To see this, we note that the integrands on the right hand side increase to $y \eta(y)$, which exceeds the integrand on the left, and the inequality follows. This implies that 
\[\sup_{s \in [0,t]} \int_1^\infty y Y_s(dy) < \infty \text{ a.s.},\] and the result follows.
\end{proof}

We next establish a preliminary integration by parts formula for $A_t(x)$. The desired formula will be obtained by taking $n \to \infty$ in the representation obtained below.

\begin{lemma} \label{lemma_pre_ibp_boundary}
Let $\psi$ be as in \eqref{e_moll_prop}, $x > x_r$ and $n \in \N$. Then with probability one, 
\begin{align*}
&\int_x^\infty (y-x) \psi((y-x)/n) Y_t(dy)  
\\ &\hspace{1 cm}= A_t(x) + \int_{(0,t] \times \R} (y-x)_+ \psi((y-x)/n) \bar{Y}(s,y)^\gamma L(ds,dy) 
\\ &\hspace{1.3 cm}+ n^{-1} \int_0^t ds \int_x^\infty [n^{-1}(y-x) \psi''((y-x)/n) + 2 \psi'((y-x)/n)] Y_s(dy), \quad t\geq 0.
\end{align*}
\end{lemma}
\begin{proof}
We assume without loss of generality that $x_r < x = 0$. For $n \in \N$, let $\psi_n(\cdot) =\psi(\cdot/n)$. We also fix $\xi \in C^\infty_c(\R)$ such that $\xi \geq 0$ and $\int \xi = 1$, and for $\eps > 0$, let $\xi_\eps(y) = \eps^{-1} \xi(y/\eps)$. Finally let $h(y) = y_+ = y \vee 0$. Then the function $f_\eps$ defined by
\[f_\eps(y) = \psi_n(y) (\xi_\eps * h)(y)\]
is smooth and compactly supported, so by \eqref{e_spde1weak} we have
\begin{equation*} 
\langle Y_t, f_\eps\rangle = \langle Y_0 , f_\eps \rangle + \int_0^t \langle Y_s, f_\eps'' \rangle ds + \int_{(0,t] \times \R} f_\eps(y) \bar{Y}(s,y)^\gamma L(ds,dy), \quad t \geq 0.
\end{equation*}
Since $x_r < 0$, we have $\langle Y_0, f_\eps \rangle = 0$ for sufficiently small $\eps$, and hence for small $\eps$ we have
\begin{equation} \label{e_distibp_pre}
\langle Y_t, f_\eps\rangle =  \int_0^t \langle Y_s, f_\eps'' \rangle ds + \int_{(0,t] \times \R} f_\eps(y) \bar{Y}(s,y)^\gamma L(ds,dy), \quad t \geq 0.
\end{equation}
To complete the proof, it suffices to show that for every $N \in \N$, with probability one, simultaneously for all $t \in (0,N]$ we can pass to the limit in each term in the above as $\eps \downarrow 0$. (Remark that the desired equation holds trivially at time $t=0$ so we need not consider this case.)

Let $f(y) = \psi_n(y) h(y)$. We observe that $f_\eps \to f$ point-wise and $K:= \sup_{\eps >0} \|f - f_\eps\|_\infty < \infty$. Since $(Y_t)_{t \geq 0} \in \mathbb{D}([0,\infty),\cM_f(\R))$, $Y_t$ is a finite measure for all $t \in (0,N]$ a.s., and dominated convergence therefore implies that the left hand side of \eqref{e_distibp_pre} converges to $\langle Y_t, f \rangle$ as $\eps \downarrow 0$ for every $t \in (0,N]$, that is,
\begin{equation} \label{e_distibp_pre_1}
\bP\bigg(\lim_{\eps \downarrow 0} \langle Y_t , f_\eps \rangle = \langle Y_t , f \rangle \,\,\, \forall \, t \in (0,N]\bigg) = 1.
\end{equation}
The second derivative of $f_\eps$ is
\[f_\eps''(y) = n^{-2} \psi''(y/n)(\xi_\eps * h)(y) + 2n^{-1} \psi'(y/n) (\xi_\eps * 1_{\{y > 0\}})(y) + \psi(y/n) \xi_\eps(y),\]
where we have used the fact that $h' = 1_{\{y>0\}}$ and $h'' = \delta_0$ in the distributional sense. As $\eps \downarrow 0$, the first two terms converge (in the bounded point-wise sense) to $n^{-2}\psi''(y/n) h(y)$ and $2n^{-1} \psi'
(y/n) 1_{\{y > 0\}}$, which are bounded. Since with probability one, $\int_0^t Y_s(\cdot) ds$ is a finite measure for all $t \in [0,N]$, we can also pass to the limit simultaneously for all $t \in (0,N]$ for these integrands by dominated convergence.

For the last term in $f_\eps''$, we observe that $\psi(y/n) \xi_\eps =\xi_\eps$ for sufficiently small $\eps > 0$, because $\psi(y/n) = 1$ in a neighbourhood of zero, and that $(\xi_\eps)_{\eps >0}$ is an approximate identity satisfying the assumptions of Lemma~\ref{lemma_At_convergence}. Thus, restricted to the subsequence $(\eps_m)_{m \in \N}$ from part (b) of that lemma, we a.s. have
\begin{equation*}
\int_0^t \langle Y_s, \xi_{\eps_m} \rangle ds = \int_0^t (\xi_{\eps_m} * Y_s)(0) ds \to A_t(0) \,\, \text{ as $m \to \infty$ for all } t \geq 0.
\end{equation*}
Combined with the previous argument for the other terms in $f_\eps''$, we have therefore shown that
\begin{equation} \label{e_distibp_pre_2}
\bP\bigg(\lim_{m \to \infty} \int_0^t \langle Y_s, f_{\eps_m}''\rangle ds = A_t(0) + \int_0^t \langle Y_s, g_n\rangle ds \, \, \forall \, t \in (0,N] \bigg) = 1.
\end{equation} 
where, $g_n(y) = n^{-1}[n^{-1}\psi''(y/n) h(y) + 2\psi'(y/n) 1_{\{y > 0\}}]$.

Finally, consider the stochastic integral
\begin{equation*}
\int_{(0,t] \times \R} (f(y) - f_\eps(y)) \bar{Y}(s,y)^\gamma L(ds,dy).
\end{equation*}
The integrand vanishes point-wise a.e. as $\eps \downarrow 0$, and its absolute value is bounded above by $K\bar{Y}(s,y)^\gamma$, which is in $\mathbb{L}^{\alpha,N}_{\text{a.s.}}$ for every $N \in \N$. Hence, by Lemma~\ref{lemma_dctbdg}, the stochastic integral above vanishes in probability, uniformly for $t \in [0,N]$, as $\eps \downarrow 0$. In particular, there is a subsequence $(\eps_m')_{m \in \N}$ of the sequence $(\eps_m)_{m \in \N}$ along which the stochastic integrals vanish uniformly a.s., and hence, with probability one,
\begin{align*} 
\lim_{m \to \infty} \int_{(0,t] \times \R} f_{\eps_m'}(y) \bar{Y}(s,y)^\gamma L(ds,dy) = \int_{(0,t] \times \R} f(y) \bar{Y}(s,y)^\gamma L(ds,dy) \,\, \forall \, t \in (0,N].
\end{align*}
Taking $\eps\downarrow 0$ along the sequence $(\eps_m')_{m \in \N}$, \eqref{e_distibp_pre_1}, \eqref{e_distibp_pre_2}, and the above now imply that, with probability one, the desired formula holds for all $t \in (0,N]$, which implies the result. \end{proof}

We now complete the integration by parts formula for $A_t(x)$. Part (c) states a (uniform) continuity result for $A_t(x)$ in the spatial variable.

\begin{lemma}\label{lemma_ibp_boundary} (a) For every $x > x_r$, the function $\phi = \phi(s,y) = (y-x)_+ \bar{Y}(s,y)^\gamma$ is in $\mathbb{L}^{\alpha,t}_{\text{a.s.}}$ for every $t>0$, and hence the stochastic integral $(\phi \cdot L)_t$ is well-defined for all $t \geq 0$.

(b) For all $x > x_r$, we a.s. have
\begin{align} \label{e_main_formula}
\int_{x}^\infty (y-x)  Y_t(dy) &= A_t(x) +\int_{(0,t] \times \R} (y-x)_+ \bar{Y}(s,y)^\gamma  L(ds,dy), \quad t\geq 0.
\end{align}

(c) If $(x_n)_{n \in \N} \subset (x_r, \infty)$ and $\lim_{n \to \infty} x_n = x > x_r$, then for every $t>0$, $A_s(x_n) \to A_s(x)$ uniformly on $s \in [0,t]$ in probability as $n \to \infty$.
\end{lemma}

\begin{proof}
We assume without loss of generality that $x_r < x = 0$. From Lemma~\ref{lemma_pre_ibp_boundary}, for $n \in \N$, with probability one, for all $t \geq 0$ we have
\begin{align} \label{e_ibplemma11}
&\int_0^\infty y \psi_n(y) Y_t(dy)  = A_t(0) + M^n_t + n^{-1} \int_0^t ds \int_0^\infty [n^{-1} y \psi''(y/n) + 2 \psi'(y/n)] Y_s(dy),
\end{align}
where $\psi$ is as in \eqref{e_moll_prop} and
\begin{equation*}
M^n_t := \int_{(0,t] \times \R} y_+ \psi_n(y) \bar{Y}(s,y)^\gamma L(ds,dy).
\end{equation*}
To take $n\to \infty$ in \eqref{e_ibplemma11} and pass to the limit, we use martingale arguments as in the proof of Lemma~\ref{lemma_integ}. Since $\psi'(\cdot/n)$ and $\psi''(\cdot / n)$ are bounded and vanish on $[2n,\infty)$, the square bracketed term is bounded by a constant $C$ in absolute value. We can thus bound the absolute value of the rightmost term in \eqref{e_ibplemma11} as in the proof of Lemma~\ref{lemma_integ}, which yields
\begin{align} \label{e_ibplemma2}
n^{-1} \left| \int_0^t ds \int [n^{-1} y \psi''(y/n) + 2 \psi'(y/n)] Y_s(dy) \right| \leq  C t n^{-1} V_t^*.
\end{align}
Since the left-hand side of \eqref{e_ibplemma11} is non-negative, using the above bound we obtain that for $s \leq t$,
\begin{equation*}
M^n_s \geq -  A_s(0) - Cs n^{-1} V_s^*.
\end{equation*}
Since $A_s(0)$ and $s V_s^*$ are both non-decreasing in $s$, it follows that
\begin{equation*}
(M^n_{t,-})^* \leq  A_t(0) + Ct n^{-1} V_t^* \leq  A_t(0) + Ct V_t^* \quad \text{ for all } n \in \N.
\end{equation*}
On the other hand, from \eqref{e_ibplemma11} and \eqref{e_ibplemma2} we obtain that
\begin{equation*}
(M_{t,+}^n)^* \leq Ct n^{-1} V_t^* + \sup_{s \in [0,t]} \int_0^\infty y Y_s(dy) \quad \text{ for all } n \in \N.
\end{equation*}
By Lemma~\ref{lemma_integ}, the right-hand side above is a.s. finite. Since $(M^n_{t,-})^*$ and $(M^n_{t,+})^*$ both have a.s. finite upper bounds which are uniform in $n$, we conclude that the family $\{|M^n_t|^* : n \in \N\}$ is tight for every $t>0$.

From Proposition~\ref{prop_integral_rep}, for each $n \in \N$ there is a stable process, which we denote by $(W^n_t)_{t \geq 0}$, such that
\begin{equation} \label{e_ibplemma_timechange}
M^n_t = W^n_{T_n(t)},
\end{equation}
where 
\begin{equation} \label{e_ibplemma_Tn}
T_n(t) = \int_{(0,t] \times [0,\infty)} y^\alpha \psi_n(y)^\alpha \bar{Y}(s,y)^p dsdy.
\end{equation}
As a consequence of \eqref{e_ibplemma_timechange}, the tightness of $\{|M^n_t|^* : n \in \N\}$ implies the tightness of $\{T_n(t) : n \in \N\}$. We will prove this claim momentarily, but first use it to conclude the proof of part (a). Since $T_n(t)$ is increasing in $n$, tightness implies that $T_n(t)$ almost surely converges to a finite limit as $n \to \infty$. Applying monotone convergence in \eqref{e_ibplemma_Tn}, we obtain that, $T_n(t) \uparrow T(t) < \infty$ a.s., where
\begin{equation*}
T(t) := \int_{(0,t] \times [0,\infty)} y^\alpha \bar{Y}(s,y)^p dsdy.
\end{equation*}
It follows that the function $(s,y) \to y \bar{Y}(s,y)^\gamma$ is in $\mathbb{L}^{\alpha,t}_{\text{a.s.}}$, which is what we wanted to show.

We now prove the claim that the tightness of $\{|M^n_t|^* : n \in \N\}$ implies the tightness of $\{T_n(t) : n \in \N\}$. Suppose, to the contrary, that $\{T_n(t) : n \in \N\}$ is not tight. We will prove that $\{|M^n_t|^* : n \in \N\}$ is not tight and thereby obtain a contradiction. Under the assumption that $\{T_n(t) : n \in \N\}$ is not tight, there exists $\eps > 0$ such that for every $K \in \N$, there exists $n_K \in \N$ such that $\bP(T_{n_K}(t) \geq K) > \eps$. We choose such an $\eps$. Next, let $\lambda >0$. It is easy to show from the scaling property that there exists $K = K(\eps,\lambda) \in \N$ such that for an $\alpha$-stable process $(W_s)_{s \geq 0}$,
\begin{equation*}
\bP(|W_K| > \lambda) \geq 1 - \eps/2.
\end{equation*}
Fix such a $K$. Then for $n = n_K$ as above, by \eqref{e_ibplemma_timechange} we have
\begin{align*}
\bP(|M^n_t|^* > \lambda) \geq \bP(|W^n_{T_n(t)}|^* > \lambda, T_n(t) > K) &\geq \bP(|W_K^n| > \lambda, T_n(t) > K),
\end{align*}
since if $T_n(t) \geq K$ we have $|W^n_{T_n(t)}|^* \geq |W^n_K|^* \geq |W^n_K|$. Finally, by our choice of $n$ and $K$, we have $\bP(|W_K^n| > \lambda) \geq 1 - \eps/2$ and $\bP(T_n(t) \geq K) > \eps$. It follows that
\begin{align*}
\bP(|W_K^n| > \lambda, T_n(t) > K) &= \bP(|W_K^n| > \lambda) + \bP(T_n(t) > K) - \bP(|W_K^n| > \lambda \text{ or } T_n(t) > K)
\\&\geq (1 -\eps/2) + \eps - 1 = \eps/2.
\end{align*}
Hence, we have shown that $\bP(|M^n_t|^* > \lambda) \geq \eps/2$ for $n = n_K$. Since we can do this for every $\lambda > 0$, we have shown that $\{|M^n_t|^* : n \in \N\}$ is not tight, obtaining the desired contradiction. This proves the claim and thus completes the proof of part (a). 

In order to prove part (b) we return to \eqref{e_ibplemma11}. By \eqref{e_ibplemma2}, the rightmost term in \eqref{e_ibplemma11} vanishes uniformly for $t \in [0,N]$ a.s. as $n\to \infty$ for any $N \in \N$. For the term on the left hand side, monotone convergence implies that with probability one, for all $t \geq 0$,  
\[ \lim_{n \to \infty} \int_0^\infty y \psi_n(y) Y_t(dy) = \int_0^\infty y Y_t(dy). \]
Thus, it suffices to establish the convergence of $M^n_t$ to the desired stochastic integral. We write
\[ M_t = \int_{(0,t] \times [0,\infty)} y \bar{Y}(s,y)^\gamma L(ds,dy)\]
and observe that
\begin{equation*}
M_t - M^n_t = \int_{(0,t] \times [0,\infty)} y[1 - \psi_n(y)]\bar{Y}(s,y)^\gamma  L(ds,dy). 
\end{equation*}
The integrand vanishes point-wise and is bounded above by $y_+ \bar{Y}(s,y)^\gamma$, which is in $\LL^{\alpha,t}_{\text{a.s.}}$ by part (a). Hence $|M_t - M_t^n|^*$ converges to $0$ in probability as $n \to \infty$ by Lemma~\ref{lemma_dctbdg} for any $t>0$. Taking $t = N \in \N$ and restricting to a subsequence $(n_k)_{k \in \N}$ on which $|M_N - M_N^{n_k}|^*$ vanishes a.s., it follows that $\bP(\lim_{k \to \infty} M^{n_k}_t = M_t \, \text{ for all } t \in [0,N]) = 1$. Combined with what we have proved for the other terms, this implies that \eqref{e_main_formula} a.s. holds simultaneously for all $t \in [0,N]$ for every $N \in \N$, and hence for all $t \geq 0$. This completes the proof of (b).

We now prove part (c). Let $t>0$ and suppose that $(x_n)_{n \in \N}$ satisfies $x_n > x_r$ for all $n \in \N$ and $\lim_{n \to \infty} x_n = 0$. If $x_n < 0$, using the representation from part (b) at both $0$ and $x_n$ gives, for every $s \in [0,t]$,
\begin{align*}
A_s(0) - A_s(x_n) = -x_n Y_s([0,\infty)) + \int_{x_n}^{0} (y-x_n) Y_s(dy)  + M_s(x_n) - M_s(0),
\end{align*}
where 
\[M_s(x) := \int_{(0,s] \times [x,\infty)} (y-x) \bar{Y}(u,y)^\gamma L(du,dy)\]
for $x=0$ and $x=x_n$. The absolute value of the sum of the first two terms on the right hand side is bounded above by $|x_n| Y_s(1)$, and one obtains the same upper bound when $x_n > 0$. Hence, uniformly for $s \leq t$, these terms are bounded above in absolute value by $|x_n| V_t^*$. Taking the supremum over $s \in [0,t]$, we obtain
\begin{align*}
|A_t(0) - A_t(x_n)|^* \leq |x_n|V_t^* + |(F_n \cdot L)_t|^*,
\end{align*} 
where $F_n(s,y) = [(y)_+ - (y - x_n)_+]\bar{Y}(s,y)^\gamma$. Since $F_n(s,y)$ vanishes point-wise as $n \to \infty$ and $|F_n(s,y)| \leq |x_n| \bar{Y}(s,y)^\gamma \leq  \bar{Y}(s,y)^\gamma \in \LL^{\alpha,t}_{\text{a.s.}}$, where the last inequality holds for sufficiently large $n$, $|(F_n \cdot L)_t|^*$ converges to $0$ in probability as $n\to \infty$ by Lemma~\ref{lemma_dctbdg}. Hence, by the above, $|A_t(0) - A_t(x_n)|^*$ converges to $0$ in probability as $n \to \infty$. This proves part (c) and completes the proof of the lemma. \end{proof}

The next lemma formalizes the idea that, if the solution initially has no mass to the right of $x$, then a positive amount of mass must pass through $x$ for there to be mass to the right of $x$ at a future time.

\begin{lemma} \label{lemma_crosstheline} Suppose that $x > x_r$. Then for any a.s. bounded stopping time $\tau$, with probability one, 
\[ A_\tau(x) = 0 \Rightarrow Y_t((x,\infty)) = 0\,\, \text{ for all } t \leq \tau .\]
\end{lemma}
\begin{proof}
Without loss of generality suppose that $x_r < x = 0$. By Lemma~\ref{lemma_ibp_boundary}(b),
\begin{equation} \label{e_lemma_crosstheline1}
\int_{0}^\infty y Y_t(dy) = A_t + M_t, \quad t \geq 0,
\end{equation}
where $A_t = A_t(0)$ and $M_t$ is a local martingale. The left-hand side of \eqref{e_lemma_crosstheline1} is non-negative, and hence $A_t + M_t \geq 0$ for all $t\geq 0$. Let $\tau$ be any a.s. bounded stopping time and let $\sigma = \inf \{s > 0 : A_s > 0\}$. Since $A_t$ is continuous and non-decreasing, $A_{\sigma \wedge t} = 0$ for all $t \geq 0$ a.s. Arguing as in the proofs of Lemmas~\ref{lemma_integ} and \ref{lemma_ibp_boundary}, we see that the stopped local martingale $M_{t \wedge \tau \wedge \sigma}$ is in fact a supermartingale started from $0$. Since $A_t + M_t \geq 0$ and $t \to A_t$ is non-decreasing, we must have $M_{t \wedge \tau \wedge \sigma} \geq - A_{\sigma \wedge t} = 0$. Hence, with probability one, for all $t \geq 0$ the negative part of $M_{t \wedge \tau \wedge \sigma}$ equals zero, and we must have that $M_{t \wedge \tau \wedge \sigma}$ is identically zero. In particular, $M_t = 0$ for all $t \leq \tau \wedge \sigma$. It now follows from \eqref{e_lemma_crosstheline1} that $\int_{0}^\infty y Y_t(dy) = 0$ for all $t \leq \tau \wedge \sigma$ almost surely. The lemma follows. \end{proof}

Before we state and prove the last main technical lemma, we prove an elementary lemma about stable processes stopped upon leaving an interval. 

\begin{lemma} \label{lemma_stable_exit} Let $(W_s)_{s \geq 0}$ be a spectrally positive $\alpha$-stable process started from $0$, let $b \in (0,1]$, $\delta \in (0,1)$, and define the stopping times 
\[\hat{\tau}_1 = \inf\{s >0 : W_s \leq -b\}, \quad \hat{\tau}_2 = \inf \{s>0 : W_s \geq b^{1-\delta}\}, \quad \hat{\tau} = \hat{\tau}_1 \wedge \hat{\tau}_2.\]
The following hold:

(a) $\bP(\hat{\tau} = \hat{\tau_2}) = \bP(\hat{\tau}_2 < \hat{\tau}_1) \leq b^\delta$.

(b) $\sup_{\lambda >0} \lambda^\alpha \bP\left(|W_{\hat{\tau}}|^* > \lambda \right) \leq 2^\alpha b^{\alpha (1-\delta)}$.
\end{lemma}
\begin{proof}
\cite[Theorem~1]{Port1970} gives the distribution of a one-sided $\alpha$-stable process at the exit time of an interval. For the interval we consider, the distribution is given by
\begin{align}
&\bP(W_{\hat{\tau}} = -b) = \left(\frac{1}{1+b^{\delta}}\right)^{\alpha-1}, \notag
\\ &\bP(W_{\hat{\tau}} \in [b^{1-\delta},b^{1-\delta} + y]) \label{e_stablehitting}
\\ &\hspace{1 cm}= \kappa_\alpha b^{1 + (\alpha - 1)(1-\delta)}\int_0^y z^{-(\alpha-1)} (b^{1-\delta}+z)^{-1} (b+b^{1-\delta}+z)^{-1} dz, \notag
\end{align}
for all $y \geq 0$, where $\kappa_\alpha = \frac{\sin(\pi(\alpha-1))}{\pi}$. Note that $\bP(W_{\hat{\tau}} < -b) = 0$ because $W$ has no negative jumps.

The first equality in \eqref{e_stablehitting} can be used to prove part (a) but we opt for a cleaner martingale argument. It is straightforward to argue that $(W_{s \wedge \hat{\tau}})_{s \geq 0}$ is a uniformly integrable martingale and $\hat{\tau} < \infty$ a.s., so the martingale convergence theorem implies that
\begin{align*}
0 = \E(W_{\hat{\tau}}) = \E(W_{\hat{\tau}}1_{\{ \hat{\tau} = \hat{\tau}_1\}}) +\E(W_{\hat{\tau}}1_{\{ \hat{\tau} = \hat{\tau}_2\}})  &= -b \bP(\hat{\tau} = \hat{\tau}_1) + \E(W_\tau 1_{\{\hat{\tau} = \hat{\tau}_2\}}) 
\\ & \geq -b \bP(\hat{\tau} = \hat{\tau}_1) + b^{1-\delta} \bP(\hat{\tau} = \hat{\tau}_2).
\end{align*}
The inequality holds because $W_{\hat{\tau}_2} \geq b^{1-\delta}$. Since $\hat{\tau} = \hat{\tau}_2$ is equivalent to $\hat{\tau}_2 < \hat{\tau}_1$, rearranging we obtain that
\begin{equation*}
\bP( \hat{\tau}_2 < \hat{\tau}_1) \leq b^{1-(1-\delta)} \bP(\hat{\tau} = \hat{\tau}_1) \leq b^\delta.
\end{equation*}
This proves part (a).

We now prove part (b). We consider two different cases for values of $\lambda$. For $\lambda \in (0,2b^{1-\delta}]$, we use the trivial bound
\begin{align*}
\sup_{\lambda \in (0,2b^{1-\delta}]} \lambda^\alpha \bP( |W_{\hat{\tau}}| > \lambda) \leq 2^\alpha b^{\alpha(1-\delta)}.
\end{align*}
Now consider $\lambda > 2b^{1-\delta}$. Since $b \leq 1$, we have $ \lambda > 2b^{1-\delta} > b$, so if $|W_{\hat{\tau}}| > \lambda$, we in fact have $W_{\hat{\tau}} > \lambda$. Hence, using \eqref{e_stablehitting}, we obtain
\begin{align*}
\bP(|W_{\hat{\tau}}| > \lambda) &= \bP(W_{\hat{\tau}} \geq b^{1-\delta}) -  \bP(W_{\hat{\tau}} \in [b^{1-\delta},\lambda])
\\ &= \kappa_\alpha b^{1 + (\alpha - 1)(1-\delta)}\int_{\lambda - b^{1-\delta}}^\infty z^{-(\alpha-1)} (b^{1-\delta}+z)^{-1} (b+b^{1-\delta}+z)^{-1} dz
\\ &\leq \kappa_\alpha b^{1 + (\alpha - 1)(1-\delta)} \int_{\lambda-b^{1-\delta}}^\infty z^{-\alpha - 1} dz
\\ &\leq \frac{\kappa_\alpha}{\alpha} b^{1 + (\alpha - 1)(1-\delta)} (\lambda - b^{1-\delta})^{-\alpha}
\\ &\leq  \frac{2^\alpha \kappa_\alpha}{\alpha} b^{1 + (\alpha - 1)(1-\delta)} \lambda^{-\alpha}.
\end{align*}
The last line holds because, since $\lambda \geq 2b^{1-\delta}$, we have $\lambda - b^{1-\delta} \geq \lambda / 2$. Combining the bounds from the two cases, we obtain that
\[\sup_{\lambda > 0 }\lambda^{\alpha} \bP(|W_{\hat{\tau}}| > \lambda) \leq \max \left\{2^\alpha b^{\alpha(1-\delta)}, \frac{2^\alpha \kappa_\alpha}{\alpha} b^{1 + (\alpha - 1)(1-\delta)} \right\}.\]
The smaller of the exponents is $\alpha(1-\delta)$ and the larger of the constants is $2^\alpha$, so for all $b \in (0,1]$ the maximum is equal to $2^\alpha b^{\alpha(1-\delta)}$ and the claimed inequality follows.
\end{proof}

We now prove the main technical result used in the proof of Theorem~\ref{thm_onedim}. It is analogous to parts (ii) and (iii) of \cite[Lemma 2.1]{Krylov1997}, but both the statement and the proof are modified in our setting owing to technical issues arising due to the stable noise, such as unboundedness of the solution and discontinuity of related stochastic integrals. We point out in particular that the lower bound from Proposition~\ref{prop_isometryint} is used in a key step, whereas the analogous argument in \cite{Krylov1997} uses the Burkholder-Davis-Gundy inequality, which cannot be applied in the same way here.

\begin{lemma} \label{lemma_probbound}
(a) Let $\gamma \in [1/\alpha,1)$. There is a universal constant $C>0$ such that the following holds: for any $t>0$, $x_0 > x_r$, any $\delta \in (0,1-\gamma)$, and all $a, b, r \in (0,1)$ there exists $x_1 \in [x_0 + r, x_0 + 2r]$ such that
\begin{equation*} 
\bP (A_t(x_1) \geq a) \leq \bP(A_t(x_0) \geq b) + b^{\delta} + C \frac{t^{(p-1)/\alpha}}{r^{1 + 1/\alpha}}\left( \frac{b^{1-\delta}}{ a^\gamma} \right).
\end{equation*}

(b) Suppose $\gamma \in (0,1/\alpha)$, and let $\theta, \beta \in (1,\alpha)$. For any $t>0$, there is a universal constant $C_1>0$ and a constant $C_2(t,\theta,\beta, Y_0(1)) > 0$ such that the following holds: for any $x_0 \geq x_r + 1$ and $K \geq 1$, any $\delta \in (0,1-\gamma)$, and all $a, b, r \in (0,1)$, there exists $x_1 \in [x_0 + r, x_0 + 2r]$ such that 
\begin{align*} 
&\bP ( A_t(x_1) \geq a ) \\
& \hspace{0.8 cm}\leq \bP(A_t(x_0) \geq b) + b^{\delta} + C_1 \frac{K^{(1-p)/\alpha}}{r^{1 + 1/\alpha}}  \left( \frac{b^{1-\delta}}{ a^{1/\alpha}} \right) + C_2(t,\theta,\beta,Y_0(1)) \left( \frac{K^{-\beta(\theta-1)/\theta}}{a} \right). 
\end{align*}
\end{lemma}

\begin{proof}
Let $t>0$, $x_0 > x_r$. For $s \in [0,t]$, we define 
\begin{align}
M_s(x_0) &:= \int_{(0,s] \times \R} (y-x_0)_+ \bar{Y}(u,y)^\gamma L(du,dy),\notag\\
T(s,x_0) &:= \int_{(0,s] \times \R} (y-x_0)_+^\alpha \bar{Y}(u,y)^{p} du dy. \label{e_probbdlemma_Tdef}
\end{align}
Then from Lemma~\ref{lemma_ibp_boundary}(b), since $x_0 > x_r$, we almost surely have
\begin{equation} \label{e_probbdlemma_formula1}
\int_{x_0}^\infty (y-x_0) Y_t(dy) = A_t(x_0) + M_t(x_0), \quad t \geq 0.
\end{equation}
Now let $\delta \in (0,1-\gamma)$ and $a,b,r \in (0,1)$. We introduce the stopping times
\begin{align*}
 \sigma_1 &:= \inf \{s > 0: A_s(x_0) \geq b\}, 
\\  \sigma_2 &:= \inf \{s > 0 : M_s(x_0) \geq b^{1-\delta}\}, 
\\ \tau &:= t \wedge \sigma_1 \wedge \sigma_2.
\end{align*}
We observe that for any $x$,
\begin{align} \label{e_probbdlemma_ineq1}
&\bP(A_t(x) \geq a)  \notag
\\ &\hspace{0.8 cm} \leq \bP(A_t(x) \geq a, \tau = \sigma_1 ) + \bP( A_t(x) \geq a, \tau = t) + \bP(A_t(x) \geq a, \tau = \sigma_2 < \sigma_1 \wedge t) \notag
\\ &\hspace{0.8 cm} \leq \bP(A_t(x_0) \geq b) + \bP(A_\tau(x) \geq a, \tau = t) + \bP(\sigma_2 < \sigma_1 ,\sigma_2 < t) \notag
\\ &\hspace{0.8 cm} \leq \bP(A_t(x_0) \geq b) + \bP(A_\tau(x) \geq a) + \bP(\sigma_2 < \sigma_1 ,\sigma_2 < t).
\end{align}
Next we obtain an upper bound for $\bP(\sigma_2 < \sigma_1, \sigma_2 < t)$. Since the left hand side of \eqref{e_probbdlemma_formula1} is non-negative, it follows that for any stopping time $\sigma$, $M^*_{\sigma,-}(x_0) \leq A_\sigma(x_0)$. In particular, if we define
\[ \hat{\sigma}_1 := \inf \{s > 0: M_s(x_0) \leq -b\},\]
then $\hat{\sigma}_1 \geq \sigma_1$ a.s., and hence \begin{equation*}
\bP(\sigma_2 < \sigma_1, \sigma_2 < t) \leq \bP(\sigma_2 < \hat{\sigma}_1, \sigma_2 < t).
\end{equation*}
Thus, it suffices to find an upper bound for the right hand side above. By Proposition~\ref{prop_integral_rep}, there is an $\alpha$-stable process $(W_s)_{s \geq 0}$ such that $M_s(x_0) = W_{T(s,x_0)}$, $s \in [0,t]$. Since $\hat{\sigma}_1$ and $\sigma_2$ are first passage times for $M_s(x_0)$, their ordering is independent of time changes. In particular, if we define $\hat{\tau}_1$ and $\hat{\tau}_2$ as in Lemma~\ref{lemma_stable_exit}, i.e.
\[\hat{\tau}_1 = \inf\{s >0 : W_s \leq -b\}, \quad \hat{\tau}_2 = \inf \{s>0 : W_s \geq b^{1-\delta}\},\] 
then $\sigma_2 < \hat{\sigma}_1$ is equivalent to $\hat{\tau}_{2} < \hat{\tau}_1$, provided $\sigma_2 < t$. Noting that $\sigma_2 < t$ is equivalent to $\hat{\tau}_2 < T(t,x_0)$ and the latter is finite a.s., we conclude that
\[ \{\sigma_2 < \hat{\sigma}_1, \sigma_2 < t\} = \{\hat{\tau}_{2} < \hat{\tau}_{1}, \hat{\tau}_2 < T(t,x_0)\}.\]
Hence, 
\begin{align*}
\bP(\sigma_2 < \hat{\sigma}_1, \sigma_2 < t) &= \bP(\hat{\tau}_{2} < \hat{\tau}_{1}, \hat{\tau}_{2} < T(t, x_0)) \leq \bP(\hat{\tau}_{2} < \hat{\tau}_{1}) \leq b^\delta.
\end{align*}
The final inequality is a direct application of Lemma~\ref{lemma_stable_exit}(a). In particular, we have shown that
\begin{equation*}
\bP(\sigma_2 < \sigma_1 , \sigma_2 < t) \leq b^\delta.
\end{equation*}
Substituting this into \eqref{e_probbdlemma_ineq1}, we obtain
\begin{align} \label{e_probbdlemma_ineq11}
\bP(A_t(x) \geq a) &\leq \bP(A_t(x_0) \geq b) + b^\delta + \bP( A_\tau(x) \geq a).
\end{align} 
The rest of the proof is divided into the two cases (a) and (b) from the lemma's statement.

{\bf Case (a): $\gamma \in [\alpha^{-1}, 1)$.} We integrate \eqref{e_probbdlemma_ineq11} over $[x_0 + r, x_0 + 2r]$ and apply Markov's inequality to obtain
\begin{equation} \label{e_probbdlemma_ineq2}
\frac 1 r \int_{x_0+r}^{x_0+2r} \bP(A_t(x) \geq a) dx \leq \bP(A_t(x_0) \geq b) + b^{\delta} + a^{-\gamma} \frac 1 r  \int_{x_0+r}^{x_0+ 2r} \E(A_\tau(x)^\gamma) dx.
\end{equation}
Now, since $\gamma \geq \alpha^{-1}$, $p = \alpha \gamma \geq 1$, so by Jensen's inequality,
\begin{equation*}
A_{\tau}(x) = \tau \left(\frac 1 \tau \int_0^{ \tau} \bar{Y}(s,x) ds\right) \leq \tau^{1 - 1/p} \bigg( \int_0^{ \tau} \bar{Y}(s,x)^{p} ds \bigg)^{1/p}, 
\end{equation*} 
and hence
\begin{equation*}
A_\tau(x)^\gamma \leq t^{(p-1)/\alpha} \bigg( \int_0^\tau \bar{Y}(s,x)^p ds\bigg)^{1/\alpha}, 
\end{equation*}
where we have also used $\tau \leq t$ and $\gamma(1-1/p) = (p-1)/\alpha$. We use the above on the last term in \eqref{e_probbdlemma_ineq2} and apply Jensen's inequality, now with concave function $y \to y^{1/\alpha}$, to obtain the following:
\begin{align} \label{e_probbdlemma_ineqchain}
\frac 1 r \int_{x_0+r}^{x_0+ 2r} \E(A_\tau(x)^\gamma) dx & \leq \frac{t^{(p-1)/\alpha}}{r} \int_{x_0+r}^{x_0+2r}\E \bigg[  \bigg( \int_0^{\tau} \bar{Y}(s,x)^{p} ds \bigg)^{1/\alpha} \bigg] dx \notag
\\ &=t^{(p-1)/\alpha} \E \bigg[ \frac 1 r \int_{x_0+r}^{x_0+2r} \bigg( \int_0^{\tau} \bar{Y}(s,x)^{p}ds \bigg)^{1/\alpha} dx \bigg] \notag
\\ &\leq t^{(p-1)/\alpha} \E \bigg[  \bigg(\frac 1 r \int_{x_0+r}^{x_0+2r}  \int_0^{\tau} \bar{Y}(s,x)^{p}ds dx \bigg)^{1/\alpha} \bigg] \notag
\\ &\leq t^{(p-1)/\alpha} \E \bigg[ \bigg( \frac{1}{r^{1+ \alpha}}  \int_{x_0+r}^{x_0+2r} \int_0^{\tau} (x-x_0)^{\alpha}  \bar{Y}(s,x)^{p} ds dx  \bigg)^{1/ \alpha} \bigg] \notag
\\ &\leq \frac{t^{(p-1)/\alpha}}{r^{1+1/\alpha}}\, \E (T(\tau, x_0 )^{1/\alpha}),
\end{align}
where we recall the definition of $T(\tau, x_0)$ from \eqref{e_probbdlemma_Tdef}. In the second-to-last inequality we have used the fact that $((x-x_0) / r)^\alpha \geq 1$ for $x \in [x_0 +r,x_0+2r]$. Combining these estimates with \eqref{e_probbdlemma_ineq2}, we obtain
\begin{equation} \label{e_probbdlemma_ineq3}
\frac 1 r \int_{x_0+r}^{x_0+ 2r} \bP(A_t(x) \geq a) dx \leq \bP(A_t(x_0) \geq b) + b^{\delta} + a^{-\gamma} \frac{t^{(p-1)/\alpha}}{r^{1 + 1/\alpha}} \E(T(\tau,x_0)^{1/\alpha}).
\end{equation}
To obtain the desired bound, we need to control $\E(T(\tau,x_0)^{1/\alpha})$. However, we note that $T(s,x_0)$ is precisely the time-change/inner-clock associated to the stochastic integral $M_s(x_0)$ in the sense of Propositions~\ref{prop_isometryint} and \ref{prop_integral_rep}. In view of Remark~\ref{remark_stop}, we can apply Proposition~\ref{prop_isometryint} at the stopping time $\tau$.  Hence, by Jensen's inequality and the lower bound in \eqref{e_prop_isometry},
\begin{align}  \label{e_problemmabd_bdg}
\E(T(\tau,x_0)^{1/\alpha}) \leq \E(T(\tau,x_0))^{1/\alpha} \leq \left( C_\alpha \sup_{\lambda > 0} \lambda^\alpha \bP (|M_\tau(x_0)|^* > \lambda)\right)^{1/\alpha}.
\end{align}
To bound the quantity above, we again use the representation of $M_s(x_0)$ as a time-changed stable process, that is, $M_s(x_0) = W_{T(s,x_0)}$ for $s \in [0,t]$, and we recall that first passage times $\hat{\tau}_1$ and $\hat{\tau}_2$ associated to $W$ introduced earlier in the proof. Then
\begin{align*}
|M_\tau(x_0)|^* = \sup_{s \in [0,\sigma_1 \wedge \sigma_2 \wedge t]} |M_s(x_0)| &= \sup_{u \in [0,\hat{\tau}_1 \wedge \hat{\tau}_2 \wedge T(t,x_0)]} |W_u| \leq  \sup_{u \in [0,\hat{\tau}_1 \wedge \hat{\tau}_2 ]} |W_u| = |W_{\hat{\tau}}|^*,
\end{align*}
where $\hat{\tau} = \hat{\tau_1} \wedge \hat{\tau}_2$. By the above and Lemma~\ref{lemma_stable_exit}(b), we have
\[\sup_{\lambda > 0} \lambda^\alpha \bP (|M_\tau(x_0)|^* > \lambda) \leq \sup_{\lambda > 0} \lambda^\alpha \bP (|W_{\hat{\tau}}|^* > \lambda) \leq 2^\alpha b^{\alpha(1-\delta)}.\]
Thus, returning to \eqref{e_problemmabd_bdg}, we obtain 
\begin{equation} \label{e_probbdlemma_Tbd}
\E(T(\tau,x_0)^{1/\alpha})\leq C b^{1-\delta}
\end{equation}
with $C = 2 C_\alpha  ^{1/\alpha}$. Substituting this into \eqref{e_probbdlemma_ineq3} yields 
\begin{equation*} 
\frac 1 r \int_{x_0+r}^{x_0+2r} \bP(A_t(x) \geq a) dx \leq \bP(A_t(x_0) \geq b) + b^{\delta} + C \frac{t^{(p-1)/\alpha}}{r^{1 + 1/\alpha}}\left( \frac{b^{1-\delta}}{ a^\gamma} \right).
\end{equation*}
This implies that there must exist $x_1 \in [x_0+ r,x_0+2r]$ such that $\bP(A_t(x_1) \geq a)$ satisfies the same inequality as the integral, and the proof is complete.

{\bf Case (b): $\gamma \in (0,1/\alpha)$.}  Let $K \geq 1$ and write
\begin{align*}
A_\tau(x) &= \bigg( \int_0^\tau \bar{Y}(s,x) 1_{\{\bar{Y}(s,x) \leq K\}} ds \bigg) + \bigg( \int_0^\tau \bar{Y}(s,x) 1_{\{\bar{Y}(s,x) > K\}} ds \bigg)
\\ & =: A^1_\tau(x) + A^2_\tau(x).
\end{align*}
If $A_\tau(x) \geq a$, then either $A^1_\tau(x) \geq a/2$ or $A^2_\tau(x) \geq a/2$. Hence from \eqref{e_probbdlemma_ineq11}, we have
\begin{align} \label{e_problemmabd_ineq5}
\bP(A_t(x) \geq a) &\leq \bP(A_t(x_0) \geq b) + b^\delta + \bP(A_\tau^1(x) \geq a/2) + \bP(A_t^2(x) \geq a/2),
\end{align}
where we have used $\tau \leq t$ and hence $A_\tau^2(x) \leq A_t^2(x)$. We first consider the term with $A^1_\tau(x)$. By Markov's inequality,
\begin{equation} \label{e_problemmabd_markovA1}
\bP(A_\tau^1(x) \geq a/2) \leq (a/2)^{-1/\alpha}\E(A_\tau^1(x)^{1/\alpha}).
\end{equation}
Next, we observe that since $p < 1$, if $u \in [0,K]$ then $u \leq u^p K^{1-p}$, and hence
\begin{align*}
A^1_\tau(x) =  \int_0^{\tau} \bar{Y}(s,x) 1_{\{\bar{Y}(s,x) \leq K \}} ds &\leq K^{1 - p} \int_0^\tau \bar{Y}(s,x)^p 1_{\{\bar{Y}(s,x) \leq K\}} ds 
\\ & \leq K^{1 - p} \int_0^\tau \bar{Y}(s,x)^p ds.
\end{align*}
From \eqref{e_problemmabd_markovA1} and the above, arguing as in \eqref{e_probbdlemma_ineqchain}, we have
\begin{align}
&\frac 1 r \int_{x_0+r}^{x_0+2r} \bP(A_\tau^1(x) \geq a/2) dx  \notag
\\ & \hspace{1 cm}\leq (a/2)^{-1/ \alpha} K^{(1-p)/\alpha} \frac 1 r \int_{x_0+r}^{x_0+2r} \E\bigg[ \bigg( \int_0^\tau \bar{Y}(s,x)^p ds \bigg)^{1/\alpha}  \bigg] dx \notag
\\ & \hspace{1 cm} \leq (a/2)^{-1/ \alpha}K^{(1-p)/\alpha}  \E \bigg[ \bigg(\frac 1 r \int_{x_0+r}^{x_0+2r} \int_0^\tau \bar{Y}(s,x)^p ds dx \bigg)^{1/\alpha}  \bigg] \notag
\\ & \hspace{1 cm} \leq (a/2)^{-1/ \alpha} K^{(1-p)/\alpha}  \E \bigg[ \bigg(\frac{1}{r^{1 +\alpha}} \int_{x_0+r}^{x_0+2r} \int_0^\tau (x-x_0)^\alpha \bar{Y}(s,x)^p ds  dx \bigg)^{1/\alpha}  \bigg] \notag
\\ &\hspace{1 cm} \leq (a/2)^{-1/ \alpha} K^{(1-p)/\alpha} r^{-1-1/\alpha} \E(T(\tau,x_0)^{1/\alpha}) \notag
\\ &\hspace{1 cm} \leq C a^{-1/ \alpha} K^{(1-p)/\alpha} r^{-1-1/\alpha} b^{1-\delta}. \label{e_problemmabd_At1intbd}
\end{align}
The final inequality uses \eqref{e_probbdlemma_Tbd} (which we note holds for the parameter regimes of part (a) and (b) of the lemma). Now we consider the term from \eqref{e_problemmabd_ineq5} with $A^2_t(x)$. Let $\theta \in (1,\alpha)$ and $\beta \in (1,\alpha)$. Then by Markov's and H\"{o}lder's inequalities,
\begin{align} \label{e_probbdlemma_At2bd1}
\bP (A_t^2(x) \geq a/2) &\leq 2a^{-1} \int_0^t \E(\bar{Y}(s,x) 1_{\{\bar{Y}(s,x) > K\}}) ds \notag
\\ & \leq  2a^{-1} \int_0^t \E(\bar{Y}(s,x)^\theta)^{1/\theta} \bP(\bar{Y}(s,x) > K)^{(\theta-1)/\theta} ds \notag
\\ &\leq 2a^{-1}  K^{-\beta(\theta - 1)/\theta}  \int_0^t \E(\bar{Y}(s,x)^\theta)^{1/\theta} \E(\bar{Y}(s,x)^{\beta})^{(\theta-1)/\theta}ds. 
\end{align}
By Theorem~\ref{thm_stochinteg}(b), in particular \eqref{e_thm_momentbd_plo}, there is some constant $C(t,\theta,\beta)\geq 1$ such that
\begin{align*}
&\E(\bar{Y}(s,x)^\theta)^{1/\theta} \E(\bar{Y}(s,x)^{\beta})^{(\theta-1)/\theta}
\\ &\hspace{1 cm} \leq C(t,\theta,\beta)
\bigg(s^{-(\alpha - 1)\frac 1 2 \frac \theta \alpha} [1 + sP_sY_0(x)]^{\theta/\alpha} + P_sY_0(x)^\theta\bigg)^{1/\theta} 
\\ &\hspace{1.5 cm} \times \bigg(s^{-(\alpha-1)\frac 1 2 \frac \beta \alpha}[1 + s P_sY_0(x)]^{\beta/\alpha} + P_sY_0(x)^\beta \bigg)^{(\theta-1)/\theta}
\\ &\hspace{1 cm} \leq C(t,\theta,\beta)
\bigg(s^{-(\alpha - 1)\frac 1 2 \frac 1 \alpha} [1 + sP_sY_0(x)]^{1/\alpha} + P_sY_0(x) \bigg)
\\ &\hspace{1.5 cm} \times \bigg(s^{-(\alpha-1)\frac 1 2 \frac \beta \alpha \frac{(\theta-1)}{\theta}}[1 + s P_sY_0(x)]^{\beta(\theta-1)/(\alpha \theta)} + P_sY_0(x)^{\beta(\theta-1)/\theta} \bigg).
\end{align*}
Let $m = \sup_{s > 0} p_s(1)< \infty$. We then have
\begin{align*}
\sup_{x \geq x_r + 1} \sup_{ s > 0} P_s Y_0(x) & \leq Y_0(1)  \sup_{s > 0} p_s(1) = m Y_0(1).
\end{align*}
We may conclude from the previous bound that for some enlarged constant $C'(t,\theta,\beta,Y_0(1))$, for all $s \in (0,t]$ and $x \geq x_r + 1$,
\begin{align*}
&\E(\bar{Y}(s,x)^\theta)^{1/\theta} \E(\bar{Y}(s,x)^{\beta})^{(\theta-1)/\theta}
\\ &\hspace{1 cm} \leq C'(t,\theta,\beta, Y_0(1))[1 + s^{-(\alpha-1)}].
\end{align*}
(To obtain the final power of $s$, we use the facts that $\theta, \beta < \alpha$ and $\theta^{-1}, (\theta-1)/\theta < 1$, as well as $s \leq t$.) We now return to \eqref{e_probbdlemma_At2bd1}. Noting that the upper bound for the moments obtained above is integrable over $s \in (0,t]$, it follows that for all $x \geq x_r+1$, 
\begin{equation*}
\bP(A^2_t(x) \geq a/2) \leq C''(t,\theta,\beta,Y_0(1)) a^{-1} K^{-\beta(\theta-1)/\theta}
\end{equation*}
where $C''(t,\theta,\beta,Y_0(1)) = 2C'(t,\theta,\beta,Y_0(1)) [t + \int_0^t s^{-(\alpha-1)}ds]$. Finally, we integrate \eqref{e_probbdlemma_ineq11} over $[x_0+r,x_0+2r]$ and substitute the above into \eqref{e_problemmabd_At1intbd} to obtain
\begin{align*}
\frac 1 r \int_{x_0 +r}^{x_0+2r} \bP(A_t(x) \geq a) dx &\leq \bP(A_t(x_0) \geq b) + b^\delta + C \frac{K^{(1-p)/\alpha} b^{1-\delta}}{r^{1+1/\alpha} a^{1/\alpha}} 
\\ &\hspace{.5 cm}+ C''(t,\theta,\beta,Y_0(1)) \frac{K^{-\beta(\theta-1)/\theta}}{a}.
\end{align*}
In the same way as in case (a), we may conclude that there exists $x_1 \in [x_0 + r, x_0 + 2r]$ which satisfies the desired bound, completing the proof. \end{proof}

Using Lemma~\ref{lemma_probbound} we can show the following result, from which we easily derive the main theorem.
\begin{proposition} \label{prop_limitprobAt} Suppose that $2-\alpha < \gamma < 1$. For every $t>0$, there exists a non-random sequence $(y_n(t))_{n \in \N}$ such that $\lim_{n \to \infty} y_n(t) = \infty$ and $\lim_{n \to \infty} \bP(A_t(y_n(t)) > 0) = 0$.
\end{proposition}

\begin{proof}
Fix $t>0$. We divide the proof into two cases: (i) $\gamma \in [1/\alpha,1)$ and (ii) $\gamma \in (2-\alpha,1/\alpha)$. We begin with case (i).  Fix $\delta \in (0, 1-\gamma)$. Let $\zeta \in (0,1)$ and, for $n \in \N \cup  \{0\}$, define $a_n = \zeta e^{-n}$ and $r_n = (n+1)^{-2}$. Let $x_0 \geq x_r + 1$. We define an increasing sequence $(x_n)_{n \in \N}$ by iteratively applying Lemma~\ref{lemma_probbound}(a). For $n \in \N$, given $x_{n-1}$, we apply Lemma~\ref{lemma_probbound}(a) with $a = a_n$, $b = a_{n-1}$ and $r= r_n$ to obtain that there exists $x_n \in [x_{n-1} + r_n,x_{n-1} + 2r_n]$ such that
\begin{align*}
&\bP(A_t(x_n) \geq a_{n}) 
\\ &\hspace{1 cm}\leq \bP(A_t(x_{n-1}) \geq a_{n-1}) + a_{n-1}^{\delta} + C \frac{t^{(p-1)/\alpha}}{r_n^{1+1/\alpha}} \bigg(\frac{a_{n-1}^{1-\delta}}{a_{n}^\gamma} \bigg) 
\\ &\hspace{1 cm}= \bP(A_t(x_{n-1}) \geq \zeta e^{-(n-1)}) + \zeta^\delta e^{-\delta (n-1)} + C t^{(p-1)/\alpha} (n+1)^{2 + 2/\alpha}\frac{(\zeta e^{-(n-1)})^{1-\delta}}{(\zeta e^{-n})^\gamma}.
\end{align*}
Applying this bound iteratively, it follows that
\begin{align*}
\bP(A_t(x_n) \geq \zeta e^{-n}) &\leq \bP(A_t(x_0) \geq \zeta) \\
&\hspace{.2 cm} + \zeta^{\delta} \sum_{k=1}^n e^{-\delta k + 1} + \zeta^{1-\gamma - \delta} C t^{\gamma - 1/\alpha} \sum_{k=1}^n (k+1)^{2+2/\alpha} e^{-(1-\gamma - \delta)k + 1}.
\end{align*}
As both sums are summable, it follows that for a sufficiently large value of $N = N(t)>0$ which is independent of $\zeta$ and $x_0$, for all $n \in \N$,
\begin{equation} \label{e_limitprop_ineq1}
\bP(A_t(x_n) \geq \zeta e^{-n}) \leq \bP(A_t(x_0) \geq \zeta) + N(\zeta^\delta + \zeta^{1-\gamma - \delta}).
\end{equation}
Since $x_n$ is increasing, $x_n - x_{n-1} \in [r_n,2r_n]$ and $r_n$ is summable, $x_n$ converges to some $y(x_0,\zeta) \in [x_0 + \rho, x_0 + 2\rho]$, where $\rho = \sum_{n=1}^\infty r_n$. The dependence on $\zeta$ is due to the fact that the sequence $(x_n)$ implicitly depends on $\zeta$. Let $\eps >0$. By Lemma~\ref{lemma_ibp_boundary}(c), there is a subsequence $(n_k)_{k \in \N}$ along which $A_t(x_{n_k}) \to A_t(y(x_0,\zeta))$ a.s., and hence
\begin{align} \label{e_cheeky_prob_estimate}
\bP(A_t(y(x_0,\zeta)) > \eps) &\leq \liminf_{k \to \infty} \bP(A_t(x_{n_k}) > \eps) \notag
\\ &\leq \bP(A_t(x_0) \geq \zeta) + N(\zeta^{\delta} + \zeta^{1-\gamma - \delta}).
\end{align}
The first inequality uses Fatou's Lemma, and the last inequality is from \eqref{e_limitprop_ineq1}. Taking $\eps \downarrow 0$, we obtain
\begin{equation} \label{e_limitprop_ineq11}
\bP(A_t(y(x_0,\zeta)) > 0) \leq \bP(A_t(x_0) \geq \zeta) + N(\zeta^\delta + \zeta^{1-\gamma-\delta}). 
\end{equation}
Thus, for every $x_0 \geq  x_r + 1$ and every $\zeta \in (0,1)$, there exists $y(x_0,\zeta)$ such that the above holds and $y(x_0,\zeta) - x_0 \in [\rho,2\rho]$. 

To complete the proof, we remark that by Markov's inequality and Lemma~\ref{lemma_At_momemt},
\begin{align*} 
\bP(A_t(x) \geq \zeta) \leq \zeta^{-1} \E(A_t(x)) & \leq \zeta^{-1} \int_0^t P_sY_0(x) ds
\\ &\leq \zeta^{-1} Y_0(1) \int_0^t p_s(x-x_r) ds.
\end{align*}
For $n \in \N$, we choose $\zeta_n = \frac{1}{2n}$. We may then choose $\hat{x}_n \in \R$ to be sufficiently large so that $\int_0^t p_s(\hat{x}_n-x_r) ds \leq n^{-2}/Y_0(1)$, in which case we have $\bP(A_t(\hat{x}_n) \geq \zeta_n) \leq 2/n$. Then \eqref{e_limitprop_ineq11} yields that
\begin{equation*}
\bP(A_t(y(\hat{x}_n,\zeta_n))>0) \leq \frac{2}{n} + 2N(n^{-\delta} + n^{-(1-\gamma - \delta)}) \to 0 \quad \text{as $n \to \infty$}.
\end{equation*}
In particular, $(y(\hat{x}_n,\zeta_n))_{n \in \N}$ is the desired sequence and the proof is complete in case (i).

We now give the proof in case (ii), when $\gamma \in (2-\alpha,1/\alpha)$. Let $\zeta \in (0,1)$ and define $a_n$ and $r_n$ as in the previous case, that is $a_n = \zeta e^{-n}$ and $r_n = (n+1)^{-2}$, $n \in \N \cup \{0\}$. A short calculation shows that when $2-\alpha < \gamma < 1/\alpha$, the inequality
\[\frac{1}{\alpha-1} < \frac{\alpha - 1}{1 - p}\]
holds, and hence the interval $(\frac{1}{\alpha-1}, \frac{\alpha - 1}{1 - p})$ is non-empty. Fix $q \in (\frac{1}{\alpha-1}, \frac{\alpha - 1}{1 - p})$, and then let $\beta \in (1 + 1/q, \alpha)$ and $\delta \in (0, 1 - 1/\alpha - q(1-p)/\alpha)$. We argue as in the previous case to iteratively define an increasing sequence $(x_n)_{n \in \N}$ satisfying certain estimates. Let $x_0 = x_r+1$. For $n \in \N$, given $x_{n-1}$, we apply Lemma~\ref{lemma_probbound}(b) with $a = a_n$, $b = a_{n-1}$, $K = K_n = a_n^{-q}$, $\beta$ and $\delta$ as above, and $\theta = \beta$, to obtain that there exists $x_n \in [x_{n-1} + r_n, x_{n-1} + 2r_n]$ such that
\begin{align*}
\bP(A_t(x_n) \geq \zeta e^{-n}) &\leq \bP(A_t(x_{n-1}) \geq \zeta e^{-n+1})  +  \zeta^{\delta} e^{-\delta (n-1)} 
\\ &\quad + C_1 (n+1)^{2 + 2/\alpha} \zeta^{-q(1-p)/\alpha} e^{n q(1-p)/\alpha} \bigg( \frac{\zeta^{1-\delta} e^{-(1-\delta)(n-1)}}{\zeta^{1/\alpha} e^{-n/\alpha}} \bigg)
\\ &\quad + C_2 \bigg(\frac{\zeta^{q(\beta - 1)}e^{-q(\beta - 1)n}}{\zeta e^{-n}}\bigg), 
\end{align*}
where the constant $C_1>0$ is universal and $C_2>0$ depends only on $t$, $q$, $\beta$, and $Y_0(1)$. As before, we apply the bound iteratively to obtain
\begin{align*}
\bP(A_t(x_n) \geq \zeta e^{-n}) &\leq \bP(A_t(x_0) \geq \zeta) + \zeta^\delta \sum_{k=1}^n e^{-\delta k + 1}  
\\ &\quad+ C_1 \zeta^{1 - q(1-p)/\alpha - 1/\alpha - \delta} \sum_{k=1}^n (k+1)^{2 + 2/\alpha}  e^{- (1 - 1/\alpha  - q(1-p)/\alpha - \delta)k+1} 
\\ & \quad +   C_2 \zeta^{q(\beta - 1) - 1} \sum_{k=1}^n e^{-(q(\beta - 1) - 1)k} .
\end{align*}
Because of choices of $q, \beta$ and $\delta$, the exponential rate in each term is negative, and hence each term is summable. Thus there is a constant $N= N(t)>0$ depending only on the parameters and $Y_0(1)$ and $t$ such that for all $n \in \N$,
\begin{equation*}
\bP(A_t(x_n) \geq \zeta e^{-n}) \leq \bP(A_t(x_0) \geq \zeta) + N(\zeta^{1 + q(1/\alpha - \gamma ) - 1/\alpha - \delta_2} +  \zeta^{\delta_2} + \zeta^{q(\alpha -1 - \delta_1) -1}). 
\end{equation*} 
Each power of $\zeta$ is positive, so the inequality above is analogous to \eqref{e_limitprop_ineq1}. The result now follows by the same argument used to prove the result in case (i) from \eqref{e_limitprop_ineq1}.
\end{proof}

\begin{proof}[Proof of Theorem~\ref{thm_onedim}] Fix $t>0$. By Proposition~\ref{prop_limitprobAt}, we may take a deterministic sequence $(y_n)_{n \in \N}$ such that $y_n > x_r$ for all $n$, $\lim_{n \to \infty} y_n = \infty$, and $\bP(A_t(y_n) > 0) \leq 2^{-n}$ for all $n$. This can be seen by taking a subsequence of the sequence from Proposition~\ref{prop_limitprobAt}. Applying Lemma~\ref{lemma_crosstheline} at each $y_n$, we obtain that
\begin{align*}
\bP \bigg( \bigcup_{n=1}^\infty \bigg\{A_t(y_n) = 0, \, Y_s((y_n,\infty)) > 0 \, \text{ for some } s \in [0,t] \bigg\} \bigg) = 0.
\end{align*}
Since $\sum_n \bP(A_t(y_n) > 0) < \infty$, by Borel-Cantelli, with probability one we have $A_t(y_n) = 0$ for some $n$. Thus, by the above, for $\omega$ outside of some $\bP$-null set there exists $n = n(\omega) \in \N$ such that $Y_s((y_n,\infty)) = 0$ for all $ s\in [0,t]$, which completes the proof.
\end{proof}

\section{Compact support for $d>1$} \label{s_highdim}
We now generalize the argument of the previous section to prove the compact support property in higher dimensions. The proof in dimensions $d>1$ goes along the same lines as the proof in one dimension. The main difference is that the occupation density at point $x$, $A_t(x)$, is replaced with the occupation density of the solution projected onto one of the coordinate axes, an object which we denote by $\A_t(z)$ for $z \in \R$. We prove an integration by parts formula involving $\A_t(z)$ analogous to \eqref{e_main_formula}, and ultimately a technical result akin to Lemma~\ref{lemma_probbound}, which we use to complete the proof in the same way.

We assume throughout the section that $(Y,L)$ is a weak solution to \eqref{e_spde1} with $\alpha \in (1,2)$, $d \in [2,\frac{2}{\alpha-1}) \cap \N$ and $\gamma \in (0,1)$, with initial data $Y_0 \in \cM_f(\R^d)$, defined on some probability space $(\Omega, \cF, (\cF_t)_{t \geq 0}, \bP)$. As in the previous section, we will use the density process $\bar{Y}(t,x)$ from Theorem~\ref{thm_stochinteg}.

Given $Y_t \in \cM_f(\R^d)$, we define a measure $\Y_t \in \cM_f(\R)$ by projecting $Y_t$ onto the first coordinate. We write $x = (x_1,\dots,x_d) \in \R^d$, and define $\Y_t \in \cM_f(\R)$ by
\begin{equation*}
\Y_t(A) := \int_{\R^d} 1_A(x_1) Y_t(dx),
\end{equation*}
for $A \subseteq \R$. Then $(\Y_t)_{t \geq 0} \in \mathbb{D}([0,\infty),\cM_f(\R))$ and is $\cF_t$-adapted. For $Y_0 \in \cM_f(\R^d)$, we write $\Y_0 \in \cM_f(\R)$ to denote its projection in the same way. 

The main result, which implies Theorem~\ref{thm_compact}(b), is the following, which is the higher dimensional analogue of Theorem~\ref{thm_onedim}. We remind the reader that, for $d>1$, we only prove the compact support property for $\gamma \in [1/\alpha,1)$.
\begin{theorem} \label{thm_highdim}
Suppose $1/\alpha \leq \gamma < 1$ and $Y_0 \in \cM_f(\R^d)$ has compact support. Then for $t>0$, with probability one there exists a random $z(t) \in \R$ such that $\Y_s((z(t),\infty)) ds = 0$ for all $s \in [0,t]$.
\end{theorem}

Given the above, one argues just as in the beginning of Section~\ref{s_dimone} that there a.s. exists a random $R_1>0$ such that $\Y_s([-R_1,R_1]^c) = 0$ for all $s \in [0,t]$. $\Y_t$ is defined by projecting $Y_t$ onto its first coordinate, but by rotational invariance of the equation we may define the projections onto each coordinate axis and similarly argue that there a.s. exists $R_2, \dots, R_d > 0$ such that the projections of $Y_s$ onto the $i$th coordinate put zero mass on $[-R_i, R_i]^c$ for all $s \in [0,t]$ almost surely. Then for $R = \max_{i=1,\dots,d} R_i$, we have $Y_s(([-R,R]^d)^c)=0$ for all $s \in [0,t]$, which proves Theorem~\ref{thm_compact}(b).

The proof method for Theorem~\ref{thm_highdim} is similar to the proof of Theorem~\ref{thm_onedim}. Many of the one-dimensional proof ingredients have analogous versions stated in terms of the projected process, and most of the martingale arguments from Section~\ref{s_dimone} do not depend on the spatial dimensions of the associated stochastic integrals. Thus, for several claims we do not give proofs, and just refer to the proofs of the analogous claims in the previous section. However, the higher-dimensional setting and the use of the projected process necessitate several substantial modifications, in particular in proving a result analogous to Lemma~\ref{lemma_probbound}. 

For the remainder of the section, we fix $Y_0 \in \cM_f(\R^d)$ with compact support. For $R>0$ we define the $d$-dimensional closed ball
\begin{equation*}
\Lambda_{R,d} =  \{x \in \R^d: |x| \leq R\}.
\end{equation*}
We define $R_0$ to be the radius of the smallest ball containing the support of $Y_0$, i.e.
\begin{equation*}
R_0 := \inf \{ R>0 : Y_0(\Lambda_{R,d}^c) = 0 \}.
\end{equation*}
For $z \in \R$ and $t>0$ we define
\begin{equation*} 
\bar{\Y}(t,z) := \int_{\R^{d-1}} \bar{Y}(t,(z,y))dy.
\end{equation*}
It is then immediate from Fubini's theorem that, $\bP$-a.s., $\Y_t(dz) = \bar{\Y}(t,z) dz$, that is, the above is a density for the projected measure almost surely. Analogously to the definition of $A_t(x)$ in the previous section, for $t>0$ and $z \in \R$ we define
\begin{equation*} 
\bar{\A}_t(z) := \int_0^t \bar{\Y}(s,z) ds. 
\end{equation*}
We observe that, just like $A_t(x)$ in the previous section, the process $t \to \bar{\A}_t(z)$ is a.s. increasing and continuous. If in addition we have $z > R_0$, provided the stochastic integral is well-defined, we define
\begin{equation}\label{def_AAt}
\A_t(z) := \int_{z}^\infty (w -z) \Y_t(dw) -  \int_{(0,t] \times \R^d} (x_1 - z)_+ \bar{Y}(s,x)^\gamma L(ds,dx).
\end{equation}
The equation defining $\A_t(z)$ is precisely the integration by parts formula for $\bar{\A}_t(z)$, i.e. the higher-dimensional analogue of \eqref{e_main_formula}, and we therefore expect that $\bar{\A}_t(z) = \A_t(z)$. Although we will indeed prove that $\A_t(z)$ is well-defined for all $z > R_0$, due to a technical issue we are only able to prove that it is equal to $\bar{\A}_t(z)$ almost everywhere. That is, we show that for a.e. $z > R_0$, $\bP(\bar{\A}_t(z) = \A_t(z)  \text{ for all }  t \geq 0) = 1$, which is sufficient for our purposes. 

We will continue to write $(P_t)_{t \geq 0}$ to denote the $d$-dimensional heat semigroup and we write $(\hat{P}_t)_{t \geq 0}$ to denote the one-dimensional heat semigroup. Similarly, we write $\hat{p}_t(\cdot)$ to denote the one-dimensional heat kernel. The next lemma gives first moment bounds for $\bar{\Y}(t,z)$ and $\bar{\A}_t(z)$.

\begin{lemma} \label{lemma_YYt_moment}
For all $t>0$ and $z \in \R$,
\begin{equation*}
\E(\bar{\Y}(t,z)) \leq \hat{P}_t \Y_0 (z)
\end{equation*}
and
\begin{equation*}
\E(\bar{\A}_t(z)) \leq \int_0^t \hat{P}_s \Y_0 (z) ds.
\end{equation*}
\end{lemma}
\begin{proof}
The first inequality is immediate from the definition of $\bar{\Y}(t,z)$ and \eqref{e_thm_meanmeasure}, and the second follows from the definition of $\bar{\A}_t(z)$ and the first inequality.
\end{proof}

For the remainder of the section we fix a mollifier $\xi \in C^\infty_c(\R)$ which is non-negative and satisfies $\xi \geq 0$ and $\int \xi = 1$. The set $U$ from the next lemma depends on $\xi$ and so further results involving this set implicitly depend on $\xi$ through $U$. The lemma is analogous to Lemma~\ref{lemma_At_convergence}. However, the argument used to prove $\LL^1$-convergence of $A^\eps_t(x)$ to $A_t(x)$ in Lemma~\ref{lemma_At_convergence}(a) does not work here, so the proof and statement are somewhat changed. 

\begin{lemma} \label{lemma_d_At_convergence}
There exists a set $U = U(\xi) \subseteq \R$ such that $U^c$ is Lebesgue-null and for all $z \in U$, $\bP(\lim_{\eps \downarrow 0} \bar{\A}^\eps_t(z) = \bar{\A}_t(z) \text{ for all }t \geq 0) = 1$, where
\[ \bar{\A}^\eps_t(z):= \int_0^t \xi_\eps * \Y_s(z) ds .\]
\end{lemma}

\begin{proof}
We begin by proving that for a fixed time $t>0$, there exists a set $U_t$ such that $U_t^c$ is Lebesgue-null, and for $z \in U_t$, $\lim_{\eps \downarrow 0} \A^\eps_t(z) = \A_t(z)$ a.s. We then upgrade the result to hold for all times simultaneously using the fact that $t \to \bar{\A}_t(z)$ is continuous and non-decreasing.

Let $t>0$. We first remark that $\bar{\A}_t \in \LL^1(\R)$ a.s. This can be seen by noting that, by Fubini's theorem, $\bar{\A}_t$ is a density for $\int_0^t \Y_s(\cdot)ds$, which a finite measure. Hence, by Fubini's theorem and a standard convolution result (see \cite[Theorem~8.14]{Folland}), with probability one we have
\begin{align*}
\bar{\A}_t^\eps(z) = \xi_\eps * \bigg( \int_0^t  \Y_s(\cdot) ds\bigg)(z) \to \bar{\A}_t(z) \text{ for a.e. $z \in \R$ as $\eps \downarrow 0$.}
\end{align*} 
We define the deterministic set $B_t \subset \R$ by
\[ B_t := \left\{ z \in \R : \bP ( \bar{\A}_t^\eps(z) \text{  does not converge to } \bar{\A}_t(z) \text{ as $\eps \downarrow 0$} ) > 0 \right\}.\]
It is elementary to argue that $\{(t,z,\omega) \in \R_+ \times \R \times \Omega : \bar{\A}_t^\eps(z) \text{ does not converge to } \bar{\A}_t(z) \}$ is jointly measurable, so Fubini's theorem and the property shown above allow us to conclude that $B_t$ is Lebesgue-null. Taking $U_t = B_t^c$ gives the set $U_t$ with the desired property.

Let $\mathcal{I} \subset \R_+$ be countable and dense, and define $B = \cup_{t \in \mathcal{I}} B_t$. Then $B$ is Lebesgue-null, and hence $U:=B^c = \cap_{t \in \mathcal{I}}U_t$ is a full-measure subset of $\R$ such that for every $z \in U$, $\bar{\A}^\eps_s(z) \to \bar{\A}_s(z)$ a.s. for every $s \in \mathcal{I}$. Now fix $z \in U$. Since $\mathcal{I}$ is countable, there is an event with probability one on which $\bar{\A}^\eps_s(z) \to \bar{\A}_s(z)$ as $\eps \to 0$ for all $s \in \mathcal{I}$. One can now argue in the same way as in the proof of Lemma~\ref{lemma_At_convergence}(b), using the fact that $t \to \bar{\A}_t(z)$ is continuous and non-decreasing, to prove that on this event we have $\lim_{\eps \downarrow 0} \bar{\A}^\eps_t(z)= \bar{\A}_t(z)$ for all $t \geq 0$. This completes the proof. \end{proof}

\begin{lemma} \label{lemma_ibp_d_Abar}
Let $z_0 \in U \cap (R_0,\infty)$. Then the function $\phi(s,x) = (x_1 - z_0)_+ \bar{Y}(s,x)^\gamma$ is in $\mathbb{L}^{\alpha,t}_{\text{a.s.}}$ for all $t>0$, and hence the stochastic integral in \eqref{def_AAt} exists for all $t \geq 0$. Moreover, we have $\bP(\bar{\A}_t(z_0) = \A_t(z_0) \text{ for all } t \geq 0) = 1$. In other words, for $z_0 \in (R_0,\infty) \cap U$, with probability one, 
\begin{equation} \label{e_Abarlemma_ibp}
\bar{\A}_t(z_0) = \int_{z_0}^\infty (z - z_0) \Y_t(dz) -  \int_{(0,t] \times \R^d} (x_1 - z_0)_+ \bar{Y}(s,x)^\gamma L(ds,dx), \quad t \geq 0.
\end{equation}
\end{lemma}

This establishes, at least for all $z_0 > R_0$ in the full-measure set $U$, the higher dimensional analogue of Lemma~\ref{lemma_ibp_boundary}(a)-(b). Its proof is nearly identical to the proofs of those results (and of the lemmas which precede them). We therefore omit the full proof and just sketch the few necessary changes. The first step in the proof is to establish that
\begin{equation} \label{e_d_integ}
\sup_{s \in [0,t]} \int_{z_0}^\infty (z-z_0) \Y_s(dz) < \infty  \, \quad\text{a.s.}
\end{equation}
This is proved along the same lines as Lemma~\ref{lemma_integ} and requires only one additional argument. One takes $\psi_n$ and $\eta$ as in the proof of Lemma~\ref{lemma_integ}, which are respectively smooth approximations of $1_{[-n,n]}$ and $1_{[1,\infty)}$. Then, for $m, n\in \N$, one applies \eqref{e_spde1weak} with $\phi(x) = (x_1-z_0) \psi_n(x_1-z_0) \eta(x_1-z_0) \prod_{i=2}^d \psi_m(x_i)$, which is in $C^\infty_c(\R^d)$. One then takes $m \to \infty$ so that $\prod_{i=2}^d \psi_m(x_i)$ converges to $1$, and after a short argument one obtains (relabelling $x_1$ as $z$)
\begin{align} \label{e_d_integ_aux}
&\int_{z_0 + \frac 1 4}^\infty (z-z_0) \eta(z-z_0) \psi_n(z-z_0) \Y_t(dz) = \int_{(0,t] \times \R}[(z-z_0) \eta(z-z_0 )\psi_n(z-z_0)]'' \Y_s(dz) ds \notag 
\\&\hspace{1.5 cm} + \int_{(0,t] \times \R^d} (x_1-z_0)_+ \eta(x_1-z_0) \psi_n(x_1-z_0) \bar{Y}(s,x)^\gamma L(ds,dx).
\end{align}
The limiting argument required to show this is straightforward because for fixed $n$, the integrands are monotone and uniformly bounded in $m$. We therefore omit the details. Given \eqref{e_d_integ_aux}, the proof of \eqref{e_d_integ} is identical to the proof of Lemma~\ref{lemma_integ} starting from \eqref{e_intlemma_ibp1}.

Using \eqref{e_d_integ}, the rest of the proof is likewise the same as the other arguments in Section~\ref{s_dimone}, in particular the proofs of Lemma~\ref{lemma_pre_ibp_boundary} and Lemma~\ref{lemma_ibp_boundary}(a)-(b), with two caveats. The first is that there is another simple limiting argument akin to the one used to show \eqref{e_d_integ_aux}. The second is that the $\bar{\A}_t(z)$ term in the integration by parts formula arises as the limit of $\bar{\A}_t^\eps(z)$, which is defined in Lemma~\ref{lemma_d_At_convergence} via the mollifer $\xi_\eps$. In particular, we pass to the limit using Lemma~\ref{lemma_d_At_convergence}, which is where the restriction $z_0 \in U$ arises. Making these small adjustments, the proof then follows the corresponding arguments in Section~\ref{s_dimone}.

Even though Lemma~\ref{lemma_ibp_d_Abar} shows that $\bar{\A}_t(z_0)$ satisfies \eqref{e_Abarlemma_ibp} all $z_0 \in (R_0,\infty) \cap U$, which has full Lebesgue measure in $(R_0,\infty)$, we will not always be able to choose $z_0 \in U$ in later arguments. Thus, we now establish a few properties of $\A_t(z_0)$, which we recall is {\it defined} by the integration by parts formula \eqref{def_AAt}. First, we show it is a well-defined process for all $z_0 > R_0$ (instead of when restricted to $z_0 \in U$ as in the previous lemma) and $t>0$. For $z_0 \not \in U$, because we do not necessarily have $\A_t(z_0) = \bar{\A}_t(z_0)$, $t \to \A_t(z_0)$ is not a priori continuous and non-decreasing. We show that these properties do in fact hold, and that $\A_t(z)$ satisfies a property akin to the one in Lemma~\ref{lemma_crosstheline}.

\begin{lemma} \label{lemma_d_ibp} Let $z_0 > R_0$.

(a) The stochastic integral in \eqref{def_AAt} exists for all $t>0$, and hence $(\A_t(z_0))_{t \geq 0}$ is well-defined.

(b) If $(z_n)_{n \in \N} \subset (R_0,\infty)$ and $\lim_{n\to\infty} z_n = z_0$, then for every $t>0$, $\A_s(z_n)$ converges to $\A_s(z_0)$ uniformly on $s \in [0,t]$ in probability.

(c) $(\A_t(z_0))_{t \geq 0}$ is a.s. continuous and non-decreasing.

(d) For any a.s. bounded stopping time $\tau$, with probability one, 
\[ \A_{\tau}(z_0) = 0 \Rightarrow \Y_t((z_0,\infty)) = 0 \,\text{ for all } t \leq \tau.\]
\end{lemma}

\begin{proof}
From Lemma~\ref{lemma_ibp_d_Abar}, for every $z_0 \in (R_0,\infty)\cap U$, $(x_1 - z_0)_+ \bar{Y}(s,x)^\gamma \in \LL^{\alpha,t}_{\text{a.s.}}$ for all $t>0$. To see that we can extend this to every $z_0 > R_0$ is simple. Let $z_0 > R_0$. Then since $U$ is dense in $\R$, there exists $w_0 \in (R_0, z_0) \cap U$, and \[(x_1 - z_0)_+ \bar{Y}(s,x)^\gamma  \leq (x_1 - w_0)_+\bar{Y}(s,x)^\gamma \, \text{ for all }(s,x) \in \R_+ \times \R^d.\] Since both integrands are non-negative and the integrand on the right is in $\LL^{\alpha,t}_{\text{a.s.}}$ for all $t>0$, so is the integrand on the left. This proves (a).

We now prove (b). Suppose that $(z_n)_{n \in \N}$ and $z_0$ are as in the statement. Arguing as in the proof of Lemma~\ref{lemma_ibp_boundary}(c), we obtain that for any $t>0$,
\begin{align*}
|\A_t(z_n) - \A_t(z_0)|^* \leq |z_n - z_0| V_t^* + |(F_n \cdot L)_t|^*
\end{align*}
where we recall that $V_t = Y_t(1) = \Y_t(1)$, and $F_n(s,x) := [(x_1-z_0)_+ - (x_1-z_n)_+]\bar{Y}(s,x)^\gamma$. Since $F_n$ vanishes pointwise as $z_n \to z_0$ and $|F_n(s,x)| \leq |z_0 - z_n|\bar{Y}(s,x)^\gamma \in \LL^{\alpha,t}_{\text{a.s.}}$, by Lemma~\ref{lemma_dctbdg}, $|(F_n \cdot L)_t|^*$ vanishes in probability as $n \to \infty$. Thus, by the above, $|\A_t(z_n) - \A_t(z_0)|^*$ vanishes in probability as $n \to \infty$, which proves (b).

To see part (c), we remark that if $z_0 \in U$, then $(\A_t(z_0))_{t\geq 0} = (\bar{\A}_t(z_0))_{t \geq 0}$ a.s. by Lemma~\ref{lemma_ibp_d_Abar}, and the latter is continuous and non-decreasing by definition, so the claim holds. If $z_0 \not \in U$, then let $(z_n)_{n \in \N} \subset U$ be a sequence converging to $z_0$. Then $(\A_t(z_n))_{t\geq 0} = (\bar{\A}_t(z_n))_{t \geq 0}$ for every $n$, and so for every $t>0$, by the first claim of part (b) we may restrict to a subsequence $(z_n')_{n \in \N}$ such that with probability one, $\bar{\A}_s(z_n) \to \A_s(z_0)$ uniformly on $[0,t]$. Hence $(\A_s(z_0))_{s \in [0,t]}$ is a.s. the uniform limit of continuous, non-decreasing functions and so is itself continuous and non-decreasing on $[0,t]$. This holds for every $t>0$, so the claim follows.

The proof of part (d) is the same as the proof of Lemma~\ref{lemma_crosstheline}. Indeed, all that is required for the proof is that $t \to \A_t(z_0)$ is continuous and non-decreasing, which hold by part (c), and satisfies \eqref{def_AAt}, which holds by definition. This completes the proof. \end{proof}

We have now established a $d$-dimensional version of all the main lemmas from Section~\ref{s_dimone} except for Lemma~\ref{lemma_probbound}. The following is the higher-dimensional version of Lemma \ref{lemma_probbound}, but there are several differences necessitated by working with the projected process. Moreover, the sub-optimality of our $q$th moment estimates for $\bar{Y}(t,x)$ for $q \in (1,\alpha)$ mean that we are not able to prove a useful version of the lemma when $\gamma < 1/\alpha$ in the higher-dimensional setting.
\begin{lemma} \label{lemma_probbound_highd}
Let $\gamma \in [1/\alpha,1)$ and $t>0$. There is constant $C = C(t,d) > 0$ such that the following holds: for $z_0 \geq R_0 + 1$, and any $\delta \in (0,1-\gamma)$, $a, b, r \in (0,1)$ and $R \geq 2R_0 \vee 2\sqrt{2(d-1)}$, there exists $z_1 \in [z_0 + r, z_0 + 2r]$ which satisfies
\begin{equation*} 
\bP (\A_t(z_1) \geq a) \leq \bP(\A_t(z_0) \geq b) + b^{\delta} + C \frac{R^{(d-1)(p-1)/\alpha}}{r^{1 + 1/\alpha}}\left( \frac{b^{1-\delta}}{ a^\gamma} \right) + C Y_0(1) \frac{\exp\left( - \frac{R^2}{16(d-1)t}\right)}{a}.
\end{equation*}
\end{lemma}
\begin{proof}
Let $t>0$ and $z_0 \geq R_0 + 1$. For $s \in [0,t]$ we define the processes
\begin{align}
M_s(z_0) &:= \int_{(0,s] \times \R^d} (x_1 - z_0)_+ \bar{Y}(u,x)^\gamma L(du,dx), \label{e_prob_d_bound_M} \\
T(s,z_0) &:= \int_{(0,s] \times \R^d} (x_1-z_0)_+^\alpha \bar{Y}(u,x)^{p} du dx. \label{e_prob_d_bound_T} 
\end{align}
Then by definition of $\A_t(z)$,
\begin{equation*}
\int_{z_0}^\infty (z-z_0) \Y_t(dz) = \A_t(z_0) + M_t(z_0), \quad t \geq 0.
\end{equation*}
Now let $\delta \in (0,1-\gamma)$ and $a,b \in (0,1)$. Following the proof of Lemma~\ref{lemma_probbound}, we introduce the stopping times
\begin{align}
\sigma_1 &:= \inf \{s > 0: \A_s(z_0) \geq b\}, \notag
\\ \sigma_2 &:= \inf \{s > 0 : M_s(z_0) \geq b^{1-\delta}\}, \notag
\\ \tau &:= t \wedge \sigma_1 \wedge \sigma_2. \notag
\end{align}
Just as in the proof of Lemma~\ref{lemma_probbound}, we use Proposition~\ref{prop_integral_rep} and Lemma~\ref{lemma_stable_exit} to show that
\begin{equation*}
\bP(\sigma_2 < \sigma_1 , \sigma_2 < t) \leq b^\delta.
\end{equation*}
Let $z > z_0$. Proceeding as in the proof of Lemma~\ref{lemma_probbound}, we derive the following:
\begin{align} \label{e_probbdlemma_d_ineq1}
&\bP(\A_t(z) \geq a) \notag
\\ &\hspace{1 cm}\leq \bP(\A_t(z) \geq a, \tau = \sigma_1 ) + \bP( \A_t(z) \geq a, \tau = t) + \bP(\A_t(z) \geq a, \tau = \sigma_2 < \sigma_1 \wedge t) \notag
\\ &\hspace{1 cm}\leq \bP(\A_t(z_0) \geq b) + \bP(\A_t(z) \geq a, \tau = t) + \bP(\sigma_2 < \sigma_1 ,\sigma_2 < t) \notag
\\ & \hspace{1 cm}\leq \bP(\A_t(z_0) \geq b) + \bP(\A_t(z) \geq a, \tau = t) + b^\delta .
\end{align}
Let $R \geq 2R_0 \vee 2\sqrt{2(d-1)}$ and recall $\Lambda_{R,d-1}$ is the $(d-1)$-dimensional closed ball of radius $R$. We decompose $\A_t(z)$ into two terms corresponding to the contributions from $\bar{Y}(s,(z,y))$ with $y \in \Lambda_{R,d-1}$ and $y \in \Lambda_{R,d-1}^c$. We write
\begin{align} \label{def_AR12}
\A_t(z) &= \int_0^t ds \int_{\R^{d-1}} \bar{Y}(s,(z,y)) dy  \notag
\\ &= \int_0^t ds \int_{\Lambda_{R,d-1}} \bar{Y}(s,(z,y)) dy + \int_0^t ds \int_{\Lambda_{R,d-1}^c} \bar{Y}(s,(z,y)) dy \notag
\\ &=: \A^1_t(z,R) + \A^2_t(z,R).  
\end{align}
We remark that we have implicitly used $\A_t(z) = \bar{\A}_t(z)$ in the first equality above, and that Lemma~\ref{lemma_ibp_d_Abar} only proves this equality for $z \in U$. However, $U$ has full Lebesgue measure and we will ultimately integrate with respect to $dz$, and hence we may ignore the potential null set on which this equality fails and proceed as above. If $\A_t(z) \geq a$, then either $\A_t^1(z,R) \geq a/2$ or $\A_t^2(z,R) \geq a/2$, and so from \eqref{e_probbdlemma_d_ineq1} and \eqref{def_AR12} we obtain
\begin{align}\label{e_probbdlemma_d_ineq2}
\bP(\A_t(z) \geq a) &\leq \bP(\A_t(z_0) \geq b) + b^\delta + \bP(\A_t^1(z,R) \geq a/2, \tau = t) + \bP(\A_t^2(z,R) \geq a/2, \tau = t) \notag
\\&\leq \bP(\A_t(z_0) \geq b) + b^\delta + \bP(\A_\tau^1(z,R) \geq a/2) +\bP(\A_t^2(z,R) \geq a/2).
\end{align}
Consider $\A_\tau^1(z,R)$. Let $\theta_\kappa >0$ denote the volume of the $\kappa$-dimensional unit ball. Since $p \geq 1$, we obtain from Jensen's inequality that
\begin{align*}
\A_\tau^1(z,R)& =  \theta_{d-1}  R^{d-1} \tau \bigg( \frac{1}{\tau \theta_{d-1} R^{d-1}} \int_0^\tau \int_{\Lambda_{R,d-1}} \bar{Y}(s,(z,y))^{p/p} dy ds \bigg) 
\\ &\leq \theta_{d-1}  R^{d-1} \tau \bigg( \frac{1}{\tau \theta_{d-1}  R^{d-1} } \int_0^\tau \int_{\Lambda_{R,d-1}} \bar{Y}(s,(z,y))^p dy ds \bigg)^{1/p} 
\\ &= \theta_{d-1}^{(p-1)/p} R^{(d-1)(p-1)/p}  \tau^{(p-1)/p} \bigg( \int_0^\tau \int_{\Lambda_{R,d-1}} \bar{Y}(s,(z,y))^p dy ds \bigg)^{1/p} .
\end{align*}
Hence, by Markov's inequality and because $t \geq \tau$,
\begin{align*}
&\bP(\A_\tau^1(z,R) \geq a/2) \leq (a/2)^{-\gamma} \E(\A_\tau^1(z,R)^\gamma) 
\\ &\hspace{1 cm}\leq (a/2)^{-\gamma} \theta_{d-1}^{(p-1)/\alpha} t^{(p-1)/\alpha} (2R)^{(d-1)(p-1)/\alpha} \E \bigg( \bigg(\int_0^\tau \int_{\Lambda_{R,d-1}} \bar{Y}(s,(z,y))^p dy ds \bigg)^{1/\alpha} \bigg).
\end{align*}
Combining the constants into a single constant $C_0>0$, we integrate the above over $z \in [z_0 + r, z_0 + 2r]$ for $r \in (0,1]$ and argue as in \eqref{e_probbdlemma_ineqchain} to obtain the following:
\begin{align} \label{e_probbdlemma_d_ineqchain}
&\frac 1 r \int_{z_0+r}^{z_0+2r} \bP(\A_\tau^1(z,R) \geq a/2) dz  \notag
\\ &\hspace{1 cm} \leq C_0 \frac{(tR^{d-1})^{(p-1)/\alpha}}{a^\gamma} \E \bigg(\frac 1 r \int_{z_0 + r}^{z_0 + 2r} \bigg(\int_0^\tau \int_{\Lambda_{R,d-1}} \bar{Y}(s,(z,y))^p dy ds \bigg)^{1/\alpha} dz \bigg)  \notag
\\ &\hspace{1 cm} \leq C_0 \frac{(tR^{d-1})^{(p-1)/\alpha}}{a^\gamma} \E \bigg(\bigg( \frac 1 r \int_{z_0 + r}^{z_0 + 2r}  \int_0^\tau \int_{\Lambda_{R,d-1}} \bar{Y}(s,(z,y))^p dy ds dz \bigg)^{1/\alpha}  \bigg)  \notag 
\\ &\hspace{1 cm} \leq C_0 \frac{(tR^{d-1})^{(p-1)/\alpha}}{a^\gamma}\E \bigg(\bigg( \frac 1 r \int_{z_0 + r}^{z_0 + 2r} \frac{(z-z_0)_+^\alpha}{r^\alpha} \int_0^\tau \int_{\Lambda_{R,d-1}} \bar{Y}(s,(z,y))^p dy ds dz \bigg)^{1/\alpha}  \bigg)  \notag 
\\ &\hspace{1 cm}\leq C_0 \frac{(tR^{d-1})^{(p-1)/\alpha}}{a^\gamma r^{1 + 1/\alpha}} \E \bigg(\bigg( \int_{(0,\tau] \times \R^d} (x_1 - z_0)^\alpha_+ \bar{Y}(s,x)^p ds dx \bigg)^{1/\alpha} \bigg)  \notag 
\\ &\hspace{1 cm} = C_0 \frac{(tR^{d-1})^{(p-1)/\alpha}}{a^\gamma r^{1 + 1/\alpha}} \E(T(\tau,z_0)^{1/\alpha}). 
\end{align}
The second inequality uses Jensen's inequality; the fourth inequality follows from a change of variables and the fact that we may bound the integral of the non-negative function $(x_1 - z_0)^\alpha_+ \bar{Y}(s,x)^p$ over $[z_0 + r, z_0 + 2r] \times \Lambda_{R,d-1}$ by its integral over all of $\R^d$. The last line simply uses from the definition of $T(\tau,z_0)$. Recalling the definitions of $M_\tau(z_0)$ and $T(\tau,z_0)$ from \eqref{e_prob_d_bound_M} and \eqref{e_prob_d_bound_T}, by Jensen's inequality and Proposition~\ref{prop_isometryint} we have
\begin{equation*}
\E(T(\tau,z_0)^{1/\alpha}) \leq \E(T(\tau,z_0))^{1/\alpha} \leq \left( C_\alpha \sup_{\lambda > 0} \lambda^\alpha \bP(|M_\tau(z_0)|^* > \lambda) \right)^{1/\alpha}.
\end{equation*}
Since the stopping times $\sigma_1, \sigma_2$ and $\tau$ are defined in the same way as in the proof of Lemma~\ref{lemma_probbound}, the same argument using the representation of $M_s(z_0)$ as a time-changed $\alpha$-stable process and applying Lemma~\ref{lemma_stable_exit}(b) applies with no modifications required, and we conclude in the same fashion that $\E(T(\tau,z_0)^{1/\alpha}) \leq 2C_\alpha^{1/\alpha}b^{1-\delta}$. Thus, from the above and \eqref{e_probbdlemma_d_ineqchain}, we conclude that there is a universal constant $C_1 > 0$ such that
\begin{equation} \label{e_problemma_A1_d1}
\frac 1 r \int_{z_0+r}^{z_0+2r} \bP(\A_\tau^1(z,R) \geq a/2) dz \leq C_1 \frac{(tR^{d-1})^{(p-1)/\alpha}}{r^{1 + 1/\alpha}} \frac{b^{1-\delta}}{a^\gamma}.
\end{equation}
We now consider $\A^2_t(z_0,R)$, which we recall is defined in \eqref{def_AR12}. By Markov's inequality and \eqref{e_thm_meanmeasure},
\begin{align} \label{e_problemma_A2_d1}
\bP(\A_t^2(z,R) \geq a/2) &\leq 2a^{-1} \int_0^t ds \int_{\Lambda_{R,d-1}^c} \E(\bar{Y}(s,(z,y))) dy \notag
\\ &\leq 2a^{-1} \int_0^t ds \int_{\Lambda_{R,d-1}^c} P_s Y_0((z,y)) dy.
\end{align}
For the time being, we parametrize $x \in \R^d$ by $x = (w, \hat x)$, with $w \in \R$ and $\hat x = (\hat{x}_1, \dots, \hat{x}_{d-1}) \in \R^{d-1}$. Let $p_s^{(1)}(\cdot)$ and $p_s^{(d-1)}(\cdot)$ denote respectively the one-dimensional and $(d-1)$-dimensional heat kernels. Since $Y_0$ is supported on $\Lambda_{R_0,d}$,
\begin{align} \label{e_problemma_A2_d2}
P_s Y_0((z,y)) = \int_{\Lambda_{R_0,d}} p^{(1)}_s(z-w) p_s^{(d-1)}(y - \hat x) \, Y_0(dx).
\end{align}
For each $y \in \Lambda_{R,d-1}^c$ and $\hat{x} \in \Lambda_{R_0,d-1}$, \[|y - \hat{x}| \geq |y| - |\hat{x}| \geq |y| - R_0 \geq |y| / 2,\]
where the last inequality uses $R \geq 2 R_0$. We also observe that 
\[ c_1(t) :=\sup_{w \geq  1} \sup_{s \in (0,t]} p^{(1)}_s(w) = \sup_{s \in (0,t]} p^{(1)}_s(1) < \infty\]
and that we may apply the above estimate to $p_s^{(1)}(z-w)$ in \eqref{e_problemma_A2_d2} since $z \geq R_0 + 1$. Combining these estimates, it follows from \eqref{e_problemma_A2_d1} and \eqref{e_problemma_A2_d2} that
\begin{equation*}
\bP(\A_t^2(z,R) \geq a/2) \leq 2c_1(t) Y_0(1) a^{-1} \int_0^t ds \int_{\Lambda_{R,d-1}^c} p^{(d-1)}_s(y/2) dy.
\end{equation*}
For $\kappa \in \N$ let $(\xi^{\kappa}_s)_{s \geq 0}$ denote a $\kappa$-dimensional standard Brownian motion started from $0$. By a change of variables and Brownian scaling,
\begin{align*}
 \int_{\Lambda_{R,d-1}^c} p_s^{(d-1)}(y/2) dy &= 2^{d-1} \bP\left(|\xi^{(d-1)}_s| \geq \frac{R}{2\sqrt 2} \right )
\\ &\leq 2^{d-1}(d-1) \bP \left(|\xi^{(1)}_s| \geq \frac{R}{2\sqrt{2( d-1) }}\right)
\\ & \leq 2^{d-1}(d-1) \bP \left(|\xi^{(1)}_t| \geq \frac{R}{2\sqrt{2(d-1)}}\right)
 \\ &\leq C(d)t^{1/2} e^{-R^2/(16t(d-1))}.
\end{align*}
The second inequality can be easily derived from the reflection principle and the last line uses the following Gaussian tail estimate: 
\begin{align*}
\bP \left(|\xi^{(1)}_t| \geq \frac{R}{2\sqrt{2(d-1)}}\right) = \frac{2}{(2\pi t)^{1/2}} \int_{\frac{R}{2\sqrt{2(d-1)}}}^\infty e^{-x^2/2t}  dx &\leq \frac{2}{(2\pi t)^{1/2}} \int_{\frac{R}{2\sqrt{2(d-1)}}}^\infty  x e^{-x^2/2t} dx
\\ &= (2\pi t)^{-1/2} 2t e^{-R^2/(16t(d-1))},
\end{align*}
where the inequality uses the assumption that $R \geq 2\sqrt{2(d-1)}$. Finally, from the previous displays we conclude that there is a constant $C_2(t,d)>0$ such that for all $z \geq R_0 + 1$ and $R \geq 2R_0 \vee 2 \sqrt{2(d-1)}$,
\begin{equation*}
\bP(\A^2_t(z,R) \geq a/2) \leq C_2(t,d) Y_0(1) \frac{e^{-R^2/(16t(d-1))}}{a}.
\end{equation*}
We integrate \eqref{e_probbdlemma_d_ineq2} over $[z_0 + r, z_0 + 2r]$ and use \eqref{e_problemma_A1_d1} and the above to conclude that
\begin{align*}
&\frac 1 r \int_{z_0 + r}^{z_0 + 2r} \bP(\A_t(z) \geq a) dz 
\\ &\hspace{1 cm} \leq \bP(\A_t(z_0) \geq b) + b^\delta + C_1 \frac{(tR^{d-1})^{(p-1)/\alpha}}{r^{1+1/\alpha}} \frac{b^{1-\delta}}{a} + C_2(t,d) Y_0(1) \frac{e^{-R^2 /(16t(d-1))}}{a}.
\end{align*} 
This implies that there exists $z_1 \in [z_0 + r, z_0 + 2r]$ such that $\bP(\A_t(z_1) \geq a)$ is bounded above by the right-hand side, and the proof is complete.  \end{proof}

The following result, which we can now prove, implies Theorem~\ref{thm_highdim}. Its proof follows the same method as the proof of Proposition~\ref{prop_limitprobAt}.

\begin{proposition}\label{prop_limitprobAt_d} Let $\gamma \in [1/\alpha,1)$ and $t>0$. There exists a non-random sequence $(w_n(t))_{n \in \N}$ such that $\lim_{n \to \infty} w_n(t) = \infty$ and $\lim_{n \to \infty} \bP(\A_t(w_n(t)) > 0 ) = 0$. 
\end{proposition}

\begin{proof} 
Fix $t>0$. We choose $\delta_1$ and $\delta_2$ such that
\[0 < \delta_1 < 1- \gamma  \,\,\text{ and } \,\,0 < \delta_2 < \alpha \frac{(1-\gamma-\delta_1)}{(d-1)(p-1)}.\] We will define a sequence of points $(z_n)_{n \geq 1}$ satisfying estimates courtesy of Lemma~\ref{lemma_probbound_highd}. Let $z_0 \geq R_0 + 1$. For $\zeta >0$ and $n \in \N \cup\{0\}$, we define $a_n = \zeta e^{-n}$. Given $z_{n-1}$, we apply Lemma~\ref{lemma_probbound_highd} with $a = a_n$, $b = a_{n-1}$, $r = r_n := (n+1)^{-2}$, and $R_n = a_n^{-\delta_2}$, which gives the following: for a constant $C = C(t,d)>0$, for all $n \in \N$ there exists $z_n \in [z_{n-1} + r_n, z_{n-1} + 2r_n]$ which satisfies
\begin{align*}
\bP(\A_t(z_n) \geq \zeta e^{-n}) &\leq \bP(\A_t(z_{n-1}) \geq \zeta e^{-(n-1)}) + \zeta^{\delta_1} e^{-\delta_1(n-1)} 
\\ & \hspace{.5 cm} + C (1+n)^{2+2/\alpha} \zeta^{-\delta_2 (d-1)(p-1)/\alpha} e^{n \delta_2(d-1)(p-1)/\alpha} \frac{\zeta^{1-\delta_1} e^{-(1-\delta_1)(n-1)}}{ \zeta^{\gamma}e^{-\gamma n}}
\\& \hspace{.5 cm} + CY_0(1) \zeta^{-1} e^n \exp \left(-\frac{\zeta^{-2 \delta_2} e^{2 \delta_2 n}}{16(d-1)t} \right).
\end{align*}
To apply the lemma we assume that $\zeta > 0$ is sufficiently small so that $2R_0 \vee 2\sqrt{2(d-1)} \leq (e/\zeta)^{\delta_2} = R_1 \leq R_n$ for all $n \in \N$. Applying the bound recursively, we obtain that for all $n \in \N$,
\begin{align*}
&\bP(\A_t(z_n) \geq \zeta e^{-n} ) 
\\ &\hspace{1 cm}\leq \bP(\A_t(z_0) \geq \zeta) + \zeta^{\delta_1} \sum_{k=1}^n \left[  e^{-\delta_1(k-1)}  \right]
\\ &\hspace{1 cm} \quad + C\zeta^{1 - \gamma - \delta_1 -\delta_2 (d-1)(p-1)/\alpha} \sum_{k=1}^n \left[ (1+k )^{2+2/\alpha} e^{k \delta_2(d-1)(p-1)/\alpha} \frac{e^{-(1-\delta_1)(k-1)}}{e^{-\gamma k}} \right]
\\&\hspace{1 cm} \quad + CY_0(1)  \left[\sum_{k=1}^n  \zeta^{-1} e^k \exp \left(-\frac{\zeta^{-2 \delta_2} e^{2 \delta_2 k}}{16(d-1)t} \right) \right].
\end{align*}
The first two sums are summable by our choices of $\delta_1$ and $\delta_2$, and the powers of $\zeta$ multiplying these terms are positive for the same reason. The third term is clearly summable, as the double exponential decay dominates the exponential growth. Moreover, as $\zeta \downarrow 0$, the value of the infinite series converges to $0$. Thus, it follows from the above that for all $n \in \N$,
\begin{equation*}
\bP(\A_t(z_n) \geq \zeta e^{-n}) \leq \bP(\A_t(z_0) \geq \zeta) + \eps(\zeta),
\end{equation*}
where $\lim_{\zeta \downarrow 0} \eps(\zeta) = 0$. Recall that $(r_n)_{n \in \N}$ is summable and write $\rho = \sum_{n=1}^\infty r_n$. Since $(z_n)_{n \in \N}$ is increasing, we must have $z_n \uparrow w(z_0,\zeta)$ for some $w(z_0,\zeta) \in [z_0 + \rho, z_0 + 2\rho]$. Recalling that $z \to \A_t(z)$ is continuous in probability (Lemma~\ref{lemma_d_ibp}(b)), we may argue as in \eqref{e_cheeky_prob_estimate} to obtain that
\begin{equation} \label{e_d_z_w_bd}
\bP(\A_t(w(z_0,\zeta)) > 0) \leq \bP(\A_t(z_0) \geq \zeta) + \eps(\zeta).
\end{equation}
For any $z_0 \geq R_0 +1$ and sufficiently small $\zeta >0$, such a $w(z_0,\zeta)$ exists. The proof from here is the same as the proof of Proposition~\ref{prop_limitprobAt} from \eqref{e_limitprop_ineq1}, and follows by choosing suitable sequences $(\zeta_n)_{n \in \N}$ and $(\hat{z}_n)_{n \in \N}$ such that $\zeta_n$ vanishes and $\hat{z}_n \to \infty$, and using these values (for $z_0$ and $\zeta$) in \eqref{e_d_z_w_bd}. The only difference is that in order to bound $\bP(\A_t(\hat{z}_n) \geq \zeta_n)$ using Markov's inequality, we must assume that $\hat{z}_n \in U$, so that $\A_t(\hat{z}_n) = \bar{\A}_t(\hat{z}_n)$ by Lemma~\ref{lemma_ibp_d_Abar} and we can bound its expectation using Lemma~\ref{lemma_At_momemt}. However, we are at liberty to choose $\hat{z}_n$ to be in $U$, so no changes are required. This yields a sequence $(w(\hat{z}_n,\zeta_n))_{n \in \N}$ which has the desired property, and the proof is complete.\end{proof}

The proof of Theorem~\ref{thm_highdim} given Proposition~\ref{prop_limitprobAt_d} is identical to the proof of Theorem~\ref{thm_onedim} given Proposition~\ref{prop_limitprobAt}, and simply uses Lemma~\ref{lemma_d_ibp}(d). The proof of Theorem~\ref{thm_highdim} is therefore complete.

\section{The stochastic integral formula} \label{s_stochintegralrep}
The main purpose of this section is to prove Theorem~\ref{thm_stochinteg}, which, given a weak solution $(Y,L)$ to \eqref{e_spde1} with $\gamma \in (0,1)$, establishes the existence of a density process $\{\bar{Y}(t,x) : t>0, x \in \R^d\}$ which satisfies a stochastic integral formula and has other useful properties. We note that this section is self-contained, i.e. none of the results proved here depend on the results of Section~\ref{s_dimone} and \ref{s_highdim}. As discussed in Section~\ref{ss_prel}, we also avail ourselves of the arguments in this section to bridge a small gap between the weak solutions to \eqref{e_spde1} constructed in \cite{M2002} and the ones which we have defined in Definition~\ref{def_sol}. We begin the section by briefly describing this issue and presenting our strategy for resolving it and proving Theorem~\ref{thm_stochinteg} simultaneously.

The solutions in \cite{M2002} are constructed as the limit of solutions to a sequence of martingale problems approximating \eqref{e_spde1}. In particular, for each $n\in\N$, there is a martingale problem with a solution consisting of a measure-valued process $(Y^n_t)_{t \geq 0} \in \mathbb{D}([0,\infty),\cM_f(\R^d))$ and density process $\{Y^n(t,x) : t >0, x\in \R^d\}$, and in fact $Y^n(t,\cdot)$ is defined as the density of $Y^n_{t-}$. It is shown that these sequences have subsequential limits and that their limits satisfy \eqref{e_spde1weak} for every $\phi \in \cS$ for some stable noise $L(ds,dx)$. The reason the solutions constructed there do not precisely correspond to Definition~\ref{def_sol} is the nature of the convergence of the approximating density processes. While $(Y^n_t)_{t \geq 0}$ converges in distribution on $\mathbb{D}([0,\infty),\cM_f(\R^d))$, the densities do not weakly converge as processes on a function space on $\R^d$. Instead, they converge in distribution as functions of space-time: for every $t>0$, $(s,x) \to Y^n(s,x)$ is tight with respect to the weak topology on $\LL^{q,t} = \LL^q((0,t] \times \R^d)$  for $q \in (1,1+ 2/d)$ (see \cite[Corollary~4.6]{M2002}), and it is in this space that its subsequential limits are taken. Convergence (weak or in norm) in $\LL^{q,t}$ does not preserve properties of the pre-limit at fixed times. Thus, even though $\bP(Y^n(t,x)dx = Y^n_{t-}(dx)) = 1$ for all $t>0$, given the nature of the convergence of the density, the strongest conclusion we can make about the relationship of the limiting density and measure-valued process is \eqref{e_densprob2intro}. Consequently, the construction in \cite{M2002} leads to the following definition.

\begin{definition} \label{def_sol2} A pair $(Y,L)$ defined on a filtered probability space $(\Omega, \cF, (\cF_t)_{t \geq 0}, \bP)$ is a type-$2$ weak solution to \eqref{e_spde1} with initial state $Y_0 \in \cM_f(\R^d)$ if the following hold:
\begin{itemize}
\item $L$ is spectrally positive $\alpha$-stable $\cF_t$-martingale measure on $\R_+ \times \R^d$.
\item $\{Y(t,x) : t>0, x \in \R^d\}$ is non-negative, predictable and satisfies \eqref{assumption_integ}.
\item There exists an $\cF_t$-adapted measure-valued process $(Y_t)_{t \geq 0} \in \mathbb{D}([0,\infty), \cM_f(\R^d))$ which satisfies \eqref{e_densprob2intro}, and such that \eqref{e_spde1weak} holds for every $\phi \in \cS$. That is, for every $\phi \in \cS$, with probability one,
\[ \langle Y_t, \phi\rangle - \langle Y_0, \phi\rangle = \int_{(0,t]} \langle Y_s, \Delta \phi\rangle ds + \int_{(0,t] \times \R^d} Y(s,x)^\gamma \phi(x) L(ds,dx), \quad t\geq 0, \tag{\ref{e_spde1weak}}\]
and
\[ \bP \big(1_{(0,t]}(s)Y_s(dx)ds = 1_{(0,t]}(s)Y(s,x)dxds\big) = 1 \,\, \text{ for all } t > 0. \tag{\ref{e_densprob2intro}} \]
\end{itemize}
\end{definition}

\begin{remark} In \cite{M2002}, the density process is assumed to be progressively measurable, not predictable. However, it is straightforward to show that, if the density is a priori progressively measurable, one can always obtain a predictable version using \eqref{e_densprob2intro} and the fact that $(Y_t)_{t \geq 0} \in \mathbb{D}([0,\infty), \cM_f(\R^d))$, and hence we can make our definition using a predictable density process. \end{remark}

In Definition~\ref{def_sol2}, the measure-valued process is not a version of the one defined in terms of the density, so denoting solutions by a pair $(Y,L)$ is a slight abuse of notation, but this should not cause any confusion. For all of the ensuing discussion, we will view our original weak solutions, from Definition~\ref{def_sol}, from the perspective of Remark~\ref{remark_densmeas}, with a measure-valued process and a density related by \eqref{e_denssol}. Thus, both types of solutions consist of a noise $L$, a density $\{Y(t,x) : t>0, x\in \R^d\}$, and a process $(Y_t)_{t \geq 0} \in \mathbb{D}([0,\infty),\cM_f(\R^d))$, which satisfy the stochastic integration by parts formula \eqref{e_spde1weak} for all $\phi \in \cS$. The assumptions on these processes are identical except for one difference: weak solutions satisfy \eqref{e_denssol}, whereas type-$2$ weak solutions satisfy \eqref{e_densprob2intro}. We remind the reader that \eqref{e_denssol} is the condition that
\[ \tag{\ref{e_denssol}}
\bP(Y_t(dx) = Y(t,x)dx) = 1 \,\, \text{ for all } t > 0.\]
Since \eqref{e_denssol} implies \eqref{e_densprob2intro}, a weak solution is a type-$2$ weak solution.

Instead of Theorem~\ref{thm_stochinteg}, we will prove the following stronger result.

\begin{theorem} \label{thm_stochint2}
Suppose $\alpha \in (1,2)$, $d \in [1,\frac{2}{\alpha-1}) \cap \N$, $\gamma \in (0,1)$, and $Y_0 \in \cM_f(\R^d)$. Then if $(Y,L)$ is a type-$2$ weak solution to \eqref{e_spde1} with initial state $Y_0$, all of the conclusions of Theorem~\ref{thm_stochinteg} hold.
\end{theorem}

This has two corollaries. One of them is Theorem~\ref{thm_stochinteg}, because a weak solution is a type-$2$ weak solution. The other is the following.

\begin{corollary} Let $\alpha \in (1,2)$, $d \in [1,\frac{2}{\alpha-1}) \cap \N$, $\gamma \in (0,1)$, and $Y_0 \in \cM_f(\R^d)$. Then there exists a weak solution (in the sense of Definition~\ref{def_sol}) to \eqref{e_spde1} with initial state $Y_0$.
\end{corollary}

\begin{proof} Let $(Y,L)$ be a type-$2$ weak solution with measure-valued process $(Y_t)_{t \geq 0}$ and density $\{Y(t,x) : t>0, x \in \R^d\}$. Then by Theorem~\ref{thm_stochint2}, there exists a predictable random field $\{\bar{Y}(t,x) : t>0,x \in \R^d\}$ such that $\bP(\bar{Y}(t,x)dx = Y_t(dx)) = 1$ for all $t>0$, and $\bar{Y}(t,x) = Y(t,x)$ for a.e. $(t,x) \in \R_+ \times \R^d$ with probability one. The first condition is exactly \eqref{e_denssol}, while the second guarantees that \eqref{assumption_integ} and \eqref{e_spde1weak} are satisfied, the latter for all $\phi \in \cS$. Thus, in view of Remark~\ref{remark_densmeas}, the processes $(Y_t)_{t \geq 0}$ and $\{\bar{Y}(t,x) : t>0,x \in \R^d\}$ along with the noise $L$ are a weak solution in the sense of Definition~\ref{def_sol}.\end{proof}

As discussed in Section~\ref{ss_prel}, there are natural reasons for preferring Definition~\ref{def_sol} to Definition~\ref{def_sol2}. The upshot of the discussion above is that the use of Definition~\ref{def_sol} is justified if we can prove Theorem~\ref{thm_stochint2}. As we have remarked, this will also imply Theorem~\ref{thm_stochinteg}, the proof of which is the main goal of this section. Hence, we can solve both of these problems simultaneously by proving Theorem~\ref{thm_stochint2}, which we do in the remainder of the section. Finally, let us note that the proof does not change at all by working with a type-$2$ weak solution. That is, the argument under assumption \eqref{e_densprob2intro} is much the same as it would be if instead we assumed \eqref{e_denssol}, so our approach of addressing both problems at once by using the weaker definition does not complicate the proof of Theorem~\ref{thm_stochinteg} in any significant way. In the sequel, when we refer to a part of Theorem~\ref{thm_stochint2} (e.g. Theorem~\ref{thm_stochint2}(b)), it refers to the corresponding statement in Theorem~\ref{thm_stochinteg} (i.e. Theorem~\ref{thm_stochinteg}(b)), but with the assumptions of Theorem~\ref{thm_stochint2}.

The rest of the section contains the proof of Theorem~\ref{thm_stochint2}. Let $\alpha \in (1,2)$, $d \in [1, \frac{2}{\alpha-1})\cap \N$, and $\gamma \in (0,1)$. These assumption are in force throughout the section. Let $Y_0 \in \cM_f(\R^d)$ and let $(Y,L)$ be a type-$2$ weak solution of \eqref{e_spde1} with initial state $Y_0$. We remind the reader one last time that this simply means that \eqref{e_denssol} is replaced with the weaker assumption \eqref{e_densprob2intro}.

Let us first state a specialized version of the stochastic Fubini theorem. We define the stopping times $\tau_k$, $k \in \N$, by 
\begin{equation} \label{def_tau_k_Y}
\tau_k := \inf \bigg\{t > 0 : \int_{(0,t] \times \R^d} Y(s,x)^p dsdx > k \bigg\}. 
\end{equation}
By \eqref{assumption_integ}, $\lim_{k \to \infty} \tau_k = \infty$ almost surely. We prove a stochastic Fubini theorem for integrands which can be localized by $\tau_k$. 
\begin{lemma} \label{lemma_fubini2}
Let $t>0$, $(G, \mathcal{G}, \mu)$ be a finite measure space, and $\phi : \Omega  \times (0,t] \times \R^d \times G \to \R$ be jointly measurable with respect to $\cP \times \mathcal{G}$. Suppose that
\begin{equation} \label{e_fublemma2_1}
\E \bigg(\int_G \mu(dx) \bigg[ \int_{(0,t]\times \R^d} 1_{\{s \leq \tau_k\}}|\phi(s,y,x)|^\alpha Y(s,y)^p \,ds dy \bigg] \bigg) < \infty
\end{equation}
for all $k \in \N$. Then with probability one, we have
\begin{align*} 
&\int_G \mu(dx) \bigg[\int_{(0,t] \times \R^d} \phi(s,y,x) Y(s,y)^\gamma L(ds,dy) \bigg]  
\\ & \hspace{2 cm} =\int_{(0,t] \times \R^d} \left[ \int_G \phi(s,y,x) \mu(dx) \right] Y(s,y)^\gamma L(ds,dy). 
\end{align*}
\end{lemma}
\begin{proof}
If $\phi$ and $\mu$ satisfy \eqref{e_fublemma2_1}, then $\phi_k$, defined by $\phi_k(s,y,x) = 1_{\{s \leq \tau_k\}} \phi(s,y,x) Y(s,y)^\gamma$, and $\mu$ satisfy the assumptions of Lemma~\ref{lemma_fubini1}. Hence with probability one we have
\begin{align*}
&\int \mu(dx) \bigg[\int_{(0,t] \times \R^d} 1_{\{s \leq \tau_k\}}  \phi(s,y,x) Y(s,y)^\gamma L(ds,dy) \bigg]  
\\ & \hspace{1.2 cm} =\int_{(0,t] \times \R^d} \left[ \int 1_{\{s \leq \tau_k\}} \phi(s,y,x)  Y(s,y)^\gamma \mu(dx) \right] L(ds,dy)
\\ & \hspace{1.2 cm} =\int_{(0,t] \times \R^d} \left[ \int 1_{\{s \leq \tau_k\}} \phi(s,y,x)  \mu(dx) \right] Y(s,y)^\gamma  L(ds,dy). 
\end{align*}
This holds for all $k \in \N$. Since $\lim_{k \to \infty} \bP(\tau_k < t) = 0$, the claim follows.
\end{proof}

We now prove Theorem~\ref{thm_stochint2} through a series of lemmas. We first use some integrability arguments to prove $\bar{Y}(t,x)$ is well-defined almost everywhere, and that $Y_t(dx) = \bar{Y}(t,x)dx$ a.s. for almost every $t>0$. We then prove several technical lemmas which we use to establish moment estimates. Finally, using the moment estimates, we can upgrade the existence of $\bar{Y}(t,x)$ from ``almost everywhere" to ``everywhere". We then refine the moment estimates and prove the remaining claims from Theorem~\ref{thm_stochint2}.

Recall that $(P_t)_{t \geq 0}$ denotes the heat semigroup and $p_t(\cdot)$ the associated heat kernel. We write
\begin{equation*}
Z(t,x) := \int_{(0,t] \times \R^d} p_{t-s}(x-y) Y(s,y)^\gamma L(ds,dy)
\end{equation*}
whenever the stochastic integral is well-defined. $Z(t,x)$ is the stochastic integral appearing in the definition of $\bar{Y}(t,x)$. The first step is to show that it is well-defined for a.e. $(t,x) \in \R_+ \times \R^d$. We define
\begin{equation*}
H_t(x) := \int_{(0,t] \times \R^d} p_{t-s}(x-y)^\alpha  Y(s,y)^{p} \,dsdy. 
\end{equation*}
From Section~\ref{s_stochcalc}, $Z(t,x)$ can be defined if $H_t(x) < \infty$ $\bP$-a.s. For $T>0$, define
\begin{equation*}
\HH_T := \int_{(0,T] \times \R^d} H_t(x) dtdx. 
\end{equation*}
We also need to consider related integrals along a sequence of stopping times. We recall the stopping times $\tau_k$ from \eqref{def_tau_k_Y}. For $k \in \N$, define
\begin{align*}
H_{t,k}(x) &= \int_{(0,t] \times \R^d} 1_{\{s \leq \tau_k\}} p_{t-s}(x-y)^\alpha Y(s,y)^p dsdy,
\\ \HH_{T,k} &= \int_{(0,T] \times \R^d} H_{t,k}(x) dt dx.
\end{align*}
\begin{lemma} \label{lemma_Htinteg} (a) $\HH_T < \infty$ a.s. for all $T>0$. 

(b) For some constant $C>0$, $\E(\HH_{T,k}) < C T^{1-(\alpha-1)\frac d 2} k$ for all $k \in \N$ and all $T>0$.
\end{lemma}
\begin{proof}
For $t>0$, by Fubini's theorem we have
\begin{align*}
\int_{\R^d} H_t(x) dx &= \int_0^t ds \iint Y(s,y)^p p_{t-s}(x-y)^\alpha \, dxdy
\\ & \leq C \int_0^t ds (t-s)^{-(\alpha - 1)\frac d 2} \iint Y(s,y)^p p_{t-s}(x-y) \, dxdy 
\\& = C \int_0^t ds (t-s)^{-(\alpha - 1)\frac d 2} \int  Y(s,y)^p dy.
\end{align*}
Hence, for $T>0$,
\begin{align*}
\HH_T &\leq C \int_0^T dt \int_0^t ds (t-s)^{-(\alpha - 1)\frac d 2} \int  Y(s,y)^p dy
\\ &= C \int_0^T ds \int  Y(s,y)^p dy \int_s^T  (t-s)^{-(\alpha-1)\frac d 2} dt 
\\ &\leq C \int_0^T ds \,(T-s)^{1 - (\alpha-1)\frac d 2}  \int  Y(s,y)^p dy 
\\&= C T^{1-(\alpha-1)\frac d 2} \int_{(0,T] \times \R^d} Y(s,y)^p \, dsdy. 
\end{align*}
By \eqref{assumption_integ} the integral above is finite almost surely, which proves (a).

To see that (b) holds, one carries out the same computation as in the proof of (a) but carries along the indicator function $1_{\{s \leq \tau_k\}}$. One then obtains
\begin{equation*}
\E(\HH_{T,k}) \leq CT^{1-(\alpha-1)\frac d 2} \E \bigg( \int_{(0,T]\times \R^d} Y(s,x)^p 1_{\{s \leq \tau_k\}} dsdx \bigg) \leq CT^{1-(\alpha-1)\frac d 2}k,
\end{equation*}
where the last inequality holds by the definition of $\tau_k$.
\end{proof}

\label{lemma_aeintegalpha} 
\begin{lemma} \label{lemma_aeintegalpha} (a) There is a non-random set $\,\cU \subset \R_+ \times \R^d$ such that $\cU^c$ is Lebesgue-null and $H_t(x) < \infty$ a.s. for each $(t,x) \in \cU$. In particular, $Z(t,x)$ is well-defined for each $(t,x) \in \cU$. 

(b) There is a non-random set $I \subset \R_+$ such that $I^c$ is Lebesgue-null and for each $t \in I$,
\begin{equation*}
\E \bigg( \int H_{t,k}(x) dx \bigg) < \infty \quad \text{ for all $k \in \N$.}
\end{equation*}
\end{lemma}
\begin{proof}
First we prove (a). Since $\HH_T = \int_0^T \int H_t(x) dx dt$, Lemma~\ref{lemma_Htinteg}(a) implies that $H_t(x) < \infty$ for a.e. $(t,x) \in (0,T] \times \R^d$ a.s. for all $T>0$, hence for a.e. $(t,x) \in \R_+ \times \R^d$. By Fubini's theorem, we obtain
\begin{equation*}
\E \bigg( \int_{\R_+ \times \R^d} 1_{\{H_t(x) = \infty\}} dt dx \bigg) = \int_{\R_+ \times \R^d} \bP(H_t(x) = \infty) dt dx.
\end{equation*}
(This requires the joint measurability of $\{(t,x,\omega) \in \R_+ \times \R^d \times \Omega : H_t(x)(\omega) = \infty\}$ with respect to the product $\sigma$-algebra $\cB(\R_+ \times \R^d) \times \cF$, which follows from a standard argument.) The integral inside the expectation on the left-hand side equals zero a.s. by the previous observation, and hence both sides of the above equal zero. This implies that $B_0 := \left\{ (t,x) \in \R_+\ \times \R^d : \bP (H_t(x) = \infty) > 0 \right\}$ is Lebesgue null. Taking $\mathcal{U} = B_0^c$ completes the proof.

To prove (b), we first observe that it is sufficient to prove that there is a subset of full measure $I_T \subset (0,T]$ with the desired property for every $T>0$. Let $T>0$ and $k \in \N$. By Lemma~\ref{lemma_Htinteg}(b), Fubini's theorem and the definition of $\HH_{T,k}$, there is at most a Lebesgue-null subset $F_{T,k} \subseteq (0,T]$ of times $t$ such that if $t \in F_{T,k}$,
\[ \E \bigg( \int H_{t,k}(x) dx \bigg) = \infty.\]
Then the set $I_T := (0,T] \cap ( \cup_{k =1}^\infty F_{T,k})^c$ has full Lebesgue measure has the desired property.
\end{proof}

Whenever $Z(t,x)$ is well-defined, we write
\begin{equation} \label{e_Y_Z_formula}
\bar{Y}(t,x) = P_t Y_0(x) + Z(t,x).
\end{equation}
The so-called mild form of \eqref{e_spde1} is the following integral equation for $\langle Y_t, f \rangle$:
\begin{align} \label{e_mild_general}
\langle Y_t, f \rangle = \langle Y_0, P_t f \rangle + \int_{(0,t] \times \R^d} P_{t-s} f(x) Y(s,x)^\gamma L(ds,dx).
\end{align}
It is a standard argument to show, starting from \eqref{e_spde1weak}, that for every $f \in \cS$ and $t>0$, \eqref{e_mild_general} holds a.s. (For a proof in the stable case in an almost identical setting, see the proof of the first part of Proposition~2.2 in \cite{YZ2017}.) The stochastic integral representation for the density of $Y_t$ is obtained formally by taking $f = \delta_x$ in \eqref{e_mild_general}. To establish that $\bar{Y}(t,x)$ is a density for $Y_t$, we need only apply our stochastic Fubini theorem, Lemma~\ref{lemma_fubini2}, to the mild form \eqref{e_mild_general}. The set $I$ in the sequel is the same whose existence was asserted in the Lemma~\ref{lemma_aeintegalpha}(b).

\begin{lemma} \label{lemma_prelim_density} (a) For each $t \in I$, $\bP(Y_t(dx) = \bar{Y}(t,x) dx) = 1$.

(b) With probability one, $Y(t,x) = \bar{Y}(t,x)$ for a.e. $(t,x) \in \R_+ \times \R^d$.
\end{lemma}
\begin{proof}
Let $f \in \cS$. If $t \in I$, then
\begin{equation*}
\E \bigg( \int |f(x)| \bigg(\int_{(0,t] \times \R^d} 1_{\{s \leq \tau_k\}} p_{t-s}(x-y)^\alpha Y(s,y)^p ds dy \bigg) dx \bigg) \leq \|f\|_\infty \, \E \bigg( \int H_{t,k}(x) dx \bigg)
\end{equation*}
is finite for all $k \in \N$. We may thus apply Lemma~\ref{lemma_fubini2} with $\phi(s,y,x) = p_{t-s}(x-y)$ and $\mu(dx) = f(x)dx$, which yields 
\begin{align*}
&\int f(x) \left[ \iint_{(0,t] \times \R^d} p_{t-s}(x-y) Y(s,y)^\gamma L(dy, ds) \right] dx  
\\ &\hspace{2 cm} = \iint_{(0,t] \times \R^d} P_{t-s} f(y) Y(s,y)^\gamma L(dy,ds) \quad \text{$\bP$-a.s.}
\end{align*}
Substituting the above into \eqref{e_mild_general}, we obtain that for $f \in \cS$,
\begin{align*}
\langle Y_t, f \rangle &= \int f(x) \bigg[P_t Y_0(x) +  \int_{(0,t] \times \R^d} p_{t-s}(x-y) Y(s,y)^\gamma L(dy, ds)  \bigg] dx 
\\ &= \int f(x) \left[P_tY_0(x) + Z(t,x) \right] dx 
\\ &= \int f(x)\bar{Y}(t,x) dx 
\end{align*}
almost surely. We may then take a countable separating class for $\cM_f(\R^d)$, denoted $(f_n)_{n \in \N} \subset \cS$, such that the above holds a.s. with $f= f_n$ for all $n \in \N$. This implies that $Y_t(dx) = \bar{Y}(t,x)dx$ a.s. for each $t \in I$.

To prove part (b), we remark that from part (a) and \eqref{e_densprob2intro}, for any $t>0$, both $Y(s,x)$ and $\bar{Y}(s,x)$ are a.s. densities for $1_{(0,t]}(s) Y_s(dx) ds$ over $(0,t] \times \R^d$, and as such they are equal for a.e. $(s,x) \in (0,t] \times \R^d$. As this is true for all $t>0$, the result follows.
\end{proof}

Part (b) of the above will be particularly important in the arguments that follow, as it allows us to replace $Y(s,x)$ with $\bar{Y}(s,x)$ in all space-time integrals, i.e. integrals with respect to $ds dx$, as well as stochastic integrals with respect to $L(ds,dx)$, which follows by Proposition~\ref{prop_isometryint}. This is done frequently, so we will note it the first time and do so without comment thereafter. 

The next step is obtain moment estimates for $\bar{Y}(t,x)$. These estimates will be used to complete the proof of the stochastic integral formula, in particular to improve it from a.e. $(t,x) \in \R_+ \times \R^d$ to all $(t,x)$. We later refine the preliminary moment estimates to obtain those stated in the theorem. The cases $p < 1$ and $p \geq 1$ require somewhat different treatments. To unify them as much as we can, we introduce the parameter
\begin{equation*}
\op := p \vee 1. 
\end{equation*}
By \eqref{assumption_integ}, $Y \in \mathbb{L}^{p,t}_{\text{a.s.}}$. We observe that we also have $Y \in \LL^{1,t}_{\text{a.s.}}$. Indeed, since $(Y_t)_{t \geq 0} \in \D([0,\infty),\cM_f(\R^d))$, it follows that $\int_0^t Y_s(\cdot)ds$ is a.s. a finite measure. Thus, by \eqref{e_densprob2intro}, $Y(s,x)$ is the density on $(0,t] \times \R^d$ of a finite measure and hence is integrable a.s. In particular, this implies that
\begin{equation} \label{assump_opinteg}
\int_{(0,t] \times \R^d} Y(s,x)^{\op} ds dx < \infty \,\,\, \text{ a.s. for all $t>0$.}
\end{equation}
For $k \in \N$, define the stopping time
\begin{equation*}
\sigma_k := \inf \bigg\{t > 0 : \int_{(0,t] \times \R^d}  Y(s,x)^{\op} ds dx > k \bigg\}.
\end{equation*}
(Of course, if $p \geq 1$ then $\sigma_k = \tau_k$, where $\tau_k$ is from \eqref{def_tau_k_Y}.) By \eqref{assump_opinteg}, $\lim_{k\to \infty} \sigma_k = \infty$ almost surely. We also note by Lemma~\ref{lemma_prelim_density}(b) that \eqref{assump_opinteg} holds with $Y(s,x)$ replaced by $\bar{Y}(s,x)$ and also that $\sigma_k$ a.s. has the same value under this exchange.

\begin{lemma} \label{lemma_finitemomentae}
For any $t>0$,
\begin{equation*}
\E \bigg( \int_{(0,t] \times \R^d} p_{t-s}(x-y) Y(s,y)^{\op} 1_{\{s \leq \sigma_k\}} \,dsdy \bigg) < \infty \quad \text{ for all $k\in \N$}
\end{equation*}
for a.e. $x \in \R^d$.
\end{lemma}
\begin{proof}
Integrating the expectation in the lemma with respect to $x$ (over $\R^d$) and changing the order of integration, by definition of $\sigma_k$ we obtain that
\begin{align*}
& \int \E \left( \int_{(0,t] \times \R^d} p_{t-s}(x-y) Y(s,y)^{\op} 1_{\{s \leq \sigma_k\}} \,ds dy \right)  dx
\\ & \hspace{2 cm} = \E \left( \int_{(0,t] \times \R^d} Y(s,y)^{\op} 1_{\{s \leq \sigma_k\}} ds dy \right) \leq k. 
\end{align*}
Hence for each $k \in \N$ there is at most a Lebesgue-null set of values of $x$ for which the expectation is infinite. The union of this countable collection of exceptional sets is Lebesgue-null and the lemma follows.
\end{proof}

When $Z(t,x)$ is defined, it is useful to view it as the value at time $t$ of the process
\begin{equation*}
s \to  \int_{(0,s] \times \R^d} p_{t-u}(x-y) Y(u,y)^\gamma  L(du,dy),
\end{equation*}
for $s \in [0,t]$. In particular, this enables us to bound the moments of $|Z(t,x)|$ using Lemma~\ref{lemma_bdg}. We will implicitly use this approach without comment in the sequel.

The next several lemmas (Lemmas~\ref{lemma_aux_integ1}-\ref{lemma_moments_allT1}) follow the arguments used by Yang and Zhou in \cite{YZ2017} to prove similar moment bounds for solutions to a class of SPDEs including \eqref{e_spde1}. However, their argument was restricted to the case $d=1$. The changes to the argument in the higher-dimensional setting are non-trivial, and because the argument is quite technical, merely sketching the necessary changes would be unintelligible. We have therefore included the full proof.
 
\begin{lemma} \label{lemma_aux_integ1}
There exists $T_0 \in (0,1]$ and $K_1 = K_1(\alpha, \gamma, d) > 0$ such that for all $t \in (0,T_0]$, 
\begin{equation} \label{e_auxinteglemma1}
\E\Bigg( \int_{(0,t] \times \R^d} \bar{Y}(s,y)^{\op} p_{t-s}(x-y) dsdy \Bigg) \leq K_1  \big[1 + Y_0(1)^{\op-1} t^{1 - (\op - 1)\frac d 2} P_t Y_0(x)\big]
\end{equation}
for a.e. $x \in \R^d$.
\end{lemma}

\begin{proof}
First, let $(t,x) \in \cU$ (as in Lemma~\ref{lemma_aeintegalpha}) so that $Z(t,x)$ is well-defined. For $k \in \N$, we define $Z_k(t,x)$ by
\begin{equation*}
Z_k(t,x) = \int_{(0,t] \times \R^d} p_{t-s}(x-y)  Y(s,y)^\gamma 1_{\{s \leq \sigma_k\}} L(ds,dy). 
\end{equation*}
By Lemma~\ref{lemma_bdg}, we have
\begin{align} \label{e_stochintegBDG1}
\E(|Z_k(t,x)|^{\op}) & \leq C \E \bigg( \int_{(0,t] \times \R^d} p_{t-s}(x-y)^\alpha Y(s,y)^p 1_{\{s \leq \sigma_k\}} ds dy\bigg)^{\op/\alpha}.
\end{align}
Suppose that $p \geq 1$. In this case, $\op = p$ and $\op / \alpha  = \gamma < 1 $. Using the inequality $u^\gamma \leq 1 + u$ for $u \geq 0$, we obtain 
\begin{align} \label{e_stochintegBDG2}
\E(|Z_k(t,x)|^{\op}) & \leq C +  C\E\bigg( \int_{(0,t] \times \R^d} p_{t-s}(x-y)^\alpha Y(s,y)^{\op} 1_{\{s \leq \sigma_k\}} ds dy\bigg).  
\end{align}
On the other hand, if $p < 1$, then $Y(s,y)^p \leq 1 + Y(s,y)$, and we obtain from \eqref{e_stochintegBDG1} that
\begin{align} \label{e_stochintegBDGG22}
\E(|Z_k(t,x)|^{\op}) &\leq  C \E\bigg( \int_{(0,t] \times \R^d} p_{t-s}(x-y)^\alpha \left[1 + Y(s,y)\right] 1_{\{s \leq \sigma_k\}} ds dy \bigg)^{1/\alpha} \notag
\\ & \leq C +  C\E\bigg( \int_{(0,t] \times \R^d} p_{t-s}(x-y)^\alpha Y(s,y) 1_{\{s \leq \sigma_k\}} ds dy\bigg) \notag
\\ &\hspace{1 cm} +  C \int_{(0,t] \times \R^d} p_{t-s}(x-y)^\alpha ds dy .
\end{align}
It is straightforward to show by scaling that
\begin{equation}\label{e_hkpowerinteg}
 \int_{(0,t] \times \R^d} p_{t-s}(x-y)^\alpha ds dy = C t^{1 -(\alpha-1)\frac d 2}
\end{equation}
for some $C>0$ that does not depend on $t$. Using \eqref{e_stochintegBDGG22} and \eqref{e_hkpowerinteg} when $p <1$, and \eqref{e_stochintegBDG2} when $p \geq 1$, we have shown that there is a positive constant $C_0$ (which depends only on $\alpha$, $\gamma$ and $d$) such that for any $(t,x) \in \cU$,
\begin{align} \label{e_stochintegBDG3}
&\E(|Z_k(t,x)|^{\op}) 
\\ &\hspace{0.6 cm} \leq C_0\bigg( 1+ t^{1-(\alpha-1)\frac d 2} + \E\bigg( \int_0^t ds\, (t-s)^{-(\alpha-1)\frac d 2} \int_{\R^d} p_{t-s}(x-y) Y(s,y)^{\op} 1_{\{s \leq \sigma_k\}} dy\bigg)\bigg).\notag
\end{align}
The upper bound of the integral term is obtained by writing $p_{t-s}(x-y)^\alpha = p_{t-s}(x-y) p_{t-s}(x-y)^{\alpha-1}$ and bounding the second term above by $\|p_{t-s}\|_\infty$. The next step is to obtain an upper bound for 
\[ \E \bigg( \int_{(0,T] \times \R^d} |Z(t,x)|^{\op} p_{T-t}(x_0 - x) 1_{\{t \leq \sigma_k\}} dt dx \bigg)\]
for a given $(T,x_0) \in \R_+ \times \R^d$. First, we make the elementary observation that \[|Z(t,x)| 1_{\{t \leq \sigma_k\}} \leq |Z_k(t,x)|\] almost surely. Given this, we may use the fact that \eqref{e_stochintegBDG3} holds for a.e. $(t,x) \in \R_+ \times \R^d$ and apply Fubini's theorem to obtain
\begin{align} \label{e_Zbd1}
&\E \bigg( \int_{(0,T] \times \R^d} |Z(t,x)|^{\op} p_{T-t}(x_0 - x) 1_{\{t \leq \sigma_k\}} dt dx \bigg) \notag
\\&\hspace{0.6 cm}\leq \E \bigg( \int_{(0,T] \times \R^d} |Z_k(t,x)|^{\op} p_{T-t}(x_0 - x) dt dx \bigg) \notag
\\ &\hspace{.6 cm} \leq C_0 T + C_0 T^{2 - (\alpha-1)\frac d 2} \notag
\\ &\hspace{1.2 cm} +C_0 \E\bigg( \int_0^T ds \,  \int_s^T dt (t-s)^{-(\alpha-1)\frac d 2} \int p_{T-s}(x_0-y) Y(s,y)^{\op} 1_{\{s \leq \sigma_k\}}dy \bigg) \notag 
\\ &\hspace{.6 cm} \leq C_1\bigg( T + T^{2 - (\alpha-1)\frac d 2}  + T^{1 - (\alpha-1)\frac d 2}\,  \E\bigg( \int_{(0,T] \times \R^d} p_{T-s}(x_0-y) \bar{Y}(s,y)^{\op} 1_{\{s \leq \sigma_k\}}ds dy \bigg)\bigg),  
\end{align}
where $C_1$ is a positive constant equal to $C_0$ multiplied by a term which depends only on $d$ and $\alpha$. The second inequality uses the semigroup property and a change of the order of integration. In the last line, we have replaced $Y(s,y)^{\op}$ with $\bar{Y}(s,y)^{\op}$, which is justified by Lemma~\ref{lemma_prelim_density}(b).

To complete the proof we need a bound involving (certain integrals of) $P_TY_0(x_0)^{\op}$. Recall that $Y_0(1)$ is the total mass of $Y_0$. For any $0 < t < T$ and $x_0 \in \R^d$, using the inequality
\[(P_tY_0(x))^{\op} = (P_tY_0(x)) (P_tY_0(x))^{\op - 1} \leq CP_tY_0(x) t^{-(\op - 1)\frac d 2} Y_0(1)^{\op-1},\] 
we obtain
\begin{align*}
\int (P_t Y_0(x))^{\op} p_{T-t}(x_0-x) dx & \leq C t^{-(\op-1)\frac d 2} Y_0(1)^{\op - 1} \int P_t Y_0(x) p_{T-t}(x-x_0) dx
\\ &= C t^{-(\op-1)\frac d 2} Y_0(1)^{\op - 1} P_T Y_0(x_0).
\end{align*}
Integrating over $t \in [0,T]$, we obtain that 
\begin{equation}  \label{e_PtYp_bd}
\int_0^T dt \int (P_tY_0(x))^{\op} p_{T-t}(x-x_0) dx \leq C_2 Y_0(1)^{\op-1} T^{1 - (\op-1)\frac d 2} P_TY_0(x_0)
\end{equation}
for all $T>0$ and $x_0 \in \R^d$, for some constant $C_2 >0$. Finally, from \eqref{e_Y_Z_formula} we have the elementary bound
\begin{equation} \label{e_qmoment_ineq}
\bar{Y}(t,x)^{\op} \leq C_3[ P_t Y_0(x)^{\op} + |Z(t,x)|^{\op}]
\end{equation}
for some $C_3>0$. We may use this bound with \eqref{e_Zbd1} and \eqref{e_PtYp_bd} to obtain that
\begin{align*}
&\E \bigg( \int_{(0,T] \times \R^d} \bar{Y}(t,x)^{\op} p_{T-t}(x_0-x) 1_{\{t \leq \sigma_k \}} dt dx  \bigg)  
\\& \hspace{.8 cm} \leq C_3\bigg( \int_{(0,T] \times \R^d} (P_t Y_0(x))^{\op} p_{T-t}(x_0-x) dtdx  \bigg) 
\\ &\hspace{1.6 cm} + C_3 \E \bigg( \int_{(0,T] \times \R^d} |Z(t,x)|^{\op} p_{T-t}(x_0-x) 1_{\{t \leq \sigma_k\}} dt dx   \bigg) 
\\ &\hspace{.8 cm}\leq C_4 \big[T + T^{2 - (\alpha-1)\frac d 2 } + Y_0(1)^{\op-1} T^{1-(\op-1)\frac d 2} P_TY_0(x_0)\big]  
\\ &\hspace{1.6 cm} +  C_4 T^{1 - (\alpha-1)\frac d 2} \E \bigg( \int_{(0,T] \times \R^d} \bar{Y}(s,y)^{\op} p_{T-s}(x_0-y) 1_{\{s \leq \sigma_k\}} dsdy  \bigg) , 
\end{align*}
where $C_4$ is a constant built from $C_1$, $C_2$ and $C_3$ and hence depends only on $(\alpha,\gamma,d)$. Note that the expectations in the first and last expressions are equal. By Lemma~\ref{lemma_finitemomentae}, for every $T>0$, the expectation is finite for all $k \in \N$ for a.e. $x_0 \in \R^d$. For such an $x_0$, it follows that if $T$ is sufficiently small such so that $C_4 T^{1-(\alpha-1)\frac d 2} \leq 1/2$, then
\begin{align*}
&\E \bigg( \int_{(0,T] \times \R^d} \bar{Y}(t,x)^{\op} p_{T-t}(x_0-x) 1_{\{t \leq \sigma_k\}} dtdx  \bigg)
\\ &\hspace{1.5 cm} \leq  2 C_4\big[T + T^{2 - (\alpha-1)\frac d 2 } + Y_0(1)^{\op-1} T^{1-(\op-1)\frac d 2} P_TY_0(x_0)\big]
\end{align*}
for all $k \in \N$. Since $\sigma_k \uparrow \infty$ a.s., we may let $k \to \infty$ and apply monotone convergence to conclude that
\begin{align*}
&\E \bigg( \int_{(0,T] \times \R^d} \bar{Y}(t,x)^{\op} p_{T-t}(x_0-x) dtdx \bigg) 
\\ &\hspace{1 cm}\leq  2 C_4 \big[T + T^{2 - (\alpha-1)\frac d 2 } + Y_0(1)^{\op-1} T^{1-(\op-1)\frac d 2} P_TY_0(x_0) \big]. 
\end{align*}
Fix $T_0 \in (0,1]$ so that $2 C_4 T_0^{1-(\alpha-1)\frac d 2} \leq 1$. For each $T \in (0,T_0]$, the above holds for a.e. $x_0 \in \R^d$. Since the positive powers of $T$ are maximized by $T_0$ (for $T \in (0,T_0]$), the proof is complete. 
\end{proof} 

\begin{lemma} \label{lemma_aux_integ2}
Let $r \in (0,1)$. There is a constant $K_2 = K_2(r,\alpha, \gamma,d) >0$ such that
\begin{align*}
&\E\bigg( \int_0^t ds\, (t-s)^{-r} \int  \bar{Y}(s,y)^{\op} p_{t-s}(x-y) dy \bigg) 
\\ &\hspace{0.13 cm} \leq K_2 \big[ t^3 +   t^{-r}\big]  \bigg( 1 + Y_0(1)^{\op-1} t^{1 - (\op-1)\frac d 2} P_t Y_0(x) +  \E\bigg( \int_{(0,t] \times \R^d}  \bar{Y}(s,y)^{\op} p_{t-s}(x-y) ds dy \bigg) \bigg) 
\end{align*}
for all $(t,x)\in \R_+ \times \R^d$.
\end{lemma}
\begin{proof}
For $\eta \in \R$, $T>0$ and $x_0 \in \R^d$, we define 
\begin{equation*}
G(T,x_0,\eta) = \E \bigg( \int_0^T dt (T-t)^{\eta} \int \bar{Y}(t,x)^{\op} p_{T-t}(x_0-x)dx \bigg).
\end{equation*}
The lemma proves an upper bound on $G(T,x_0,-r)$ for $r \in (0,1)$, but we will require this function for several different values of $\eta$. 

Now let $\eta < 1$. We may argue as we did to obtain \eqref{e_PtYp_bd} to show that for all $(T,x_0) \in \R_+ \times \R^d$,
\begin{align*} 
&\int_0^T dt (T-t)^{-\eta} \int P_t Y_0(x)^{\op} p_{T-t}(x_0-x) dx  
\\ &\hspace{1.5 cm} \leq CY_0(1)^{\op-1} P_T Y_0(x_0) \int_0^T (T-t)^{-\eta} t^{-(\op-1)\frac d 2} dt 
\\ &\hspace{1.5 cm} \leq C Y_0(1)^{\op-1} T^{1-\eta-(\op-1)\frac d 2} P_T Y_0(x_0)
\end{align*}
for a constant $C$ which depends only on $\alpha$, $\gamma$ and $\eta$. Then by \eqref{e_qmoment_ineq} and the above, for $(T,x_0) \in \R_+ \times \R^d$,
\begin{align} \label{e_singintegbdg1}
G(T,x_0,-\eta) &  \leq C \E\bigg( \int_0^T dt (T-t)^{-\eta} \int \big[ P_tY_0(x)^{\op} + |Z(t,x)|^{\op} \big] p_{T-t}(x_0-x) dx \bigg)  \notag
\\ & \leq C Y_0(1)^{\op-1} T^{1-\eta-(\op-1)\frac d 2} P_T Y_0(x_0)\notag
\\ & \hspace{1 cm} + C \E \bigg(\int_0^T dt (T-t)^{-\eta} \int |Z(t,x)|^{\op} p_{T-t}(x_0-x)dx  \bigg).
\end{align}
Next, we argue exactly as in \eqref{e_stochintegBDG1}-\eqref{e_stochintegBDG3}, but without the term $1_{\{s \leq \sigma_k\}}$, to obtain that
\begin{equation} \label{e_integ_qbd}
\E(|Z(t,x)|^{\op}) \leq  C \E \bigg(1 + t^{1 - (\alpha-1)\frac d 2} + \int_0^{t} ds\, (t-s)^{-(\alpha-1)\frac d 2} \int  p_{t-s}(x-y) \bar{Y}(s,y)^{\op} dy  \bigg)
\end{equation}
for all $(t,x) \in \cU$. We now proceed similarly to \eqref{e_Zbd1} in order to bound the last term appearing in \eqref{e_singintegbdg1}, and we use \eqref{e_integ_qbd} to compute
\begin{align*}
&\E \bigg( \int_0^T dt\, (T-t)^{-\eta} \int |Z(t,x)|^{\op} p_{T-t}(x_0-x) \,dx \bigg) 
\\ &\hspace{.1 cm} \leq C T^{1-\eta} + CT^{2-\eta-(\alpha-1)\frac d 2}
\\ &\hspace{0.22 cm} + C \E \bigg( \int_0^T ds \int_s^T dt \,(t-s)^{-(\alpha-1)\frac d 2} (T-t)^{-\eta} \iint \bar{Y}(s,y)^{\op} p_{t-s}(x-y) p_{T-t}(x_0-x)  dydx \bigg) 
\\ &\hspace{.1 cm} \leq C T^{1-\eta} + CT^{2-\eta-(\alpha-1)\frac d 2}
\\ &\hspace{0.22 cm} + C \E \bigg( \int_0^T ds \int \bar{Y}(s,y)^{\op} p_{T-s}(x_0-y) \, dy \int_s^T dt\, (T-t)^{-\eta} (t-s)^{-(\alpha-1)\frac d 2}\bigg) 
\\&\hspace{.1 cm} \leq C \bigg( T^{1-\eta} + T^{2-\eta-(\alpha-1)\frac d 2}  + \E\bigg( \int_0^T ds (T-s)^{1 - (\alpha-1)\frac d 2 - \eta} \int \bar{Y}(s,y)^{\op} p_{T-s}(x_0-y) dy \bigg) \bigg). 
\end{align*}
The constant $C$ depends on $\alpha$, $\gamma$, $d$ and $\eta$. We let $\kappa := 1 - (\alpha-1)\frac d 2 > 0$ and observe that the last term in the bottom line is equal to $G(T,x_0,-\eta + \kappa)$. In particular, the above combined with \eqref{e_singintegbdg1} gives
\begin{equation*}
G(T,x_0,-\eta) \leq C[T^{1-\eta} + T^{2-\eta - (\alpha-1)\frac d 2} +Y_0(1)^{\op-1} T^{1-\eta-(\op-1)\frac d 2} P_T Y_0(x_0) + G(T,x_0,\kappa -\eta)].
\end{equation*}
Hence, at the cost of several additive error terms, we can bound the expected value $G(T,x_0,-\eta)$, of the integral with singularity $(T-t)^{-\eta}$, by the expected value of the integral with singularity $(T-t)^{-(\eta - \kappa)}$. This process can then be iterated until the power of $(T-t)$ is no longer negative.

Let $r<1$, $x_0 \in \R^d$, and $T>0$. We apply the above bound iteratively, first with $\eta = r$, then $\eta = r - \kappa$, and so on. (This is the case unless $- r + \kappa > 0$, in which case the iteration which we now describe has been completed in a single step.) Doing so $\lceil r/\kappa \rceil$ times, we obtain the following: 
\begin{align*}
&G(T,x_0,-r)
\\ &\hspace{.4 cm} \leq C\sum_{l=0}^{\lceil r/\kappa \rceil - 1} \big[T^{1-r + l\kappa} + T^{2-r - (\alpha - 1)\frac d 2 + l\kappa} + Y_0(1)^{\op - 1} T^{1-r - (\op-1)\frac d 2 + l\kappa} P_T Y_0(x_0) \big] 
\\ &\hspace{1 cm} + C G(T,x_0,\zeta),
\end{align*}
where $\zeta = \kappa \cdot \lceil r/\kappa \rceil  - r \geq 0$. The constants may change with each iteration, but because there are only finitely many iterations we can and do choose the maximum constant which arises. Since $\zeta \geq 0$, we have $(T-t)^{\zeta} \leq T^\zeta$ for all $t \in (0, T]$, and hence
\[ G(T,x_0,\zeta) \leq T^\zeta \E \bigg( \int_{(0,T] \times \R^d} \bar{Y}(t,z)^{\op} p_{T-t}(x_0-x) \,dxdt \bigg) .\]
Using the above in the previous display and collecting the powers of $T^{-r}$, we have
\begin{align*}
&G(T,x_0,-r)
\\ &\hspace{.4 cm} \leq C T^{-r} \Bigg[  \sum_{l=0}^{\lceil r/\kappa \rceil - 1} \big[T^{1 + l\kappa} + T^{2 - (\alpha - 1)\frac d 2 + l\kappa} + Y_0(1)^{\op - 1} T^{1 - (\op-1)\frac d 2 + l\kappa} P_T Y_0(x_0) \big] 
\\ &\hspace{1 cm} + T^{\zeta+r} \E\bigg( \int_{(0,T] \times \R^d} \bar{Y}(t,x)^{\op} p_{T-t}(x-x_0) \,dtdx \bigg) \Bigg] 
\end{align*}
for a constant $C>0$ which depends only on $\alpha$, $\gamma$, $d$ and $r$. Note that all the remaining powers of $T$ inside the square-bracketed term are positive, and let $\sigma' >0 $ be the largest positive power of $T$ appearing in the square-bracketed term above. (Clearly $\sigma' > 1$, and a short calculation shows that $\sigma' < 2 - (\alpha-1)\frac d 2 + r < 3$.) Collecting terms and bounding above by the largest ones, it is easy to argue that there exists $K_2>0$ such that
\begin{align*}
G(T,x_0,-r) \leq & K_2 T^{-r} (1 +  T^{\sigma'}) \bigg[ 1 +  Y_0(1)^{\op - 1} T^{1 - (\op-1)\frac d 2} P_T Y_0(x_0) 
\\ &\hspace{1 cm} +  \E\bigg( \int_{(0,T] \times \R^d} \bar{Y}(t,x)^{\op} p_{T-t}(x_0-x) \,dxdt \bigg) \bigg].
\end{align*}
Since $\sigma' < 3$ and $r>0$, possibly increasing the value of $K_2$, we may bound the expression above by the one stated in the lemma, and the proof is complete.
\end{proof} 

We can now obtain moment bounds which hold a.e. for short times.
\begin{lemma} \label{lemma_moments_smalltime}
Let $T_0$ be as in Lemma~\ref{lemma_aux_integ1}. There is a constant $C = C(\alpha, \gamma, d)> 0$ such that for a.e. $(t,x) \in (0,T_0] \times \R^d$,
\begin{equation} \label{e_smalltimemomentslemma}
\E(\bar{Y}(t,x)^{\op}) \leq  C t^{-(\alpha-1)\frac d 2} \big[ 1 + Y_0(1)^{\op-1} t^{1-(\op-1)\frac d 2} P_tY_0(x)\big] + C P_tY_0(x)^{\op}.
\end{equation}
\end{lemma}
\begin{proof}
We first observe that by Lemma~\ref{lemma_aux_integ1}, there is a subset $B$ of $(0,T_0] \times \R^d$ of full Lebesgue measure such that $B \subseteq \cU$ and \eqref{e_auxinteglemma1} holds for all $(t,x) \in B$. It suffices to show that \eqref{e_smalltimemomentslemma} holds for all $(t,x) \in B$. For such $(t,x)$, by \eqref{e_integ_qbd}, Lemma~\ref{lemma_aux_integ2} with $r = (\alpha - 1)\frac d 2$, and \eqref{e_auxinteglemma1}, we have
\begin{align}
&\E(|Z(t,x)|^{\op}) \notag 
\\ &\hspace{.6 cm} \leq C \bigg( 1 + t^{1-(\alpha-1)\frac d 2} +K_2 \big[t^3 + t^{-(\alpha-1)\frac d 2} \big] \bigg[ 1 +  Y_0(1)^{\op - 1} t^{1 - (\op-1)\frac d 2} P_t Y_0(x)  \notag
\\ &\hspace{1.2 cm}+  \E\bigg( \int_{(0,t] \times \R^d} \bar{Y}(s,y)^{\op} p_{t-s}(x-y) \,dsdy \bigg) \bigg]\bigg) \notag
\\ &\hspace{.6 cm} \leq C \bigg( 1 + t^{1-(\alpha-1)\frac d 2} +K_2 \big[t^3 + t^{-(\alpha-1)\frac d 2} \big] \bigg[ 1 +  Y_0(1)^{\op - 1} t^{1 - (\op-1)\frac d 2} P_t Y_0(x)  \notag
\\ &\hspace{1.2 cm}+  K_1\big[1 + Y_0(1)^{\op-1} t^{1 - (\op-1)\frac d 2} P_tY_0(x)\big] \bigg]\bigg), \notag
\end{align}
where $K_2 = K_2((\alpha-1)d/2)$, and $C>0$ is independent of $t$. Since $t \leq T_0$, we may collect terms in the above and conclude that for some constant $C>0$,
\begin{equation*}
\E(|Z(t,x)|^{\op}) \leq C(1 + t^{-(\alpha-1)\frac d 2})\big[ 1 + Y_0(1)^{\op-1} t^{1-(\op-1)\frac d 2} P_tY_0(x) \big].
\end{equation*}
By \eqref{e_qmoment_ineq}, we have
\begin{align*}
\E(\bar{Y}(t,x)^{\op}) &\leq CP_tY_0(x)^{\op} + C\E(|Z(t,x)|^{\op}), 
\end{align*}
and the bound above thus implies that
\begin{align*}
\E(\bar{Y}(t,x)^{\op}) &\leq  C(1 + t^{-(\alpha-1)\frac d 2})\big[ 1 + Y_0(1)^{\op-1} t^{1-(\op-1)\frac d 2} P_tY_0(x)\big] + C P_tY_0(x)^{\op}
\\ &\leq  C t^{-(\alpha-1)\frac d 2} \big[ 1 + Y_0(1)^{\op-1} t^{1-(\op-1)\frac d 2} P_tY_0(x)\big] + C P_tY_0(x)^{\op},
\end{align*}
where the second line holds by enlarging the constant, since $t \leq T_0$. Thus the desired inequality holds for all $(t,x) \in B$ for some constant $C \geq 1$, which completes the proof. \end{proof}

Given the moment bound holds for small times, one can bootstrap and iterate to prove that the bound from Lemma~\ref{lemma_aux_integ1} holds, with an enlarged constant, for $t \in (0,T]$ for any $T>0$. That is, there exists an increasing family of constants $K_1(T)>0$ such that for all $t \in (0,T]$,
\begin{equation} \label{e_integ_extra}
\E \bigg( \int_{(0,t] \times \R^d} \bar{Y}(s,y)^{\op} p_{t-s}(x-y) dsdy \bigg) \leq K_1(T) \big[ 1 + Y_0(1)^{\op-1} t^{1-(\op-1)\frac d 2} P_t Y_0(x) \big]
\end{equation}
for a.e. $x \in \R^d$. The process is iterative: given the moment estimate from Lemma~\ref{lemma_moments_smalltime} holds for $t \in (0,T_{n-1}]$ for some $T_{n-1}$, one can bootstrap to prove that \eqref{e_integ_extra} holds for $t \in (0,T_n]$, where $T_n > T_{n-1}$, then establishes that the moment estimate from Lemma~\ref{lemma_moments_smalltime} holds (with an enlarged constant) for $t \in (0,T_n]$, and so on, along a sequence $T_n \uparrow \infty$. This technical bootstrapping argument is carried out by Yang and Zhou; see Steps 3 and 4 of the proof of Proposition~2.4 in the Appendix of \cite{YZ2017}. Although elements of our proof in the steps above differ from theirs, the proofs are structurally the same and in particular their bootstrapping argument works in our setting, under our assumptions, with virtually no modification. We therefore omit the proof.

Given \eqref{e_integ_extra}, one can repeat the argument in the proof of Lemma~\ref{lemma_moments_smalltime} for $t \in (0,T]$ for any $T>0$. We now state the preliminary $\op$th moment bound with no restriction on the time parameter.
\begin{lemma} \label{lemma_moments_allT1}
For all $T>0$ there are constants $C_T = C(T,\alpha,\gamma,d)\geq 1$, increasing in $T$, such that for a.e. $(t,x) \in (0,T] \times \R^d$, 
\begin{equation*} 
\E(\bar{Y}(t,x)^{\op}) \leq  C_T t^{-(\alpha-1)\frac d 2}\big[ 1 + Y_0(1)^{\op-1} t^{1-(\op-1)\frac d 2} P_tY_0(x)\big] + C_T P_tY_0(x)^{\op}.
\end{equation*}
\end{lemma}

Next, we use the $\op$th moment bound to improve Lemma~\ref{lemma_aeintegalpha} and show that $Z(t,x)$, and hence $\bar{Y}(t,x)$, is well-defined for all $(t,x) \in \R_+ \times \R^d$. It suffices to show that $\phi$ defined by $\phi(s,y) = p_{t-s}(x-y) \bar{Y}(s,y)^\gamma$ is in $\LL^{\alpha,t}_{\text{a.s.}}$. This will follow if we show that it belongs to $\LL^\alpha(\LL^{\alpha,t})$, which we can now do using moment estimates. First suppose $p\geq 1$ and $(t,x) \in \R_+ \times \R^d$. For $s \leq t$, from Lemma~\ref{lemma_moments_allT1} and the now-familiar upper bound $P_sY_0(x)^p \leq Cs^{-(p-1)\frac d 2} Y_0(1)^{p-1} P_sY_0(x)$, we have
\begin{align*}
\E(\bar{Y}(s,y)^p) &\leq C(t,Y_0(1))\big[s^{-(\alpha-1)\frac d 2} (1 + P_sY_0(y)) + s^{-(p-1)\frac d 2} P_s Y_0(y) \big]
\\& \leq C'(t,Y_0(1))s^{-(\alpha-1)\frac d 2} (1 + P_sY_0(y)),
\end{align*}
where in the second line we increase the constant as necessary to make it hold for all $s \leq t$. We thus obtain that 
\begin{align*} 
&\E \bigg(\int_{(0,t] \times \R^d} \bar{Y}(s,y)^p p_{t-s}(x-y)^\alpha \,dsdy \bigg) 
\\ &\hspace{0.6 cm} = \int_{(0,t] \times \R^d} \E (\bar{Y}(s,y)^p) p_{t-s}(x-y)^\alpha \,dsdy  
\\ &\hspace{0.6 cm} \leq C'(t,Y_0(1)) \int_0^t ds \,(t-s)^{-(\alpha-1)\frac d 2} s^{-(\alpha-1)\frac d 2} \int (1 + P_sY_0(y)) p_{t-s}(x-y) dy  
\\ & \hspace{0.6 cm} = C'(t,Y_0(1))(1+ P_tY_0(x)) \int_0^t ds \, (t-s)^{-(\alpha-1)\frac d 2} s^{-(\alpha-1)\frac d 2}   < \infty,
\end{align*}
where last line uses the semigroup property and the bound is finite because the powers of $s$ and $(t-s)$ are all integrable. To show that the expectation is finite when $p<1$ is similar and in fact easier; one uses the inequality $\bar{Y}(s,y)^p \leq 1 + \bar{Y}(s,y)$ and argues the same way. Thus $(s,y) \to p_{t-s}(x-y) \bar{Y}(s,y)^\gamma$ is in $\LL^\alpha(\LL^{\alpha,t})$ for all $(t,x) \in \R_+ \times \R^d$, and hence we may define $Z(t,x)$ for all $(t,x)$. Moreover, we may similarly show that the expectation from Lemma~\ref{lemma_finitemomentae} is finite for all $k\in \N$, for all $(t,x) \in \R_+ \times \R^d$. Thus the arguments from Lemmas~\ref{lemma_aux_integ1}-\ref{lemma_moments_allT1} may be repeated with no restriction on $(t,x)$ to obtain moment bounds for all values of $(t,x)$. 
 
We collect our results up to this point in the following proposition. Because we can define $Z(t,x)$ and $\bar{Y}(t,x)$ for all $(t,x)$, the ``almost everywhere" statements from Lemma~\ref{lemma_prelim_density} can now be made with no restrictions.

\begin{proposition} \label{prop_integral_summary} (a) For every $(t,x) \in \R_+ \times \R^d$, $Z(t,x)$ and $\bar{Y}(t,x)$ are well-defined. 

(b) For each $t>0$, $\bP(Y_t(dx) = \bar{Y}(t,x) dx) = 1$. We also have, with probability one, $\bar{Y}(t,x) = Y(t,x)$ a.e. on $\R_+ \times \R^d$.

(c) For all $T>0$ there are constants $C_T = C(T,\alpha,\gamma,d) >0$, increasing in $T$, such that 
\begin{equation*} 
\E(\bar{Y}(t,x)^{\op}) \leq  C_T t^{-(\alpha-1)\frac d 2}\big[ 1 + Y_0(1)^{\op-1} t^{1-(\op-1)\frac d 2} P_tY_0(x)\big] + C_T P_tY_0(x)^{\op}
\end{equation*}
for all $(t,x)\in (0,T] \times \R^d$. 
\end{proposition}

Part (a) of Theorem~\ref{thm_stochint2} is established by parts (a) and (b) of the above. We remark that there exists a predictable version of $\bar{Y}(t,x)$ by \cite[Lemma A.2]{C2017b}. We still need to prove the moment estimates (as stated) and prove part (c). We first finish with the $q$th moment estimates for $q \in [1,\alpha)$, i.e. \eqref{e_thm_momentbd_pgo} and \eqref{e_thm_momentbd_plo}, with the following lemma.

\begin{lemma} \label{lemma_moments_final}
Let $q \in [1,\alpha)$. There exists an increasing family of constants $C_T = C(T,q,\alpha,\gamma,d)$, for $T>0$, such that the following holds: if $p>1$, for all $(t,x) \in (0,T] \times \R^d$, 
\begin{equation*}
\E(\bar{Y}(t,x)^q) \leq C_T t^{-(\alpha-1)\frac d 2 \frac q \alpha} ( 1 + t^{1-(p-1) \frac d 2}Y_0(1)^{p-1} P_t Y_0(x))^{\frac q \alpha} + C_T P_t Y_0(x)^q;
\end{equation*}
and if $p \leq 1$,
\begin{equation*}
\E(\bar{Y}(t,x)^q) \leq C_T t^{-(\alpha-1)\frac d 2 \frac q \alpha} ( 1 + t P_t Y_0(x))^{\frac q \alpha} + C_T P_t Y_0(x)^q.
\end{equation*}
\end{lemma}
\begin{proof}
Let $T>0$, $(t,x) \in (0,T] \times \R^d$ and $q \in [1,\alpha)$. By Lemma~\ref{lemma_bdg},
\begin{align} \label{e_lastmomentlemmaaux1}
\E(|Z(t,x)|^q) &\leq C_q  \E  \bigg( \int_{(0,t] \times \R^d} p_{t-s}(x-y)^\alpha Y(s,y)^p dsdy \bigg)^{q/\alpha}  \notag
\\ &\leq C_q \E \bigg( \int_0^t ds \, (t-s)^{-(\alpha-1)\frac d 2} \int p_{t-s}(x-y) \bar{Y}(s,y)^p dy\bigg)^{q/\alpha}.
\end{align} 
First suppose that $p > 1$. In view of the discussion preceding the statement of Proposition~\ref{prop_integral_summary}, \eqref{e_integ_extra} holds for all $(t,x) \in (0,T]\times \R^d$, not merely almost everywhere. Combining this with Lemma~\ref{lemma_aux_integ2}, we therefore have that 
\begin{align}
&\E \bigg( \int_0^t ds \, (t-s)^{-(\alpha-1)\frac d 2} \int p_{t-s}(x-y) \bar{Y}(s,y)^p dy \bigg)  \notag
\\ &\hspace{0.6 cm} \leq K_2[t^3 + t^{-(\alpha-1)\frac d 2}]\bigg[1 + Y_0(1)^{p-1}t^{1-(p-1)\frac d 2} P_tY_0(x)  \notag
\\ &\hspace{ 1.2 cm} +  \E \bigg( \int_{(0,t] \times \R^d} \bar{Y}(s,y)^p p_{t-s}(x-y) dsdy \bigg)\bigg] \notag
\\ &\hspace{0.6 cm} \leq C(T) t^{-(\alpha-1)\frac d 2}(1 + t^{1-(p-1)\frac d 2}Y_0(1)^{p-1} P_tY_0(x)). \notag
\end{align}
Substituting this into \eqref{e_lastmomentlemmaaux1} and using the fact that $\bar{Y}(t,x)^q \leq C(P_tY_0(x)^q + |Z(t,x)|^q)$ for a constant $C$ which only depends on $q$, the stated bound follows. 

To prove the result for $p<1$, one simply notes that $\bar{Y}(s,y)^p \leq 1 + \bar{Y}(s,y)^{\op}$ (recall $\op = 1$ in this case), and argues as in the previous case. Since $\op=1$, we have $Y_0(1)^{\op-1} = t^{-(\op-1)\frac d 2} = 1$, which leads to the simpler form of the bound.
\end{proof}

The next lemma establishes Theorem~\ref{thm_stochint2}(c).
\begin{lemma} \label{lemma_moll} Let $\psi \in C^\infty_c(\R^d)$ be non-negative and satisfy $\int \psi = 1$, and for $\eps>0$ define $\psi_\eps(x) = \eps^{-d} \psi(x/\eps)$. Then for all $(t,x) \in \R_+ \times \R^d$, $\psi_\eps * Y_t(x) \to \bar{Y}(t,x)$ in $\LL^q(\bP)$ as $\eps \downarrow 0$ for all $q \in [1,\alpha)$.
\end{lemma}
\begin{proof}
Let $(t,x) \in \R_+ \times \R^d$ and let $\psi_\eps$ be as in the statement of the lemma. Since $Y_t(dx) = \bar{Y}(t,x)dx$ a.s., by \eqref{e_Y_Z_formula} we have
\begin{align*}
 \psi_\eps * Y_t(x) = \psi_\eps * P_t Y_0(x) + \int \psi_\eps(x-z) Z(t,z)dz 
\end{align*}
almost surely. It is straightforward to see that the conditions of Lemma~\ref{lemma_fubini2} are satisfied and we can apply the result to the last term, which gives
\begin{align*}
\int \psi_\eps(x-z) Z(t,z)dz &= \int \psi_\eps(x-z) \bigg[ \int_{(0,t] \times \R^d} Y(s,y)^\gamma  p_{t-s}(z-y) \,L(ds,dy) \bigg] dz 
\\ &= \int_{(0,t]\times \R^d}  \bigg[\int \psi_\eps(x-z) p_{t-s}(z-y) \, dz\bigg] Y(s,y)^\gamma L(ds,dy) 
\\ &= \int_{(0,t]\times \R^d} (\psi_\eps * p_{t-s})(x-y)Y(s,y)^\gamma L(ds,dy). 
\end{align*}
Let $g^\eps(s,y) =  p_s(y) - \psi_{\eps}* p_s(y)$. Then for any $q \in [1,\alpha)$, there is some constant $C>0$ such that
\begin{align*}
|\bar{Y}(t,x) - \psi_\eps * Y_t(x)|^q & \leq C|P_t Y_0(x) - \psi_\eps * P_t Y_0(x)|^q 
\\ &\hspace{1 cm}+ C\bigg| \int_{(0,t]\times \R^d}g^\eps(t-s,x-y) Y(s,y)^\gamma L(ds,dy) \bigg|^q.
\end{align*} 
The regularity and integrability properties of $P_t Y_0$ are such that $|P_t Y_0(x) - \psi_\eps * P_t Y_0(x)|^q$ vanishes as $\eps \to 0$. It thus suffices to control the second term in the above. By Lemma~\ref{lemma_bdg} and Fubini, the expectation of this term is bounded above by a constant multiple of 
\begin{align} \label{e_diffint0}
 \bigg( \int_{(0,t)\times \R^d} |g^\eps(t-s,x-y)|^\alpha  \E(\bar{Y}(s,y)^p) \,dsdy \bigg)^{q/\alpha}.
\end{align}
It suffices to show that the integral in the above vanishes as $\eps \downarrow 0$. To do so, we break it into the integral over $(0,t/2)$ and $[t/2,t]$. First, we note that by Lemma~\ref{lemma_moments_final}, 
\[ \sup_{s \in [t/2,t], y \in \R^d} \E(\bar{Y}(s,y)^p) < \infty.\]
Hence, to show that the integral over $[t/2,t] \times \R^d$ vanishes is equivalent, by a change of variables, to proving that
\begin{align} \label{e_diffint1}
\lim_{\eps \downarrow 0} \int_{[0,t/2] \times \R^d} |p_s(y) - \psi_\eps * p_s(y)|^\alpha dsdy = 0.
\end{align}
By Young's convolution inequality (see \cite[Theorem~8.7]{Folland}), for each $s \in (0,t/2]$ we have
\begin{equation} \label{e_Youngconv}
\| \psi_\eps * p_s \|_\alpha \leq \|\psi_\eps\|_1 \|p_s\|_\alpha = \|p_s \|_\alpha
\end{equation}
for all $\eps \in (0,1]$. We also have, for $s\in (0,t/2]$,
\begin{align*}
\|p_s - \psi_\eps * p_s\|_\alpha \leq \|p_s\|_\alpha + \|\psi_\eps * p_s\|_\alpha &\leq  2 \|p_s\|_\alpha \leq 2C s^{\frac{(\alpha-1)}{\alpha}\frac d 2 }
\end{align*}
for all $\eps \in (0,1]$, where the last inequality can be seen by scaling, similar to \eqref{e_hkpowerinteg}. Now let $\delta \in (0,t/2)$. Then
\begin{align*}
\int_{(0,\delta]\times \R^d} |p_s(y) - \psi_\eps * p_s(y)|^\alpha ds dy &= \int_0^\delta \|p_s - \psi_\eps * p_s\|_\alpha^\alpha ds
\\ &\leq (2C)^\alpha \int_0^\delta s^{ (\alpha-1)\frac d 2} ds 
\\ &\leq C \delta^{1 - (\alpha-1)\frac d 2}
\end{align*}
for some enlarged constant $C$. Thus the part of the integral over $(0,\delta]$ is bounded, uniformly for $\eps \in (0,1]$, by a quantity which vanishes as $\delta \downarrow 0$. Thus we may restrict our attention to the integral over $s\in (\delta,t/2]$ for $\delta >0$.  

Arguing as above, we note that 
\[\int_{(\delta,t/2] \times \R^d} |p_s(y) - \psi_\eps * p_s(y)|^\alpha ds dy = \int_\delta^{t/2} \|p_s - \psi_\eps * p_s\|_\alpha^\alpha ds.\] 
By \cite[Theorem~8.14]{Folland}, $\psi_\eps * p_s \to p_s$ in $\LL^\alpha$ for all $s \in (\delta,t/2]$, and hence the integrand vanishes point-wise as $\eps \downarrow 0$. The bound \eqref{e_Youngconv} allows us to apply dominated convergence. This proves \eqref{e_diffint1}. Thus we have handled the part of the integral in \eqref{e_diffint0} from $s \in [t/2,t)$. 

We now handle the integral over $s \in (0,t/2)$. Since $\psi_\eps$ is an approximation of the identity (i.e. $\psi_\eps \to \delta_0$ as a distribution as $\eps \downarrow 0$), $g^\eps(t-s,x-y) \to 0$ point-wise as $\eps \downarrow 0$, and hence we just need to apply dominated convergence.  We need only obtain a uniform upper bound for $|g^\eps(t-s,x-y)|^\alpha  \E(\bar{Y}(s,y)^p)$ over $(s,y) \in (0,t/2) \times \R^d$. 

Again, suppose that $x=0$. We claim that, uniformly in $s \in (0,t/2)$ and $\eps \in (0,1]$, $|g^\eps(t-s,\cdot)|$ is bounded above by a bounded function with Gaussian tails. That it is bounded is straightforward to see: $\sup_{s \in (0,t/2)} \|p_{t-s}\| \leq C (t/2)^{-d/2}$ for a positive constant $C$, and hence $\sup_{s\in(0,t/2), y \in \R^d} |g^\eps(t-s,y)|$ is finite. 

We may assume without loss of generality that $\psi$ is supported on $B(0,1)$, the closed ball of radius one centered at $0$. Since the heat kernel is radially decreasing and $\psi_\eps$ has support of radius at most $\eps$, for all $u \in (t/2,t)$, $\eps \in (0,1]$, and $y \in \R^d$ with $|y| \geq 1$,
\begin{align*}
\psi_\eps * p_u(y) \leq  p_u(|y| - \eps) \leq p_u(|y| - 1),
\end{align*}
where, if $d>1$, for $r>0$ we define $p_u(r)$ as $p_u(y)$ for any $y \in \R^d$ with $|y| = r$. It is straightforward to see that there exists some $R \geq 2$ which depends on $t$ such that for all $u \in (t/2,t)$ and $y \in \R^d$ with $|y| \geq R-1$, $p_u(y) \leq p_t(y)$. Combined with the previous inequality, it follows that for $\eps \in (0,1]$, $u \in (t/2,t)$, and $|y| \geq R$,
\begin{align*}
|g^\eps(u,y)| \leq p_u(y) + \psi_\eps * p_u(y) \leq p_u(y) + p_u(|y|-1)
& \leq 2 p_t(|y|-1).
\end{align*}
This establishes that $|g^\eps(t-s,\cdot)|$ has an upper bound with Gaussian tails, uniformly in $s \in (0,t/2)$ and $\eps \in (0,1]$. Consequently, so does $|g^\eps(t-s,\cdot)|^\alpha$. Since we have shown it is bounded as well, it is now easy to argue using Lemma~\ref{lemma_moments_final} that for $(s,y) \in (0,t/2) \times \R^d$, $|g^\eps(t-s,x-y)|^\alpha \E(\bar{Y}_s(y)^p)$ has an integrable upper bound which is uniform in $\eps \in (0,1]$. Hence the dominated convergence theorem applies and this part of the integral vanishes as $\eps \downarrow 0$. This completes the proof. \end{proof}

Between Proposition~\ref{prop_integral_summary}, Lemma~\ref{lemma_moments_final} and Lemma~\ref{lemma_moll}, we have proved every claim in Theorem~\ref{thm_stochint2} except for \eqref{e_thm_meanmeasure}. It is enough to prove the first moment bound $\E(\bar{Y}(t,x)) \leq P_tY_0(x)$, since the bound for $q \in (0,1)$ then follows by Jensen's inequality. The first moment bound is established in the following lemma by a martingale argument. If the stochastic integral in the definition of $\bar{Y}(t,x)$ were a true martingale, we would obtain the mean-measure formula $\E(\bar{Y}(t,x)) = P_t Y_0(x)$. However, the stochastic integral is a priori only a local martingale; we content ourselves then with the following one-sided mean measure bound.

\begin{lemma} \label{lemma_meanmeasure} For every $(t,x) \in \R_+ \times \R^d$,
\begin{equation*} \E ( \bar{Y}(t,x)) \leq P_t Y_0(x).\end{equation*}
\end{lemma}
\begin{proof}
For $x \in \R^d$, $t>0$ and $0\leq s \leq t$, define
\begin{equation*}
M^t_s(x) := \int_{(0,s] \times \R^d} p_{t-u}(x-y) Y(u,y)^\gamma L(du,dy).
\end{equation*}
Then for any $(t,x) \in \R_+ \times \R^d$, by definition of $\bar{Y}(t,x)$,
\begin{equation} \label{e_lemma_mm_mart0}
\bar{Y}(t,x) = P_tY_0(x) + M^t_t(x).
\end{equation}
To prove the result, it suffices to show that $(M^t_s(x))_{s \in [0,t]}$ is a supermartingale, as this implies it has expectation at most $0$. We will show that $M^t_s(x)$ is bounded below, since a local martingale which is bounded below is a supermartingale. (For a non-negative local martingale, this fact follows from an application of Fatou's Lemma, and any local martingale which is bounded below can be shifted to be non-negative, whence the claim follows.)

Fix $t>0$ and $s \in (0,t]$. By the semigroup property,
\begin{align*}
M^t_s(x) &= \int_{(0,s] \times \R^d} P_{t-s}p_{s-u}(x-y) Y(u,y)^\gamma L(du,dy)
\\ &= \int_{(0,s] \times \R^d} \bigg(  \int p_{t-s}(x-z) p_{s-u}(z-y)dz \bigg) Y(u,y)^\gamma L(du,dy).
\end{align*}
The conditions of Lemma~\ref{lemma_fubini2} are satisfied by $\phi(u,y,z) = p_{s-u}(z-y)$ and $\mu(dz) = p_{t-s}(x-z)dz$. Applying that result, we obtain
\begin{align*}
M^t_s(x) &= \int p_{t-s}(x-z) M^s_s(z) dz = (P_{t-s}M^s_s)(x).
\end{align*}
Hence by \eqref{e_lemma_mm_mart0} (at time $s$),
\begin{align*}
M^t_s(x) &=  (P_{t-s} \bar{Y}(s,\cdot))(x) - P_{t-s} P_sY_0(x)
\\ &=  P_{t-s}Y_s(x)- P_tY_0(x) 
\\ &\geq - P_tY_0(x)
\end{align*}
almost surely. The second line uses Theorem~\ref{thm_stochinteg}(a) to assert that $\bar{Y}(s,y) dy = Y_s(dy)$ a.s. Let $(s_n)_{n \in \N}$ be a countable and dense subset of $(0,t]$. Then with probability one, 
\begin{equation*}
M^t_{s_n}(x) \geq - P_tY_0(x) \quad \text{ for all $n\in \N$.}
\end{equation*}
Since the process $(M^t_{s})_{s \in [0,t]}$ is c\`adl\`ag, to prove that it is bounded below it suffices to show that it is bounded below on a dense subset. This is exactly what we have shown, and the proof is complete.
\end{proof}

\appendix
\section{Appendix: proofs of properties of the stochastic integral}
In this appendix, we provide a few more details on the construction of the stable stochastic integral and give proofs for some results in Section~\ref{s_stochcalc}. For convenience, we follow the construction of Balan \cite{B2013}, which is particular to the stable case, rather than a more general construction, e.g. that in \cite{Bichteler}.  

As is usual, one first defines the stochastic integral for elementary processes, that is processes $\phi$ of the form
\begin{equation} \label{e_elem_function}
\phi(\omega,s,x) = \xi(\omega)1_{(t_1,t_2]}(s) 1_A(x).
\end{equation}
where $0 \leq t_1 < t_2<\infty$, $\xi$ is an $\cF_{t_1}$-measurable random variable, and $A \in \underline{\cB}(\R^d)$, the collection of Borel sets with finite Lebesgue measure. The stochastic integral of $\phi$ with respect to $L$ is defined to be
\begin{equation*} 
(\phi \cdot L)_t := \xi (L_{t \wedge t_2}(A) - L_{t \wedge t_1}(A)).
\end{equation*}
A process is simple if it can be expressed as a finite sum of elementary processes. Starting with the elementary processes, it is straightforward to define the stochastic integral for all simple processes. One then defines the stochastic integral for more general processes by approximation. In particular, the bounded simple processes are dense in $\LL^{\alpha}(\LL^{\alpha,t})$. One then extends the integral to integrands in $\LL^\alpha(\LL^{\alpha,t})$ by way of the upper bound from Proposition~\ref{prop_isometryint}. This bound was proved in \cite{B2013} as Theorem 13 and Lemma 14, and allows the integral to be extended to $\LL^\alpha(\LL^{\alpha,t})$. It is first proved for all simple integrands, and implies that the stochastic integrals associated to a $\LL^{\alpha}(\LL^{\alpha,t})$-convergent sequence of simple integrands is Cauchy with respect to the quasi-norm on processes given by $\sup_{\lambda > 0 }\lambda^\alpha \bP(\sup_{s \in [0,t]} |X_s| > \lambda)$. The stochastic integral of $\phi \in \LL^\alpha(\LL^{\alpha,t})$ is then defined as the limit of the integrals of the approximating sequence, and the upper bound in Proposition~\ref{prop_isometryint} is subsequently shown to hold for all $\phi \in \LL^\alpha(\LL^{\alpha,t})$. 

As in \cite{RW1986}, the definition of the stochastic integral can be extended from $\LL^\alpha(\LL^{\alpha,t})$ to $\LL^{\alpha,t}_{\text{a.s.}}$ by truncating along a sequence of stopping times. Let $\phi \in \LL^{\alpha,t}_{\text{a.s.}}$. For each $k\in \N$ we define the stopping time
\begin{equation} \label{def_tau_k}
\tau_k(\phi) := \inf \{s> 0 : T_\phi(s) > k\}.
\end{equation}
Our convention is that $\inf \emptyset = \infty$. By definition of $\LL^{\alpha,t}_{\text{a.s.}}$, it follows that $\lim_{k \to \infty} \bP(\tau_k(\phi) \leq t) = 0$. We then define
\begin{equation*} 
\phi^{(k)}(s,x) := \phi(s,x) 1_{\{s \leq \tau_k(\phi)\}}.
\end{equation*}  
Then $T_{\phi^{(k)}}(t) = T_{\phi}(t \wedge \tau_k(\phi)) \leq k$, and hence $\phi^{(k)} \in \LL^{\alpha}(\LL^{\alpha,t})$ and we may define the stochastic integral $(\phi^{(k)} \cdot L)$ on $[0,t]$. For $s \leq \tau_k(\phi)$, we define $(\phi \cdot L)_s := (\phi^{(k)} \cdot L)_s$. Thus we can define the stochastic integral $(\phi \cdot L)$ on $[0,t]$ on the event $\{\tau_k(\phi) >t\}$, which allows us to give an almost sure construction because $\bP(\cup_{k=1}^\infty \{\tau_k(\phi) > t\}) = 1$.

We now prove the results from Section~\ref{s_stablecalc} whose proofs we postponed. In order, they are the lower bound in Proposition~\ref{prop_isometryint}, Lemma~\ref{lemma_bdg}, Proposition~\ref{prop_integral_rep}, and Lemma~\ref{lemma_dctbdg}. We recall the notation introduced in Section~\ref{s_dimone} that, for a stochastic process $(X_s)_{s \in [0,t]}$, $|X_t|^* :=\sup_{s \in [0,t]} |X_s|$. To prove the lower bound in Proposition~\ref{prop_isometryint}, we first prove that it holds for simple functions. 

\begin{lemma} \label{lemma_lwr_iso}
There is a constant $c_\alpha > 0$ such that for any $t>0$ and simple $\phi \in \LL^\alpha(\LL^{\alpha,t})$,
\begin{equation} \label{iso_lwr_bd}
\sup_{\lambda > 0} \lambda^\alpha \bP(|(\phi \cdot L)_t|^* > \lambda) \geq c_\alpha \E(T_\phi(t)).
\end{equation}
\end{lemma}
\begin{proof}
For this proof, we realize the stable martingale measure $L(ds,dx)$ as a compensated Poisson random measure as in \eqref{e_Poisson_integral}. In particular, the jumps of $L(ds,dx)$ are the points $(s,x,r)$ of a Poisson random measure $N(ds,dx,dr)$ on $\R_+ \times \R^d \times \R_+$ with compensator $\hat{N}(ds,dx,dr) =  ds dx \sigma_\alpha  r^{-1-\alpha}dr$, where $\sigma_\alpha = \alpha(\alpha-1)/\Gamma(2-\alpha)$.

Let $\phi \in \LL^\alpha(\LL^{\alpha,t})$ be simple. Without loss of generality we can write
\begin{equation*}
\phi(s,x) = \sum_{i = 1}^n \xi_{i} 1_{A_{i}}(x) 1_{(t_{i},t_{i+1}]}(s),
\end{equation*}
with $0 \leq t_1 < t_2 < \dots < t_{n+1} \leq t$, and $A_i \in \underline{\cB}$ and $\xi_i \in \cF_{t_i}$ for all $i = 1,\dots,n$. (With such a representation, some of the $\xi_i$ may be identically zero, but this is not an issue.) We further define $T_i = |\xi_i|^\alpha |A_i| (t_{i+1} - t_i)$ for $i = 1,\dots,n$, and remark that $\sum_{i=1}^n T_i = T_\phi(t)$.

We define the process of jumps $(\Delta (\phi \cdot L)_s)_{s \in (0,t]}$ given by $\Delta (\phi \cdot L)_s = (\phi\cdot L)_s - (\phi\cdot L)_{s-}$, and we define $\Delta^* = \max_{s \leq t} |\Delta (\phi \cdot L)_s |$ to be the largest jump which occurs by time $t$. (There a.s. exists a largest jump, so the supremum of the jump sizes is achieved and hence is a maximum.) Let $\lambda>0$. We observe that if $(\phi \cdot L)_s$ has a jump of absolute value greater than $2\lambda$ by time $t$, then $|(\phi \cdot L)_t|^* > \lambda$. In particular,
\begin{equation*}
\bP( |(\phi\cdot L)_t|^* > \lambda) \geq \bP (\Delta^* > 2\lambda).
\end{equation*}
We will show that 
\begin{equation} \label{e_simpleiso_suff}
\lim_{\lambda \to \infty} \lambda^\alpha \bP (\Delta^* > 2\lambda) = c_\alpha \E(T_\phi(t)),
\end{equation}
for some $c_\alpha>0$. By the previous inequality, this implies that $\liminf_{\lambda \to \infty} \lambda^\alpha \bP( |(\phi\cdot L)_t|^* > \lambda) \geq c_\alpha \E(T_\phi(t))$, which implies the desired result.

For $i=1,\dots,n$, let $\Delta_i^* = \max_{s \in (t_i,t_{i+1}]} |\Delta (\phi \cdot L)_s|$. We remark that, conditioned on $\cF_{t_i}$, $\Delta_i^* >2\lambda$ if and only if there is a point $(s,x,r)$ in $N$ such that $s \in (t_i,t_{i+1}]$, $x \in A_i$, and $r \geq 2\lambda / |\xi_i|$. If $\xi_i \neq 0$, the intensity of such points is 
\begin{equation*}
\int_{t_i}^{t_{i+1}} \int_{A_i} \int_{2\lambda / |\xi_i|}^\infty \hat{N}(ds,dx,dr) = \frac{\sigma_\alpha}{\alpha} |A_i|(t_{i+1} - t_i) \bigg( \frac{2\lambda}{|\xi_i|} \bigg)^{-\alpha} =  c_\alpha \lambda^{-\alpha} T_i,
\end{equation*}
where $c_\alpha = 2^{-\alpha} \sigma_\alpha/\alpha$, and the intensity is $0$ if $\xi_i = 0$, which also equals $c_\alpha \lambda^{-\alpha} T_i$ in this case, because $T_i = 0$. Hence, the conditional probability (given $\cF_{t_i}$) that there is at least one such point equals $1-e^{-c_\alpha\lambda^{-\alpha}T_i}$. In particular,
\begin{align} \label{e_simpleiso_pois}
\bP(\Delta_i^* > 2\lambda \, | \, \cF_{t_i}) = 1 - e^{-c_\alpha \lambda^{-\alpha} T_i}.
\end{align}
Similarly to as in the proof of \cite[Theorem~3.1]{RW1986}, we expand $\bP(\Delta^* > 2\lambda)$ as follows:
\begin{align*}
\bP (\Delta^* > 2\lambda)  & = \sum_{i=1}^n \bP( \Delta_1^* \leq 2\lambda,\dots, \Delta_{i-1}^* \leq 2\lambda, \Delta_i^* > 2\lambda) 
\\& = \sum_{i=1}^n \E( 1_{\{\Delta_1^* \leq 2\lambda,\dots, \Delta_{i-1}^* \leq 2\lambda \}} \bP( \Delta_i^* > 2\lambda \,|\,  \cF_{t_i})) 
\\ &= \sum_{i=1}^n \E( 1_{\{\Delta_1^* \leq 2\lambda,\dots, \Delta_{i-1}^* \leq 2\lambda \}} (1 - e^{-c_\alpha \lambda^{-\alpha} T_i})). 
\end{align*}
The second line uses the fact that $\Delta_j \in \cF_{t_i}$ for $j < i$, and the last line uses \eqref{e_simpleiso_pois}. Since $\Delta_j^*$ is a.s. finite for each $j$, $1_{\{\Delta_1^* \leq 2\lambda,\dots, \Delta_{i-1}^* \leq 2\lambda \}}$ increases to $1$ a.s. as $\lambda \to \infty$ for each $i = 1,\dots,n$. Furthermore, elementary reasoning shows that $\lambda^\alpha(1 - e^{-c_\alpha \lambda^{-\alpha} T_i})$ increases to $c_\alpha T_i$ as $\lambda \to \infty$. In particular, the random variable $\lambda^\alpha 1_{\{\Delta_1^* \leq 2\lambda,\dots, \Delta_{i-1}^* \leq 2\lambda \}} (1 - e^{-c_\alpha \lambda^{-\alpha} T_i})$ converges increasingly to $c_\alpha T_i$ a.s. as $\lambda \to \infty$. Hence, multiplying the previous equation by $\lambda^\alpha$ and applying monotone convergence, we deduce that
\begin{align*}
\lim_{\lambda \to \infty} \lambda^\alpha \bP (\Delta^* > 2\lambda) &= \sum_{i=1}^n \E(\lim_{\lambda \to \infty} \lambda^\alpha 1_{\{\Delta_1^* \leq 2\lambda,\dots, \Delta_{i-1}^* \leq 2\lambda \}} (1 - e^{-c_\alpha \lambda^{-\alpha} T_i})) 
\\&= \sum_{i=1}^n \E(c_\alpha T_i) \\ &= c_\alpha \E(T_\phi(t)).
\end{align*}
This establishes \eqref{e_simpleiso_suff}, and thus we have proved the result. \end{proof}

The result for simple integrands implies the following corollary, which gives the promised lower bound in Proposition~\ref{prop_isometryint}.

\begin{corollary} For any $t>0$, \eqref{iso_lwr_bd} holds for all $\phi \in \LL^\alpha(\LL^{\alpha,t})$.
\end{corollary}
\begin{proof}
Let $\phi \in \LL^\alpha(\LL^{\alpha,t})$. By density of the simple processes in $\LL^\alpha(\LL^{\alpha,t})$, there exists a sequence of simple processes $(\phi_n)_{n \in \N}$ such that $\lim_{n \to \infty} \phi_n = \phi$ in $\LL^\alpha(\LL^{\alpha,t})$. By the upper bound in \eqref{e_prop_isometry}, $(\phi_n \cdot L)$ converges to $(\phi \cdot L)$ with respect to the quasi-norm $\|X\| = \sup_{\lambda > 0} \lambda^\alpha \bP(|X_t|^* > \lambda)$. By Lemma~\ref{lemma_lwr_iso}, for all $n \in \N$ we have
\[c_\alpha \E(T_{\phi_n}(t)) \leq \sup_{\lambda >0} \lambda^\alpha \bP(|(\phi_n\cdot L)_t|^* > \lambda).\]
We take $n \to \infty$ on both sides. Since $\phi_n$ and $(\phi_n \cdot L)$ converge respectively to $\phi$ and $(\phi \cdot L)$ with respect to the norm/quasi-norm in the above, we may exchange limit and expectation on both sides, and we obtain \eqref{iso_lwr_bd} for $\phi$.
\end{proof}

Next, we prove Lemma~\ref{lemma_bdg} as an easy consequence of (the upper bound from) Proposition~\ref{prop_isometryint}.

\begin{proof}[Proof of Lemma~\ref{lemma_bdg}]
It suffices to prove the result for $\phi \in \LL^\alpha(\LL^{\alpha,t})$, since if $\phi \not \in \LL^\alpha(\LL^{\alpha,t})$ the right hand side of \eqref{e_lemma_bdg} is infinite, so the inequality is trivial. Let $\phi \in \LL^\alpha(\LL^{\alpha,t})$ and $q \in [1,\alpha)$. For simplicity, we denote $Z = |(\phi \cdot L)_t|^*$, so that the upper bound in Proposition~\ref{prop_isometryint} gives
\[ \sup_{\lambda >0 } \lambda^\alpha \bP(Z \geq \lambda) \leq C_\alpha \E(T_\phi(t)).\]
Using the layer cake representation for the $q$th moment of $Z$, we obtain
\begin{align*}
\E(Z^q) &= \int_0^\infty \bP (Z^q \geq \lambda) d\lambda 
\\ & \leq \E(T_\phi(t))^{q/\alpha} + \int_{\E(T_\phi(t))^{q/\alpha}}^\infty \bP (Z \geq \lambda^{1/q}) d\lambda 
\\ &\leq \E(T_\phi(t))^{q/\alpha} + C_\alpha \E(T_\phi(t))\int_{\E(T_\phi(t))^{q/\alpha}}^\infty \lambda^{-\alpha/q} d\lambda
\\ &= \E(T_\phi(t))^{q/\alpha} + C_\alpha \frac{q}{\alpha - q} \,\E(T_\phi(t))^{1 + q/\alpha(1-\alpha/q)}
\\ &= \left(1+C_\alpha\frac{q}{\alpha - q}\right) \E(T_\phi(t))^{q/\alpha}.
\end{align*}
This completes the proof.
\end{proof}

We continue with the proof of Proposition~\ref{prop_integral_rep}, the representation of stochastic integrals with respect to $L$ as time-changed stable processes. For a proof using approximation by simple functions in the non-spatial setting, see \cite[Theorem 3.1]{RW1986}. Our proof with It\^{o}'s lemma is more along the lines of the proof of \cite[Lemma~2.15]{FMW2010}.

\begin{proof}[Proof of Proposition~\ref{prop_integral_rep}]
We have defined the one-sided $\alpha$-stable process and noise in terms of their Laplace transforms, but for this proof it is more convenient to use the Fourier transform. For a spectrally positive $\alpha$-stable process $(W_t)_{t \geq 0}$, for $\lambda \in \R$ we have
\begin{equation} \label{e_stablechar}
\log \E(e^{i \lambda W_t}) = -t \Psi(\lambda),
\end{equation}
where 
\begin{equation} \label{e_stableexpfn}
\Psi(\lambda) = |\lambda|^\alpha \left(1 - i \text{sgn}(\lambda) \tan\left(\frac{\pi \alpha}{2}\right)\right).
\end{equation}
(We again refer to Chapter VIII of Bertoin \cite{Bertoin96}.) Let $\phi \in \LL^{\alpha,t}_{\text{a.s.}}$ with $\phi \geq 0$. For $s \in [0,t]$, $\lambda \in \R$, we define
\begin{equation*}
M_s^\lambda = \exp(i \lambda (\phi \cdot L)_s + \Psi(\lambda) T_\phi(s)).
\end{equation*}
Suppose that we have constructed the stable noise $L$ via a compensated Poisson random measure $\tilde{N}(ds,dx,dr)$ as in \eqref{e_Poisson_integral}. The compensator is $\hat{N}(ds,dx,dr) = ds dx \nu(dr)$, with $\nu(dr)$ as in \eqref{def_jumpmeasure}. Then by It\^{o}'s lemma (see e.g. \cite[Theorem~II.5.1]{IK}),
\begin{align} \label{e_tchange_ito}
M_s^\lambda - 1 &= \int_0^s \int_{\R^d} \int_0^\infty  M_{u-}^\lambda \left[ e^{i \lambda \phi(u,x) r} - 1 \right] \tilde{N}(du,dx,dr) \notag
\\ & \hspace{0.4 cm} + \int_0^s \int_{\R^d} \int_0^\infty M^\lambda_u \left[ e^{i\lambda \phi(u,x)r} - 1 - i \lambda \phi(u,x) r \right] du dx \nu(dr) \notag
\\ & \hspace{0.4 cm} + \Psi(\lambda) \int_0^s M_u^\lambda  \bigg( \int_{\R^d} \phi(u,x)^\alpha dx \bigg) du. 
\end{align}
Since $\nu$ is the jump measure of the $\alpha$-stable process, we have
\begin{equation*}
\int_0^\infty \left[ e^{i\lambda \phi(u,x)r} - 1 - i \lambda  \phi(u,x) r \right] \nu(dr) = -\Psi (\lambda \phi(u,x)) = -\phi(u,x)^\alpha \Psi(\lambda),
\end{equation*}
where the second inequality follows from \eqref{e_stableexpfn} and the fact that $\phi(u,x) \geq 0$. In particular, the second and third terms on the right-hand side of \eqref{e_tchange_ito} cancel each other out and we are left with a stochastic integral of a complex integrand with respect to a compensated Poisson point measure. Thus, $(M_s(\phi) : s \in [0,t])$ is a complex local martingale. The result now follows, essentially, by applying optional stopping at the first passage time of $T_\phi(s)$ at level $u$ for all $u>0$. However, to do so we first extend the process because of the possibility that $T_\phi(t) < u$.

Let $(\hat{W}_s)_{s \geq 0}$ be an independent spectrally positive $\alpha$-stable process which we use to extend $(\phi \cdot L)$ as follows: we define a process $X_s$ by
\[X_s = (\phi \cdot L)_{s \wedge t} + \hat{W}_{(s-t) \vee 0},\] 
so that $X_s = (\phi \cdot L)_s$ for $s \leq t$ and is continued by an independent stable process afterwards. For $\lambda \in \R$, for all $s>0$ we define
\[\hat{M}^\lambda_s = M_{s \wedge t}^\lambda \cdot \exp(i\lambda \hat{W}_{(s-t)\vee  0} + \Psi(\lambda)((s-t)\vee  0)),\]
which, by \eqref{e_stablechar} and the previous argument concerning $M^\lambda_s$, is a complex local martingale. Next, for $s > 0$ we define
\begin{equation*}
\hat{T}_\phi(s) =\begin{cases} T_\phi(s) &\text{ if } s \leq t,
\\T_\phi(t) + (s-t) & \text{ if } s > t. \end{cases}
\end{equation*}
Since $\phi \in \LL^{\alpha,t}_{\text{a.s.}}$, $T_\phi(t) < \infty$ a.s. and so the above is well-defined. Moreover, we remark that
\begin{equation} \label{eq:martmod}
\hat{M}^\lambda_s = \exp(i\lambda X_s + \Psi(\lambda) \hat{T}_\phi(s)).
\end{equation} 
Finally, for $u\geq 0$ we define the stopping times $\tau(u) = \inf\{ s > 0 : \hat{T}_\phi(s) > u\}$. Its definition implies that $\tau(u) \leq t+u$, and in particular $\tau(u)$ is bounded. We remark that $|\hat{M}_{s \wedge \tau(u)}^\lambda| \leq \exp(|\lambda|^\alpha u)$ for all $s>0$, and hence $M^\lambda_{s \wedge \tau(u)}$ is a bounded complex martingale. Since $\tau(u)$ is a bounded stopping time, we conclude from optional stopping that $\E(\hat{M}^\lambda_{\tau(u)}) = 1$ for all $\lambda \in \R$. In particular, by \eqref{eq:martmod}, for all $\lambda$ we have
\begin{align*}
\log \E(\exp(i\lambda X_{\tau(u)})) = -\Psi(\lambda)u.
\end{align*}
In the above, we have used the fact that $\hat{T}_\phi(\tau(u)) = u$. Hence $(X_{\tau(u)})_{u \geq 0}$ is a spectrally positive $\alpha$-stable process. We note that by definition, $X_{\tau(u)} = (\phi \cdot L)_{\tau(u)}$ for all $u$ such that $\tau(u) \leq t$. This is the same as $u \leq T_\phi(t)$, and it follows that $((\phi \cdot L)_{\tau(u)})_{u \in [0,T_\phi(t)]}$ is a stable process run until time $T_\phi(t)$. Since $\tau$, when restricted to $[0,T_\phi(t)]$, is just the right continuous inverse of $T_\phi(s)$, changing time again by $T_\phi$ proves the result. \end{proof}

Finally, we give the proof of the dominated convergence theorem for stochastic integrals, which is elementary. For a similar result with more general noises but stated in terms of a different norm, see \cite[Lemma A.1]{CDH2019}.

\begin{proof}[Proof of Lemma~\ref{lemma_dctbdg}]
Let $(\phi_n)_{n \in \N}$ and $\phi$ be as in the statement, and let $\tau_k(\phi)$ be as in \eqref{def_tau_k}. By Lemma~\ref{lemma_bdg},
\begin{align} \label{e_dctbdg1}
\E\bigg(\sup_{s \in [0,t \wedge \tau_k(\phi)]} |(\phi_n \cdot L)_s| \bigg) \leq C_1 \E (T_{\phi_n}(t \wedge \tau_k(\phi)))^{1/\alpha}.
\end{align}
Since $\phi_n \to 0$ point-wise and $|\phi_n|^\alpha \leq \phi^\alpha$, which is integrable on $[0,t \wedge \tau_k(\phi)] \times \R^d$ almost surely, $T_{\phi_{n}}(t \wedge \tau_k(\phi))$ vanishes almost surely as $n \to \infty$ for all $k \in \N$ by dominated convergence. Next, for each $n \in \N$, $T_{\phi_n}(t \wedge \tau_k(\phi)) \leq T_\phi(t \wedge \tau_k(\phi)) \leq k$. Thus $T_{\phi_n}(t \wedge \tau_k(\phi))$ has an integrable upper bound, uniformly in $n$, for each $k$. Since it converges to $0$ a.s., dominated convergence implies that $\E(T_{\phi_n}(t \wedge \tau_k(\phi))) \to 0$ as $n \to \infty$ for each $k \in \N$. In particular, from \eqref{e_dctbdg1} this implies that
\begin{equation} \label{e_dctbdg2}
\sup_{s \in [0,t \wedge \tau_k(\phi)]} |(\phi_n \cdot L)_s| \to 0 
\end{equation}
in probability for each $k \in \N$. 

Let $\delta, \eps > 0$.  We have $\lim_{k \to \infty} \bP(\tau_k(\phi) < t) = 0$, so in particular there exists $k = k(\delta) $ depending $\delta$ such that $\bP(\tau_k(\phi) < t) < \delta / 2$. By \eqref{e_dctbdg2}, there exists $N \in \N$ depending on $k(\delta)$ such that for $n \geq N$,
\[  \bP \bigg(\sup_{s \in [0,t \wedge \tau_k(\phi)]} |(\phi_n \cdot L)_s| > \eps \bigg) < \delta / 2.\]
Combining these estimates, we obtain that
\begin{align*}
\bP\bigg(\sup_{s \in [0,t]} |(\phi_n \cdot L)_s| > \eps \bigg) &\leq \bP(\tau_k(\phi) \leq t)  + \bP\bigg(\sup_{s \in [0,t \wedge \tau_k(\phi)]} |(\phi_n \cdot L)_s| > \eps\bigg)
\\&\leq \delta / 2 + \delta / 2 =\delta.
\end{align*}
This completes the proof.
\end{proof}

\providecommand{\bysame}{\leavevmode\hbox to3em{\hrulefill}\thinspace}
\providecommand{\MR}{\relax\ifhmode\unskip\space\fi MR }
\providecommand{\MRhref}[2]{%
  \href{http://www.ams.org/mathscinet-getitem?mr=#1}{#2}
}
\providecommand{\href}[2]{#2}

\vspace{0.5 cm}

\noindent \textbf{Acknowledgements.} The author thanks Ed Perkins, Raluca Balan, and Carsten Chong for useful discussions and comments, and two anonymous referees for their thorough reports and helpful suggestions. This work was partially completed while the author was supported by an NSERC Postdoctoral Fellowship which was held at McGill University.


\end{document}